\documentclass[11pt,reqno]{amsart}

\usepackage{amsfonts}
\usepackage{eurosym}
\usepackage{amssymb}
\usepackage{amsthm}
\usepackage{amsmath}
\usepackage{amsaddr}
\usepackage{bm}
\usepackage{cite}
\usepackage{mathrsfs}
\usepackage{xcolor}
\usepackage[OT1]{fontenc}
\usepackage[left=2.5cm, right=2.5cm, top=3cm, bottom=2.5cm]{geometry}
\usepackage{hyperref}
\usepackage{mathtools}
\hypersetup{colorlinks=true, linkcolor=blue, citecolor=red, urlcolor=blue}
\usepackage{times}
\usepackage{placeins}

\usepackage{enumitem}
\usepackage{xpatch}
\makeatletter
\AtBeginDocument{\xpatchcmd{\@thm}{\thm@headpunct{.}}{\thm@headpunct{}}{}{}}
\makeatother

\allowdisplaybreaks
\newtheorem{theorem}{Theorem}[section]

\newtheorem{definition}[theorem]{Definition}
\newtheorem{lemma}[theorem]{Lemma}
\newtheorem{proposition}[theorem]{Proposition}

\theoremstyle{remark}
\newtheorem{remark}[theorem]{Remark}

\numberwithin{equation}{section}

\usepackage{parskip}
\DeclareMathAlphabet\mathbfcal{OMS}{cmsy}{b}{n}

\newcommand{\abs}[1]{|#1|}

\newcommand{\Bigabs}[1]{\Big| #1 \Big|}

\newcommand{\norm}[1]{\| #1 \|}

\newcommand{\ang}[2]{ \langle #1 , #2  \rangle}
\newcommand{\bigang}[2]{ \big< #1 , #2  \big>}

\newcommand{\scp}[2]{ \left( #1 , #2  \right)}
\newcommand{\bigscp}[2]{\big( #1 , #2 \big)}

\newcommand{\meano}[1]{{\langle #1 \rangle}_{\Omega}}
\newcommand{\meang}[1]{{\langle #1 \rangle}_{\Gamma}}
\newcommand{\mean}[2]{\textnormal{mean}(#1,#2)}
\newcommand{\bigmean}[2]{\textnormal{mean}\big(#1,#2\big)}

\newcommand{\R}{\mathbb R}
\newcommand{\N}{\mathbb N}
\newcommand{\n}{\mathbf{n}}

\newcommand{\intO}{\int_\Omega}

\newcommand{\intG}{\int_\Gamma}
\newcommand{\intS}{\int_\Sigma}

\newcommand{\dtau}{\;\mathrm d\tau}

\newcommand{\dx}{\;\mathrm{d}x}

\newcommand{\dt}{\;\mathrm dt}

\newcommand{\dxt}{\;\mathrm{d}x\;\mathrm{d}t}
\newcommand{\dGt}{\;\mathrm{d}\Ga\;\mathrm{d}t}
\newcommand{\dxs}{\;\mathrm{d}x\;\mathrm{d}s}
\newcommand{\dGs}{\;\mathrm{d}\Ga\;\mathrm{d}s}
\newcommand{\dxtau}{\;\mathrm{d}x\;\mathrm{d}\tau}
\newcommand{\dGtau}{\;\mathrm{d}\Ga\;\mathrm{d}\tau}

\newcommand{\dG}{\;\mathrm d\Ga}

\newcommand{\ddt}{\frac{\mathrm d}{\mathrm dt}}

\newcommand{\del}{\partial}
\newcommand{\delt}{\partial_{t}}

\newcommand{\deln}{\partial_\n}
\newcommand{\delthp}{\del_{t,h}^+}
\newcommand{\delthm}{\del_{t,h}^-}

\newcommand{\Grad}{\nabla}
\newcommand{\Lap}{\Delta}
\newcommand{\Div}{\textnormal{div}}
\newcommand{\Gradg}{\nabla_\Ga}
\newcommand{\Lapg}{\Delta_\Ga}
\newcommand{\Divg}{\textnormal{div}_\Ga}

\newcommand{\emb}{\hookrightarrow}

\newcommand{\suchthat}{\;\ifnum\currentgrouptype=16 \middle\fi|\;}

\newcommand{\Om}{\Omega}
\newcommand{\Ga}{\Gamma}

\definecolor{rosso}{rgb}{0.8,0,0}
\definecolor{violet}{rgb}{0.65,0,0.65}
\definecolor{darkgreen}{rgb}{0,0.5,0}

\newcommand{\J}{\mathbf J}
\newcommand{\K}{\mathbf K}
\newcommand{\f}{\mathbf f}

\newcommand{\bv}{\mathbf v}
\newcommand{\bw}{\mathbf w}
\newcommand{\wv}{\widehat\bv}
\newcommand{\ww}{\widehat\bw}
\newcommand{\tv}{\widetilde{\bv}}
\newcommand{\tw}{\widetilde{\bw}}
\newcommand{\D}{\mathbf D}
\newcommand{\Dg}{\mathbf D_\Ga}

\newcommand{\uloc}{\mathrm{uloc}}
\newcommand{\tot}{\textup{tot}}
\newcommand{\free}{\mathrm{free}}

\def \no#1#2#3 {{\bf #1} (#3), #2.}
\def \eds#1#2#3 {#1, #2, #3.}
\makeatletter
\def\@settitle{\begin{center}%
  \baselineskip14\p@\relax
    \huge%<- NEW
  \@title
  \end{center}%
}
\makeatother

\begin{document}

\title[On  bulk-surface NSCH model: Existence of weak solutions and asymptotic limits]
{\emph{On a bulk-surface Navier--Stokes--Cahn--Hilliard model: \\ Existence of weak solutions and asymptotic limits}}
\author[Jonas Stange]{Jonas Stange}

\address{Universit\"{a}t Regensburg \\ 
Fakult\"{a}t f\"{u}r Mathematik \\
Universit\"atsstr. 31, D-93053 Regensburg, Germany\\
\href{mailto:jonas.stange@ur.de}{jonas.stange@ur.de}
}


% 35K35, %%% Nonlinear parabolic equations
% 35D30, %%% Weak solutions to PDEs
% 35A01, %%% Existence problems for PDEs: global existence, local existence, non-existence
% 35A02, %%% Uniqueness problems for PDEs: global uniqueness, local uniqueness, non-uniqueness
% 35Q32, %%% PDEs in connection with fluid mechanics
% 76D05, %%% Navier-Stokes equations for incompressible viscous fluids

%

%\date{\today}

\begin{abstract}
We study a thermodynamically consistent bulk-surface Navier--Stokes--Cahn--Hilliard system describing two-phase flows exhibiting moving contact lines, variable contact angles, and mass transfer. We establish the existence of global weak solutions for non-degenerate mobility functions and singular free-energy potentials. The proof is based on an implicit time-discretization scheme and a subdifferential characterization of the convex part of the bulk-surface free energy. Finally, we study the simultaneous high-friction and decoupling limit, and prove that corresponding weak solutions converge, up to a subsequence, to a weak solution of the Abels--Garcke--Grün model. One key ingredient of the proof is the Mosco convergence of the convex part of the bulk-surface free energy, which might be of independent interest.
% Abstract \\
% Abstract \\
% Abstract \\
% Abstract \\
% Abstract \\
% Abstract \\
% Abstract \\
% We analyze a recently derived bulk-surface Navier--Stokes--Cahn--Hilliard model for the modeling of fluid mixtures, which allows for moving contact lines, variable contact angles, and mass transfer.
% We prove the existence of global weak solutions using an implicit time-discretization scheme. For one parameter need to consider the asymptotic limit. To this end, we first prove Mosco convergence of the free energy functional when letting the parameter controlling the coupling go to $0$ and $\infty$, respectively. Lastly, by sending all coupling parameters to the decoupling regime, in particular high friction limit, we show the convergence of weak solutions to weak solutions of the Abels--Garcke--Grün model.
\end{abstract}

\maketitle

\section{Introduction}
\label{SECT:INTRO}
\noindent
The mathematical description of mixtures of two or more immiscible materials is a fundamental topic in fluid mechanics and materials science, with numerous applications in biology, chemistry, and engineering. Typical phenomena include phase separation, droplet breakup, thermocapillary flows, and the motion of contact lines. A central difficulty in the modeling of such systems is the description of the evolving interface separating the different components. Two principal approaches have been developed for this purpose: sharp interface methods and diffuse interface methods, see, for instance, \cite{Abels2018,Du2020,Giga2018,Pruess2016} and the references therein.

In sharp interface models, the different materials occupy distinct regions separated by an interface of zero thickness. The interface is represented by an evolving hypersurface whose motion is coupled to the governing equations in the surrounding bulk regions and suitable transmission conditions across the interface. Since the position of the interface is itself part of the unknown, this leads to a free boundary problem. While this approach provides a natural description when the interface remains sufficiently regular, the mathematical treatment becomes particularly challenging in the presence of topological changes such as the merging or breakup of droplets.

Diffuse interface models provide an alternative description in which the sharp interface is replaced by a thin transition layer of positive width. The local composition of a binary mixture is described by an order parameter, usually called the phase-field, which takes values close to the pure phases away from the interface and varies continuously across the interfacial region. In contrast to the sharp interface approach, no explicit tracking of the moving interface is required, and the evolution can be formulated on a fixed spatial domain. This makes diffuse interface models particularly convenient for describing topological changes and complex interfacial dynamics. Moreover, in many situations, their relation to the corresponding sharp interface models can be justified through suitable sharp interface limits.

A classical diffuse interface model for two incompressible viscous fluids is the so-called Model H \cite{Hohenberg1977, Gurtin1996, Starovoitov1994}, which couples the incompressible Navier--Stokes equations with a convective Cahn--Hilliard equation. Its classical formulation assumes that the two fluids have the same constant density. To describe mixtures with different densities, Abels, Garcke, and Grün \cite{Abels2012} introduced a thermodynamically consistent diffuse interface model for incompressible two-phase flows with unmatched densities, also called the AGG model. It reads as
\begin{subequations}\label{System:AGG}
    \begin{alignat}{2}
        &\delt(\rho(\phi)\bv) + \Div\big(\bv\otimes(\rho(\phi)\bv + \J)\big) - \Div(2\nu_\Om(\phi)\D\bv) + \Grad\, p = -\Div(\Grad\phi\otimes\Grad\phi) &&\qquad\text{in~}Q, \\
        &\Div\,\bv = 0 &&\qquad\text{in~}Q, \\
        &\delt\phi + \Div(\phi\bv) = \Div(m_\Om(\phi)\Grad\mu) &&\qquad\text{in~}Q, \\
        &\mu = -\Lap\phi + F^\prime(\phi) &&\qquad\text{in~}Q, \\
        &\bv = \mathbf 0, \qquad \deln\phi = m_\Om(\phi)\deln\mu = 0 &&\qquad\text{on~}\Sigma, \\
        &\bv\vert_{t=0} = \bv_0, \qquad \phi\vert_{t=0} = \phi_0 &&\qquad\text{in~}\Om.
    \end{alignat}
\end{subequations}
Here, $\Om\subset\R^d$, $d\in\{2,3\}$, is a bounded domain with boundary $\Ga\coloneqq\partial\Om$, and we set
\begin{align*}
    Q\coloneqq\Om\times(0,\infty), \qquad \Sigma\coloneqq\Ga\times(0,\infty).
\end{align*}
The velocity field of the mixture is denoted by $\bv$, while $\phi$ denotes the phase-field and $\mu$ the associated chemical potential. The density is given by the affine interpolation
\begin{align*}
    \rho(\phi) \coloneqq \frac{\tilde\rho_2-\tilde\rho_1}{2}\phi + \frac{\tilde\rho_2+\tilde\rho_1}{2},
\end{align*}
where $\tilde\rho_1,\tilde\rho_2>0$ denote the constant densities of the pure fluids. The additional mass flux
\begin{align*}
    \J \coloneqq -\frac{\tilde\rho_2-\tilde\rho_1}{2}m_\Om(\phi)\Grad\mu
\end{align*}
accounts for mass transport induced by diffusion. In particular, if the densities are matched, i.e.,
$\tilde\rho_1=\tilde\rho_2$, then $\rho$ is constant and $\J=
\equiv\mathbf 0$, and the system reduces to Model H.
In addition, $\nu_\Om$ denotes the viscosity of the mixture in the bulk, and in general depends on the phase-field. If the two pure fluids have constant viscosities $\nu_1$ and $\nu_2$, respectively, a common choice is the affine interpolation
\begin{align*}
    \nu_\Om(\phi) = \frac{\nu_2-\nu_1}{2}\phi + \frac{\nu_2+\nu_1}{2} \qquad\text{in~}Q.
\end{align*}
Moreover, $\D$ denotes the symmetric gradient operator, and is defined by
\begin{align*}
    \D\bv\coloneqq \frac12\Big((\Grad\bv) + (\Grad\bv)^\top\Big).
\end{align*}
The function $m_\Om$ is the so-called Onsager mobility. In general, it depends on the phase-field variable $\phi$, and describes both the spatial distribution and the magnitude of the underlying diffusion mechanisms. The function $F$ is a double-well potential describing the tendency of the mixture to separate into two phases. A physically relevant example is the logarithmic Flory--Huggins potential
\begin{align*}
    W_\mathrm{log}(s) = \frac{\Theta}{2}\big((1+s)\ln(1+s)+(1-s)\ln(1-s)\big) - \frac{\Theta_c}{2}s^2, \qquad s\in[-1,1],
\end{align*}
where $0<\Theta<\Theta_c$. Since
$W_\mathrm{log}^\prime(s)\to\pm\infty$ as $s\to\pm1$, this potential belongs to the class of singular free-energy densities.

The classical boundary conditions in \eqref{System:AGG}, although mathematically convenient, impose several restrictions on the description of phenomena occurring near the boundary. First, the homogeneous Neumann condition $\deln\phi=0$ enforces the diffuse interface to meet the boundary orthogonally and therefore corresponds to a fixed contact angle of ninety degrees. Second, the no-flux condition for the chemical potential prevents any transfer of material between the bulk and the boundary. Finally, the no-slip condition $\bv=\mathbf 0$ suppresses tangential fluid motion directly at the boundary and is therefore not well suited for accurately describing moving contact lines, \cite{Dussan1979,Qian2006,Seppecher1996}. These limitations have motivated the development of diffuse interface models with dynamic boundary conditions and generalized Navier-slip conditions.

In particular, dynamic boundary conditions for the Cahn--Hilliard equation allow the boundary to possess its own free energy and phase-separation dynamics, see, e.g., \cite{Goldstein2011,Liu2019,Knopf2021a,Knopf2020}. Depending on the coupling mechanism, they can describe variable contact angles as well as exchange of material between the bulk and the surface. In two-phase flow models, such conditions are naturally combined with generalized Navier boundary conditions, which relate tangential stresses to slip velocities and capillary forces at the boundary, see, e.g., \cite{Chan2025, Gal2016, Gal2019, Giorgini2023, Gal2024, Gal2023, Knopf2025a}. This provides a substantially richer description of contact-line dynamics than classical no-slip and homogeneous Neumann boundary conditions.

Motivated by these considerations, Knopf and the author \cite{Knopf2025a} recently introduced and analyzed a thermodynamically consistent bulk-surface Navier--Stokes--Cahn--Hilliard model. In contrast to previous models with dynamic boundary conditions, the model additionally incorporates a surface Navier--Stokes equation to account for momentum transport along the boundary. Thus, not only the phase separation but also the viscous fluid dynamics are described both in the bulk and on the boundary. The inclusion of the surface Navier--Stokes equation is motivated, in particular, by applications to biological membranes. According to the classical fluid-mosaic model of Singer and Nicolson \cite{Singer1972}, biological membranes can be viewed as laterally incompressible two-dimensional viscous fluid layers containing embedded proteins. This suggests treating the boundary as an active material layer which supports both compositional evolution and tangential momentum transport.

The bulk-surface Navier--Stokes--Cahn--Hilliard system considered in the present paper reads as \\[-1.22em]
\begin{subequations}\label{System:NSCH}
    \begin{alignat}{2}
        \label{System:NSCH:1}
        &\delt(\rho(\phi)\bv) + \Div\big(\bv\otimes(\rho(\phi)\bv + \J)\big) - \Div(2\nu_\Om(\phi)\D\bv) + \Grad\, p = -\Div(\Grad\phi\otimes\Grad\phi) &&\qquad\text{in~}Q,\\
        \label{System:NSCH:2}
        &\Div\;\bv = 0 &&\qquad\text{in~}Q,
        \\
        \label{System:NSCH:3}
        &\delt(\sigma(\psi)\bw) + \Divg\big(\bw\otimes(\sigma(\psi)\bw + \K)\big) - \Divg(2\nu_\Ga(\psi)\Dg\bw) + \Gradg\,q \nonumber 
        \\
        &\, = -\Divg(\Gradg\psi\otimes\Gradg\psi) - 2\nu_\Om(\phi)\big[ \D\bv\,\n \big]_\tau + \tfrac12(\J\cdot\n)\bw + \tfrac12(\J_\Ga\cdot\n)\bw + \alpha\deln\phi\Gradg\psi \nonumber \\
        &\quad - \gamma(\phi,\psi)\bw &&\qquad\text{on~}\Sigma, 
        \\
        \label{System:NSCH:4}
        &\Divg\;\bw = 0 &&\qquad\text{on~}\Sigma,
        \\
        \label{System:NSCH:5}
        &\bv\vert_\Ga = \bw,\qquad \bv\cdot\n = 0 &&\qquad\text{on~}\Sigma,
        \\[0.3em]
        \label{System:NSCH:6}
        &\delt\phi  + \Div(\phi\bv) = \Div(m_\Om(\phi)\Grad\mu) &&\qquad\text{in~}Q, 
        \\
        \label{System:NSCH:7}
        &\mu = -\Lap\phi + F^\prime(\phi) &&\qquad\text{in~}Q, 
        \\
        \label{System:NSCH:8}
        &\delt\psi + \Divg(\psi\bw) = \Divg(m_\Ga(\psi)\Gradg\theta) - \beta m_\Om(\phi)\deln\mu &&\qquad\text{on~}\Sigma, 
        \\
        \label{System:NSCH:9}
        &\theta = -\Lapg\psi + \alpha\deln\phi + G^\prime(\psi) &&\qquad\text{on~}\Sigma, 
        \\
        \label{System:NSCH:10}
        &K\deln\phi = \alpha\psi - \phi,
        \qquad Lm_\Om(\phi)\deln\mu = \beta\theta - \mu, \qquad K,L\in[0,\infty], &&\qquad\text{on~}\Sigma, 
        \\[0.3em]
        \label{System:NSCH:11}
        &\bv\vert_{t=0} = \bv_0, \qquad \phi\vert_{t=0} = \phi_0 &&\qquad\text{in~}\Om, 
        \\
        \label{System:NSCH:12}
        &\bw\vert_{t=0} = \bw_0, \quad\;\; \psi\vert_{t=0} = \psi_0 &&\qquad\text{on~}\Ga.
    \end{alignat}
\end{subequations}
The bulk density $\rho$ and the bulk flux $\J$ are defined as above, while their surface counterparts are
\begin{align*}
    \sigma(\psi) \coloneqq\frac{\tilde\sigma_2-\tilde\sigma_1}{2}\psi+\frac{\tilde\sigma_2+\tilde\sigma_1}{2}, \qquad\K \coloneqq-\frac{\tilde\sigma_2-\tilde\sigma_1}{2}m_\Ga(\psi)\Gradg\theta, \qquad \J_\Ga \coloneqq-\beta\frac{\tilde\sigma_2-\tilde\sigma_1}{2}m_\Om(\phi)\Grad\mu.
\end{align*}
Here, $\bw$ denotes the surface velocity field, $\psi$ the surface phase-field, and $\theta$ the corresponding surface chemical potential. The functions $\nu_\Ga$ and $m_\Ga$ represent the surface viscosity and surface mobility, respectively, while $\gamma$ is a friction coefficient. The bulk and surface velocities are coupled through the trace relation $\bv\vert_\Ga=\bw$. Together with the non-penetration condition $\bv\cdot\n=0$, this implies that $\bw$ is tangential to the boundary. The surface momentum equation can therefore be regarded as an extension of generalized Navier-slip conditions in which the boundary itself possesses inertia, viscosity, and tangential momentum transport.

The time evolution of $\phi$ and $\mu$ is governed by the convective bulk Cahn--Hilliard subsystem \eqref{System:NSCH:6}-\eqref{System:NSCH:7}, while the dynamics of $\psi$ and $\theta$ are described by the convective surface Cahn--Hilliard subsystem \eqref{System:NSCH:8}-\eqref{System:NSCH:9}. These two subsystems are coupled through the terms involving the normal derivatives $\deln\phi$ and $\deln\mu$. Moreover, the phase-fields $\phi$ and $\psi$ are directly linked by the boundary condition \eqref{System:NSCH:10}$_1$, which has to be understood in the sense that
\begin{align*}
    \begin{cases}
        \phi = \alpha\psi &\text{if~} K = 0, \\
        K\deln\phi = \alpha\psi - \phi &\text{if~}K\in(0,\infty), \\
        \deln\phi = 0 &\text{if~} K = \infty
    \end{cases}
    \qquad\text{on~}\Sigma.
\end{align*}
The boundary condition \eqref{System:NSCH:10}$_2$, coupling the chemical potentials $\mu$ and $\theta$, has to be understood in a similar sense. In these relations, the parameters $K,L\in[0,\infty]$ distinguish different coupling regimes, whereas $\alpha,\beta\in\R$ represent material-dependent physical effects. We refer to \cite{Giorgini2023, Knopf2024} for a more comprehensive overview.

For prescribed velocity fields $\bv$ and $\bw$, the convective bulk-surface Cahn--Hilliard subsystem \eqref{System:NSCH:6}-\eqref{System:NSCH:9} subject to the coupling conditions \eqref{System:NSCH:10} with $K,L\in[0,\infty]$ was first analyzed in \cite{Knopf2024} in the case of regular potentials, and then in \cite{Knopf2025a,Giorgini2026,Stange2026,Knopf2026} for singular potentials. The non-convective version, i.e., if $\bv\equiv\mathbf 0$ and $\bw\equiv\mathbf 0$, has been the subject of extensive investigation in the literature. We refer to \cite{Lv2024,Lv2024a,Lv2024b,Stange2026,Fukao2021,Colli2020,Goldstein2011,Liu2019,Knopf2020,Knopf2021a}.

The system \eqref{System:NSCH} is associated with the total energy
\begin{align*}
	E^K_\tot(\bv,\bw,\phi,\psi) \coloneqq E_{\mathrm{kin}}(\bv,\bw,\phi,\psi) + E_\free^K(\phi,\psi),
\end{align*}
where the kinetic and free energy are defined as
\begin{align*}
    E_{\mathrm{kin}}(\bv,\bw,\phi,\psi) \coloneqq \intO \frac12\rho(\phi)\abs{\bv}^2\dx + \intG \frac12\sigma(\psi)\abs{\bw}^2\dG
\end{align*}
and
\begin{align*}
    E_\free^K(\phi,\psi) \coloneqq \intO \frac12\abs{\Grad\phi}^2 + F(\phi)\dx + \intG \frac12\abs{\Gradg\psi}^2 + G(\psi)\dG + \chi(K)\intG \frac12(\alpha\psi - \phi)^2\dG,
\end{align*}
respectively.
Moreover, the function $\chi$ is defined as
\begin{align}\label{DEF:CHI}
    \chi:[0,\infty]\rightarrow[0,\infty), \qquad \chi(r) \coloneqq \begin{cases} 0 & \text{if~}r\in\{0,\infty\}, \\
    \frac1r &\text{if~}r\in(0,\infty),
    \end{cases}
\end{align}
and is used to distinguish the different cases corresponding to the choice of $K\in[0,\infty]$. We usually suppress the explicit dependence of the free energy $E_\free^K$ and the total energy $E_\tot^K$ on $K\in[0,\infty]$, writing simply $E_\free$ and $E_\tot$, respectively. Sufficiently regular solutions to \eqref{System:NSCH} satisfy the mass conservation laws
\begin{align}\label{Intro:MCL}
    \begin{dcases}
        \beta\intO \phi(t)\dx + \intG \psi(t)\dG = \beta\intO \phi_0 \dx + \intG \psi_0\dG &\textnormal{if~} L\in[0,\infty), \\
        \intO\phi(t)\dx = \intO\phi_0\dx \quad\textnormal{and}\quad \intG\psi(t)\dG = \intG\psi_0\dG &\textnormal{if~} L = \infty
    \end{dcases}
\end{align}
for all $t\in[0,\infty)$ as well as the energy dissipation law
\begin{align}\label{Intro:ED}
	\ddt E_\tot(\bv(t),\bw(t),\phi(t),\psi(t)) &= - \intO 2\nu_\Om(\phi)\abs{\D\bv}^2\dx - \intG 2\nu_\Ga(\psi)\abs{\Dg\bw}^2\dG \nonumber \\
	&\quad - \intG \gamma(\phi,\psi)\abs{\bw}^2\dG  - \intO m_\Om(\phi)\abs{\Grad\mu}^2\dx \\
	&\quad - \intG m_\Ga(\psi)\abs{\Gradg\theta}^2\dG - \chi(L)\intG (\beta\theta - \mu)^2\dG \nonumber
\end{align}
for all $t\in[0,\infty)$.
Thus, the model simultaneously incorporates viscous dissipation in the bulk and on the surface, friction between the bulk and boundary motion, bulk and surface diffusion, and, when $L\in(0,\infty)$, dissipation associated with mass transfer between the two regions.

Before presenting the main results, we show that the bulk-surface Navier--Stokes equations can be rewritten in a way that is more useful for our analysis later on. To this end, we start by noting the standard decompositions
\begin{alignat*}{2}
    -\Div(\Grad\phi\otimes\Grad\phi) &= -\Grad\bigg(\frac12\abs{\Grad\phi}^2 + F(\phi)\bigg) + \mu\Grad\phi &&\qquad\text{in~}Q, \\
    -\Divg(\Gradg\psi\otimes\Gradg\psi) &= -\Gradg\bigg(\frac12\abs{\Gradg\psi}^2 + G(\psi)\bigg) + \theta\Gradg\psi - \alpha\deln\phi\Gradg\psi &&\qquad\text{on~}\Sigma,
\end{alignat*}
which follow from \eqref{System:NSCH:7} and \eqref{System:NSCH:9}, respectively. Then, defining
\begin{align*}
    \overline{p} \coloneqq p + \frac12\abs{\Grad\phi}^2 + F(\phi), \qquad \overline{q} \coloneqq q + \frac12\abs{\Gradg\psi}^2 + G(\psi),
\end{align*}
we can rewrite \eqref{System:NSCH:1}-\eqref{System:NSCH:2} as
\begin{subequations}\label{System:NSCH:Korteweg}
    \begin{alignat}{2}
        &\delt(\rho(\phi)\bv) + \Div\,(\bv\otimes(\rho(\phi)\bv + \J)) + \Div\,(2\nu_\Om(\phi)\D\bv) + \Grad\overline{p} = \mu\Grad\phi &&\;\;\quad\text{in~}Q, \\[.15em]
        &\delt(\sigma(\psi)\bw) + \Divg\,(\bw\otimes(\sigma(\psi)\bw + \K)) + \Divg\,(2\nu_\Ga(\psi)\Dg\bw) + \Grad\overline{q} \nonumber \\
        &\; = \theta\Gradg\psi - 2\nu_\Om(\phi)[\D\bv\,\n]_\tau + \tfrac12(\J\cdot\n)\bw + \tfrac12(\J_\Ga\cdot\n)\bw - \gamma(\phi,\psi)\bw &&\;\;\quad\text{on~}\Sigma.
    \end{alignat}
\end{subequations}

\medskip
\noindent
\textbf{Goals and novelties of this paper.} The existence of weak solutions for constant mobility functions was first established in \cite{Knopf2025a} by means of a semi-Galerkin scheme combined with Schauder's fixed point theorem. More precisely, the velocity fields were discretized using eigenfunctions of a newly introduced bulk-surface Stokes operator. This approach relies on a continuous dependence estimate for weak solutions of the associated convective Cahn--Hilliard system. For non-degenerate mobility functions, such estimates are currently available only in two spatial dimensions, see \cite{Conti2025} for the classical Cahn--Hilliard equation with homogeneous Neumann boundary conditions, and \cite{Stange2025,Stange2026} for the (convective) Cahn--Hilliard equation with dynamic boundary conditions. Consequently, a direct extension of the semi-Galerkin approach to non-degenerate mobilities in three spatial dimensional does not appear to be available. 
We also mention the work \cite{Stange2025}, which established the existence and uniqueness of a global-in-strong solution in two spatial dimensions with non-degenerate mobilities, exploiting the aforementioned results for the (convective) bulk-surface Cahn--Hilliard equation in \cite{Stange2025,Stange2026}.

In the present contribution, we pursue a different approach based on an implicit time-discretization scheme. This method was previously employed in \cite{Abels2013} to establish the existence of weak solutions for the Abels--Garcke--Grün model \eqref{System:AGG}, and later also in \cite{Abels2026} for a diffuse interface model for two-phase flows with phase transition. A key ingredient in that analysis is the characterization of the subdifferential of the convex part of the Ginzburg--Landau energy  $\widetilde E_\Om: L^2(\Om)\rightarrow(-\infty,\infty]$, defined by
\begin{align*}
	\widetilde E_\Om(u) = \begin{cases} 
	\intO \frac12\abs{\Grad u}^2 + F_0(u)\dx &\text{if~} u\in\mathrm{dom}\,\widetilde E_\Om, \\
	+\infty &\text{else},
	\end{cases}
\end{align*}
with $F_0$ as in \ref{Assumption:NSCH:Potentials}, and
\begin{align*}
	\mathrm{dom}\,\widetilde E_\Om = \{u\in H^1(\Om): \abs{u}\leq 1 \text{~a.e.~in~}\Om\},
\end{align*}
see \cite{Abels2007a}. We also refer to\cite{Abels2007}, where the ambient space was chosen as $L^2_{(m)}(\Om)$ for some $m\in(-1,1)$.

Implicit time-discretization schemes of this type have been successfully adapted to two-phase flow models with dynamic boundary conditions, see, for example, \cite{Gal2019, Gal2024}. However, since a corresponding characterization of the subdifferential of the convex part of the bulk-surface free energy $\widetilde E_\free^K$ was not available, the works relied on several approximation and regularization procedures.

This gap was closed in \cite{Giorgini2026}, where the subdifferential of the convex part of the bulk-surface free energy was characterized. Using this result, we can adapt the approach of \cite{Abels2013} directly to the bulk-surface setting without introducing further regularizations of the singular potentials. The main additional difficulties arise from the intricate coupling induced by the dynamic boundary conditions. To the best of our knowledge, the present contribution is the first to use this subdifferential formulation to establish the existence of weak solutions for a two-phase flow model with dynamic boundary conditions.

We expect that this approach, and in particular the underlying subdifferential characterization, can also be applied to a broader class of diffuse-interface models with dynamic boundary conditions, such as those studied in \cite{Giorgini2023, Gal2016, Gal2019, Gal2023, Gal2024, Chan2025}.

The implicit time-discretization scheme applies directly when $K\in(0,\infty]$. The case $K = 0$, however, requires a separate argument due to the trace constraint $\phi = \alpha\psi$ on~$\Sigma$. Indeed, the time-discretization scheme requires sufficiently regular initial data. Although Lemma~\ref{NSCH:Lemma:ApproxIC} provides a sequence $\{(\phi_0^N,\psi_0^N)\}_{N\in\N}\subset\mathcal{H}^2$ converging strongly to $(\phi_0,\psi_0)$ in $\mathcal{H}^1$ as $N\rightarrow\infty$, while preserving the mean-value constraints and pointwise bounds in \eqref{NSCH:cond:init}, this approximation does not, in general, preserve the trace constraint
$\phi_0^N = \alpha\psi_0^N$ on~$\Ga$. Alternative regularizations that preserve the trace constraint do not yield the requires pointwise bound for the surface phase-field when $\abs{\alpha} < 1$, see Section~\ref{Section:Existence} for more details.

For this reason, the existence proof proceeds in two steps. First, we establish the existence of weak solutions for $K\in(0,\infty]$ and $L\in[0,\infty]$ by means of an implicit time-discretization scheme. We then treat the case $K = 0$ by passing to the asymptotic limit $K\rightarrow 0$. This is similar in spirit to the strategy employed in \cite{Knopf2025a}, where the limit $L\rightarrow 0$ was used to establish existence of weak solutions for constant mobilities in the case $L = 0$.

Lastly, we study the simultaneous high-friction and decoupling limit
\begin{align*}
    \gamma\rightarrow\infty, \qquad K\rightarrow\infty, \qquad L\rightarrow\infty.
\end{align*}
We prove that every corresponding sequence of weak solutions to \eqref{System:NSCH} admits a subsequence converging, in a suitable sense, to a weak solution of the Abels--Garcke--Grün model \eqref{System:AGG}. This provides a rigorous connection between the bulk-surface two-phase flow model with dynamic boundary condition derived in \cite{Knopf2025a} and the classical Abels--Garcke--Grün model \cite{Abels2012}.

Both asymptotic analyses rely on the variational formulation of the chemical potentials
\begin{align*}
    \big(\mu - c_F\phi,\theta - c_G\psi\big) = \partial\widetilde E_\free^K(\phi,\psi).
\end{align*}
To pass to the limit directly in this equation, we establish the Mosco convergence of the family of convex bulk-surface energies $\{\widetilde E_\free^K\}_{K>0}$ both as $K\rightarrow 0$, as required for the treatment of the trace-coupled case, and as $K\rightarrow\infty$, as required for the high-friction limit. The resulting convergence of the associated subdifferentials allows us to identify the corresponding limiting chemical potential relations. 

\textbf{Structure of this paper.}
In Section~\ref{Section:Prelim}, we present all the necessary notation and preliminaries. Then, in Section~\ref{Section:MainResults}, we introduce the notion of a weak solution to \eqref{System:NSCH} and state our main results regarding the existence of a global-in-time weak solution and the high friction limit. Afterwards, in Section~\ref{Section:MoscoConvergence}, we study the Mosco convergence of the convex part of the free energy and the graph convergence of its subdifferential. Section~\ref{Section:Existence} is then dedicated to the proof of existence of weak solutions. Lastly, in Section~\ref{Seciton:HighFrictionLimit}, we investigate the high friction limit $\gamma\rightarrow\infty$ and the convergence to the Abels--Garcke--Grün model.

\noindent
\section{Functional framework, assumptions and preliminaries}
\label{Section:Prelim}

\subsection{General notation and function spaces}

We write $\N$ to denote the set of natural numbers (excluding zero). For any real Banach space $X$ with norm $\norm{\cdot}_X$, its dual space is denoted by $X^\prime$. The corresponding duality pairing of elements $\phi\in X^\prime$ and $\zeta\in X$ is denoted by $\ang{\phi}{\zeta}_X$. If $X$ is a Hilbert space, we write $\scp{\cdot}{\cdot}_X$ to denote its inner product. Given $p\in[1,\infty]$ and an interval $I\subset[0,\infty)$, the Bochner space $L^p(I,X)$ consists of all Bochner measurable, $p$-integrable functions defined on $I$ with values in $X$. We further define $W^{1,p}(I;X)$ as the space of all functions $f\in L^p(I;X)$ with vector-valued distributional time derivative $\delt f\in L^p(I;X)$. In particular, we set $H^1(I;X) = W^{1,2}(I;X)$. The set of continuous functions $f:I\rightarrow X$ is denoted by $C(I;X)$. Moreover, $C_w(I;X)$ denotes the space of functions $f:I\to X$ which are continuous on $I$ with respect to the weak topology on $X$, i.e., the map $I\ni t\mapsto \ang{\phi}{f(t)}_X$ is continuous for all $\phi\in X^\prime$. 
At this point, we also recall the following lemma, which can be found in \cite[Lemma~4.1]{Abels2009}.
\begin{lemma}\label{Prelim:Lemma:BC_w}
    Let $I\subset\R$ be a compact interval, and let $X,Y$ be Banach spaces with $Y\emb X$ and $X^\prime\emb Y^\prime$ densely. Then the continuous embedding $L^\infty(I;Y)\cap C(I;X)\emb C_w(I;Y)$ holds.
\end{lemma}

In the remainder of this section, let $\Om\subset\R^d$, $d=2,3$, be a bounded domain with sufficiently smooth boundary $\Ga\coloneqq\partial\Om$, and let $\n$ denote the exterior unit normal vector field on $\Gamma$. For any $s\geq 0$ and $p\in[1,\infty]$, the Lebesgue and Sobolev–Slobodeckij spaces for functions mapping from $\Om$ to $\R$ are denoted as $L^p(\Om)$ and $W^{s,p}(\Om)$, respectively. We write $\norm{\cdot}_{L^p(\Om)}$ and $\norm{\cdot}_{W^{s,p}(\Om)}$ to denote the standard norms on these spaces. If $p = 2$, we use the notation $H^s(\Om) = W^{s,2}(\Om)$, with the convention that $H^0(\Om)$ is identified with $L^2(\Om)$. For the Lebesgue and Sobolev–Slobodeckij spaces on the boundary $\Ga$, we use an analogous notation. The corresponding spaces of vector-valued mapping into $\R^d$ or matrix-valued functions mapping into $\R^{d\times d}$ are denoted by boldface letters, namely $\mathbf L^p$, $\mathbf W^{s,p}$ and $\mathbf H^s$. In particular, we write $\mathbf L^p_\tau(\Gamma)$, $\mathbf W^{s,p}_\tau(\Gamma)$ and $\mathbf H^s_\tau(\Gamma)$ to denote the corresponding spaces of tangential vector fields on $\Gamma$.

Moreover, for $p\in[1,\infty]$, we further introduce the first-order homogeneous Sobolev spaces
\begin{align*}
    \dot W^{1,p}(\Omega) 
    &\coloneqq
    \big\{ u \in L^p_\mathrm{loc}(\Omega) 
    \,:\, \Grad u \in \mathbf L^p(\Omega) \big\}\big/_{\sim} \,,
    \\
    \dot W^{1,p}(\Gamma) 
    &\coloneqq
    \big\{ v \in L^p(\Gamma) 
    \,:\, \Gradg v \in \mathbf L^p(\Gamma) \big\}\big/_{\sim} \,.
\end{align*}
Here, $L^p_\mathrm{loc}(\Omega)$ is the collection of all functions, which belong to $L^p(K)$ for every compact subset $K\subset \Omega$. On the boundary, we simply consider $L^p(\Gamma)$, since $\Gamma$ is itself a compact submanifold.
Moreover, the notation $/_{\sim}\,$ indicates that functions differing only by an additive constant are considered equivalent and thus identified with each other. Consequently, the functionals ${\norm{\Grad\,\cdot}_{\mathbf L^p(\Omega)}}$ and ${\norm{\Gradg\,\cdot}_{\mathbf L^p(\Gamma)}}$ define norms on the spaces $\dot W^{1,p}(\Omega)$ and $\dot W^{1,p}(\Gamma)$, respectively.

Furthermore, we write $\mathbf{P}^\Gamma$ to denote the orthogonal projection of $\R^d$ into the tangent space of the submanifold $\Gamma$.
This means that $\mathbf{P}^\Gamma$ is a linear map, which can be represented as
\begin{align*}
    \mathbf{P}^\Gamma = \mathbf{I} - \n\otimes\n,
\end{align*}
where $\mathbf{I}\in\R^{d\times d}$ is the unity matrix.
For any vector field $\bw$ on $\Gamma$, we also use the notation
\begin{align*}
    [\bw]_\tau = \mathbf{P}^\Gamma(\bw).
\end{align*}

\subsection{Differential operators in the bulk and on the surface} \label{NOT:DIFFOP}
Let $f:\Omega\to\R$ and $g:\Gamma\to\R$ be scalar functions, let $\bv:\Omega\to \R^d$ and $\bw:\Gamma\to\R^d$ be vector fields, and let $\mathbf{A}:\Omega\to\R^{d\times d}$ and $\mathbf{B}:\Gamma\to\R^{d\times d}$ be matrix-valued functions. 
For now, we assume that the boundary $\Gamma$ and all these functions are sufficiently regular.
As usual, $\Grad f$, $\Div\,\bv$, and $\Lap f = \Div(\Grad f)$ denote the gradient of $f$, the divergence of $\bv$ and the Laplacian of $f$. 
Analogously, $\Gradg g$, $\Divg\,\bw$, and $\Lapg g = \Divg(\Gradg g)$ are the tangential gradient of $g$, the surface divergence of $\bw$ and the Laplace--Beltrami operator applied to $g$. 
We point out that $\Grad f$ and $\Gradg g$ are to be understood as column vectors in $\R^d$.
For bulk vector fields, the gradient is defined as
\begin{align*}
    \Grad \bv = 
    \begin{pmatrix}
        (\Grad \bv_1)^\top \\ \vdots \\ (\Grad \bv_d)^\top
    \end{pmatrix}
    \in \R^{d\times d},
\end{align*}
where $\bv_1,\ldots,\bv_d$ denote the components of $\bv$. In other words, $\Grad\bv$ is the Jacobian of $\bv$. On the surface $\Gamma$, the projective (covariant) gradient of $\bw$ is defined as 
\begin{align*}
    \Gradg\bw = \mathbf{P}^\Gamma \Grad\tw \mathbf{P}^\Gamma,
\end{align*}
where $\tw$ is an extension of $\bw$ to an open neighborhood of $\Gamma$ in $\R^d$. However, the expression $\Gradg\bw$ does not depend on the concrete choice of this extension. If $\bw$ is a tangential vector field, its gradient can also be expressed as
\begin{align*}
    \Gradg \bw = \mathbf{P}^\Gamma
    \begin{pmatrix}
        (\Gradg \bw_1)^\top \\ \vdots \\ (\Gradg \bw_d)^\top
    \end{pmatrix}
    \in \R^{d\times d},
\end{align*}
where $\bw_1,\ldots,\bw_d$ denote the components of $\bw$.  
For the matrix-valued functions $\mathbf{A}$ and $\mathbf{B}$, we further define the divergences
\begin{align*}
    \Div\,\mathbf{A} = 
    \begin{pmatrix}
        \Div\, \mathbf{A}_{1\ast} \\ \vdots \\ \Div\, \mathbf{A}_{d\ast}
    \end{pmatrix},
    \qquad
    \Divg\,\mathbf{B} = \mathbf{P}^\Gamma
    \begin{pmatrix}
        \Divg\, \mathbf{B}_{1\ast} 
        \\ 
        \vdots 
        \\ \Divg\, \mathbf{B}_{d\ast}  
    \end{pmatrix}.
\end{align*}
Here, $\mathbf{A}_{1\ast},\ldots,\mathbf{A}_{d\ast}$ and $\mathbf{B}_{1\ast},\ldots,\mathbf{B}_{d\ast}$ denote the rows of $\mathbf{A}$ and $\mathbf{B}$, respectively.

\subsection{Bulk-surface product spaces}
For any $s\geq 0$ and $p\in[1,\infty]$, we define
\begin{align*}
	\mathcal{L}^p \coloneqq L^p(\Om)\times L^p(\Ga), \qquad \mathcal{W}^{s,p} \coloneqq W^{s,p}(\Om)\times W^{s,p}(\Ga).
\end{align*}
We write $\mathcal{H}^s = \mathcal{W}^{s,2}$ and identify $\mathcal{L}^2$ with $\mathcal{H}^0$. We point out that $\mathcal{H}^s$ is a Hilbert space with respect to the inner product
\begin{align*}
	\big((\zeta,\zeta_\Ga),(\xi,\xi_\Ga)\big)_{\mathcal{H}^s} \coloneqq (\zeta,\xi)_{H^s(\Om)} + (\zeta_\Ga,\xi_\Ga)_{H^s(\Ga)} \qquad\text{for all~}(\zeta,\zeta_\Ga),(\xi,\xi_\Ga)\in\mathcal{H}^s,
\end{align*}
and its induced norm $\norm{\cdot}_{\mathcal{H}^s} \coloneqq \scp{\cdot}{\cdot}_{\mathcal{H}^s}^{\frac12}$. We further recall that the duality pairing on $\mathcal{H}^1$ satisfies
\begin{align*}
	\big\langle(\zeta,\zeta_\Ga),(\xi,\xi_\Ga)\big\rangle_{\mathcal{H}^1} = (\zeta,\xi)_{L^2(\Om)} + (\zeta_\Ga,\xi_\Ga)_{L^2(\Ga)}
\end{align*}
for all $(\xi,\xi_\Ga)\in\mathcal{H}^1$ if $(\zeta,\zeta_\Ga)\in\mathcal{L}^2$. 

For every $\zeta\in H^1(\Om)^\prime$, we denote by $\meano{\zeta} = \abs{\Om}^{-1}\ang{\zeta}{1}_{H^1(\Om)}$ its generalized mean value over $\Om$. If $\phi\in L^1(\Om)$, its spatial mean can simply be expressed as $\meano{\zeta} = \abs{\Om}^{-1}\intO\phi\dx$. The spatial mean of any $\zeta_\Ga\in H^1(\Ga)^\prime$, denoted by $\meang{\zeta_\Ga}$, is defined similarly. We further define for any $p\in[2,\infty]$ the space
\begin{align*}
    \mathcal{L}^p_{(0)} \coloneqq L^p_{(0)}(\Om)\times L^p_{(0)}(\Ga),
\end{align*}
where
\begin{align*}
	L^p_{(0)}(\Om) 
    \coloneqq \big\{\zeta\in L^p(\Om): \meano{\zeta} = 0\big\}, 
    \qquad 
    L^p_{(0)}(\Ga) 
    \coloneqq \big\{\zeta_\Ga\in L^p(\Ga): \meang{\zeta_\Ga} = 0\big\}.
\end{align*}

Now, let $L\in[0,\infty]$ and $\beta\in\R$. We introduce the linear subspace
\begin{align*}
	\mathcal{H}^1_{L,\beta} \coloneqq \begin{cases}
	\mathcal{H}^1 &\text{if~} L\in(0,\infty], \\
	\big\{(\zeta,\zeta_\Ga)\in\mathcal{H}^1: \zeta = \beta\zeta_\Ga\text{~a.e.~on~}\Ga\big\} &\text{if~} L = 0.
	\end{cases}
\end{align*}
Subject to the inner product $\scp{\cdot}{\cdot}_{\mathcal{H}^1_{L,\beta}} \coloneqq \scp{\cdot}{\cdot}_{\mathcal{H}^1}$ and its induced norm, $\mathcal{H}^1_{L,\beta}$ is a Hilbert space. Furthermore, we define the product
\begin{align*}
	\big\langle(\zeta,\zeta_\Ga),(\xi,\xi_\Ga)\big\rangle_{\mathcal{H}^1_{L,\beta}} = (\zeta,\xi)_{L^2(\Om)} + (\zeta_\Ga,\xi_\Ga)_{L^2(\Ga)}
\end{align*}
for all $(\zeta,\zeta_\Ga),(\xi,\xi_\Ga)\in\mathcal{L}^2$. Using the Riesz representation theorem, this product can be extended to a duality pairing on $(\mathcal{H}^1_{L,\beta})^\prime\times\mathcal{H}^1_{L,\beta}$, which will also be denoted as $\ang{\cdot}{\cdot}_{\mathcal{H}^1_{L,\beta}}$.

Next, for $(\zeta,\zeta_\Ga)\in(\mathcal{H}^1_{L,\beta})^\prime$, we define the generalized bulk-surface mean
\begin{align*}
	\mean{\zeta}{\zeta_\Ga} \coloneqq \frac{\bigang{(\zeta,\zeta_\Ga)}{\scp{\beta}{1}}_{\mathcal{H}^1_{L,\beta}}}{\beta^2\abs{\Om} + \abs{\Ga}},
\end{align*}
which reduces to
\begin{align*}
	\mean{\zeta}{\zeta_\Ga} = \frac{\beta\abs{\Om}\meano{\zeta} + \abs{\Ga}\meang{\zeta_\Ga}}{\beta^2\abs{\Om} + \abs{\Ga}}
\end{align*}
if $(\zeta,\zeta_\Ga)\in\mathcal{L}^2$. We then define the closed linear subspace
\begin{align*}
	\mathcal{V}^1_{L,\beta} = \begin{cases}
	\{(\zeta,\zeta_\Ga)\in\mathcal{H}^1_{L,\beta}:  \mean{\zeta}{\zeta_\Ga} = 0\} &\text{if~}L\in[0,\infty), \\
	\{(\zeta,\zeta_\Ga)\in\mathcal{H}^1: \meano{\zeta} = \meang{\zeta_\Ga} = 0\} &\text{if~} L = \infty.
	\end{cases}
\end{align*}
Note that $\mathcal{V}^1_{L,\beta}$ is a Hilbert spaces with respect to the inner product $\scp{\cdot}{\cdot}_{\mathcal{H}^1}$. We further introduce the bilinear form
\begin{align*}
	\bigscp{(\zeta,\zeta_\Ga)}{(\xi,\xi_\Ga)}_{L,\beta} \coloneqq &\intO \Grad\zeta\cdot\Grad\xi\dx + \intG \Gradg\zeta_\Ga\cdot\Gradg\xi_\Ga\dG \\
	&\qquad + \chi(L) \intG (\beta\zeta_\Ga - \zeta)(\beta\xi_\Ga - \xi)\dG
\end{align*}
for all $(\zeta,\zeta_\Ga), (\xi,\xi_\Ga)\in\mathcal{H}^1$, where
\begin{align*}
	\chi:[0,\infty]\rightarrow [0,\infty), \qquad 
    \chi(r) \coloneqq 
    \begin{cases} 
    \frac{1}{r} &\text{if~}r\in(0,\infty), \\
	0 &\text{if~} r\in\{0,\infty\}.
	\end{cases}
\end{align*}
Moreover, we define
\begin{align*}
	\norm{(\zeta,\zeta_\Ga)}_{L,\beta} \coloneqq \big((\zeta,\zeta_\Ga),(\zeta,\zeta_\Ga)\big)_{L,\beta}^{\frac12}
\end{align*}
for all $(\zeta,\zeta_\Ga)\in\mathcal{H}^1$. The bilinear form defines an inner product on $\mathcal{V}^1_{L,\beta}$, and $\norm{\cdot}_{L,\beta}$ defines a norm on $\mathcal{V}^1_{L,\beta}$, that is equivalent to the norm $\norm{\cdot}_{\mathcal{H}^1}$, see, e.g., \cite[Corollary A.2]{Knopf2021}. In particular, $(\mathcal{V}^1_{L,\beta},\scp{\cdot}{\cdot}_{L,\beta},\norm{\cdot}_{L,\beta})$ is a Hilbert space. 

We further recall the following bulk-surface Poincar\'{e} inequality, which was established in \cite[Lemma A.1]{Knopf2021}.

\medskip

\begin{lemma}\label{Lemma:Poincare}
	Let $K\in[0,\infty)$ and $\alpha,\beta\in\R$ with $\alpha\beta\abs{\Om} + \abs{\Ga} \neq 0$. Then there exists a constant $C_K > 0$ depending only on $K,\alpha,\beta$ and $\Om$ such that
	\begin{align*}
		\norm{(\zeta,\zeta_\Ga)}_{\mathcal{L}^2} \leq C_K \norm{(\zeta,\zeta_\Ga)}_{K,\alpha}
	\end{align*}
	for all pairs $(\zeta,\zeta_\Ga)\in\mathcal{H}^1_{K,\alpha}$ satisfying $\mean{\zeta}{\zeta_\Ga} = 0$.
\end{lemma}

\subsection{Spaces of tangential and divergence-free vector fields}

For $p\in[2,\infty]$, we introduce the spaces 
\begin{align*}
	&\mathbf{L}^p_\Div(\Om) \coloneqq \{ \bv\in\mathbf{L}^p(\Om): \Div\,\bv = 0 \ \text{in~}\Om, \ \bv\cdot\n = 0 \ \text{on~}\Ga\} \,, \\
	&\mathbf{L}^p_\Div(\Ga) \coloneqq \{ \bw\in\mathbf{L}^p(\Ga): \Divg\,\bw = 0, \ \bw\cdot\n = 0 \ \text{on~}\Ga\},
\end{align*}
and we set
\begin{align*}
    \mathbfcal{L}^p_\Div \coloneqq \mathbf{L}^p_\Div(\Om)\times\mathbf{L}^p_\Div(\Ga).
\end{align*}
Then, for $s\geq 0$ and $p\in[2,\infty]$, we define $\mathbf{W}^{s,p}_\Div(\Om) = \mathbf{W}^{s,p}(\Om)\cap \mathbf{L}^2_\Div(\Om)$. Analogously, we define $\mathbf{W}^{s,p}_\Div(\Ga) = \mathbf{W}^{s,p}(\Ga)\cap \mathbf{L}^2_\Div(\Ga)$, and then set $\mathbfcal{W}^{s,p}_\Div = \mathbf{W}^{s,p}_\Div(\Om)\times\mathbf{W}^{s,p}_\Div(\Ga)$. Furthermore, for $s>\frac 12$, we set 
\begin{align*}
    \mathbfcal{W}^{s,p}_0 &\coloneqq \{(\bv,\bw)\in\mathbfcal{W}^{s,p}: \bv\cdot\n = 0, \ \bv\vert_\Ga = \bw \ \text{on~}\Ga\}\,,
    \\
    \mathbfcal{W}^{s,p}_{0,\Div} &\coloneqq \mathbfcal{W}^{s,p}_0 \cap \mathbfcal{W}^{s,p}_\Div \,.
\end{align*}
As before, we use the notation $\mathbfcal{H}^s_0 = \mathbfcal{W}^{s,2}_0$ as well as $\mathbfcal{H}^s_{0,\Div} = \mathbfcal{W}^{s,2}_{0,\Div}$. 

Let $r\in[2,\infty)$. We define the Helmholtz projection on $\Om$ as
\begin{align*}
    \mathbf{P}_\Div^\Om:\mathbf{L}^r(\Om)\rightarrow\mathbf{L}^r_\Div(\Om),
    \quad
    \mathbf{P}_\Div^\Om(\bv) = \bv - \Grad p,
\end{align*}
where $p\in \dot{W}^{1,r}(\Om)$ is (up to an additive constant) uniquely determined by
\begin{align*}
	\intO \Grad p \cdot \Grad\zeta\dx = \intO \bv\cdot\Grad\zeta\dx
    \quad\text{for all $\zeta\in \dot{W}^{1,r'}(\Om)$},
\end{align*}
cf.~\cite[Theorem~1.4]{Simader1992}. Here, $r'$ denotes the dual Sobolev exponent of $r$. In a similar manner, the surface Helmholtz projection is defined via
\begin{align*}
    \mathbf{P}_\Div^\Ga: \mathbf{L}^r_{\tau}(\Gamma) \to \mathbf{L}^r_{\Div}(\Gamma),
    \quad
    \mathbf{P}_\Div^\Ga(\bw) = \bw - \Gradg q, %\qquad\bw\in \mathbf{L}^r_{\tau}(\Ga),
\end{align*}
where $q \in \dot{W}^{1,r}(\Ga)$ is (up to an additive constant) uniquely determined by
\begin{align*}
    \intG \Gradg q\cdot\Gradg \zeta_\Ga \dG = \intG \bw\cdot\Gradg\zeta_\Ga\dG \qquad\text{for all~}\zeta_\Ga\in \dot{W}^{1,r^\prime}(\Ga),
\end{align*}
cf.~\cite[Lemma A.1]{Simonett2022}.
Here, $r'$ denotes again the dual Sobolev exponent of $r$, i.e., $r^\prime = \frac{r}{r-1}$.
The existence of this projection is also related to the surface Helmholtz decomposition, which can be found, for example, in \cite[Theorem~4.2]{Reusken2020}.
According to \cite[Sect.~2]{Simonett2022}, we further have
\begin{align*}
    \intG \mathbf{P}_\Div^\Ga(\bw) \cdot \mathbf u\dG = \intG \bw\cdot \mathbf{P}_\Div^\Ga(\mathbf u)\dG
\end{align*}
for all $\bw\in \mathbf{L}^r_\tau(\Ga)$ and $\mathbf u\in \mathbf{L}^{r^\prime}_\tau(\Ga)$.

We conclude this preliminary section with the following bulk-surface Korn-type inequality, which has been established in \cite[Lemma~5.1]{Knopf2025a}.

\begin{lemma}
    There exists a constant $C_K > 0$ such that
    \begin{align*}
        \norm{(\bv,\bw)}_{\mathbfcal{H}^1} \leq C_K\Big(\norm{\D\bv}_{\mathbf{L}^2(\Om)} + \norm{\Dg\bw}_{\mathbf{L}^2(\Ga)} + \norm{\bw}_{\mathbf{L}^2(\Ga)}\Big)
    \end{align*}
    for all $(\bv,\bw)\in\mathbfcal{H}^1_0$.
\end{lemma}

\section{Main results}
\label{Section:MainResults}

\subsection{Main assumptions}

We begin by stating our main assumption used throughout this manuscript.

\begin{enumerate}[label=\textnormal{\bfseries(A\arabic*)}]
	\item \label{Assumption:NSCH:1} The set $\Om\subset\R^d$, $d\in\{2,3\}$, is a non-empty, bounded domain with $C^2$-boundary $\Gamma \coloneqq \partial\Omega$.
	\item \label{Assumption:NSCH:2} The constants occurring in system \eqref{System:NSCH} satisfy $\alpha\in[-1,1]$ and $\beta\in\R$ with $\alpha\beta\abs{\Om} + \abs{\Ga}\neq 0$.
    \item \label{Assumption:NSCH:Density} The density functions $\rho$ and $\sigma$ are given by
    \begin{align*}
    	\rho(s) &= \frac{\tilde{\rho}_2 - \tilde{\rho}_1}{2}s + \frac{\tilde{\rho}_1 + \tilde{\rho}_2}{2}, \\
    	\sigma(s) &= \frac{\tilde{\sigma}_2 - \tilde{\sigma}_1}{2}s + \frac{\tilde{\sigma}_1 + \tilde{\sigma}_2}{2}
    \end{align*}
    for all $s\in[-1,1]$, respectively. Here, $\tilde{\rho}_1, \tilde{\rho}_2>0$ and $\tilde{\sigma}_1, \tilde{\sigma}_2>0$ are the specific densities of the two fluid components in the bulk and on the surface, respectively. Moreover, we use the notation
    \begin{alignat*}{2}
        &\rho_\ast \coloneqq \min\{\tilde\rho_1,\tilde\rho_2\}, \qquad 
        &&\rho^\ast \coloneqq \max\{\tilde\rho_1,\tilde\rho_2\},
        \\
        &\sigma_\ast \coloneqq \min\{\tilde\sigma_1,\tilde\sigma_2\}, \qquad 
        &&\sigma^\ast \coloneqq \max\{\tilde\sigma_1,\tilde\sigma_2\}.
    \end{alignat*}
    By this definition, we clearly have
    \begin{align*}
        0 < \rho_\ast \leq \rho(s) \leq \rho^\ast \quad\text{and}\quad 0 < \sigma_\ast \leq \sigma(s) \leq \sigma^\ast \qquad\text{for all~}s\in[-1,1].
    \end{align*}
    \item \label{Assumption:NSCH:Coefficients} The mobility functions $m_\Om,m_\Ga:[-1,1]\rightarrow\R$, the viscosity functions $\nu_\Om, \nu_\Ga:[-1,1]\rightarrow\R$ and the friction coefficient $\gamma:[-1,1]^2\rightarrow\R$ are continuous, bounded and uniformly positive. In particular, there exist constants $m_\ast, m^\ast, \nu_\ast, \nu^\ast, \gamma_\ast$ and $\gamma^\ast$ such that
    \begin{align*}
    	0 < m_\ast \leq m_\Om(s), m_\Ga(s) \leq m^\ast, \quad 0 < \nu_\ast \leq \nu_\Om(s), \nu_\Ga(s) \leq \nu^\ast \quad\text{for all~}s\in[-1,1],
    \end{align*}
    and
    \begin{align*}
    	0 < \gamma_\ast \leq \gamma(s,r) \leq \gamma^\ast \quad\text{for all~}(s,r)\in[-1,1]^2.
    \end{align*}
	\item \label{Assumption:NSCH:Potentials} We assume that the potentials $F,G:[-1,1]\rightarrow\R$ are of the form
	\begin{align*}
		F(s) = F_0(s) - \frac{c_F}{2}s^2, \qquad G(s) = G_0(s) - \frac{c_G}{2}s^2 \qquad\text{for~}s\in[-1,1]
	\end{align*}
	with constants $c_F,c_G > 0$, where $F_0, G_0\in C([-1,1])\cap C^2(-1,1)$ are such that
    \begin{align*}
        \lim_{s\searrow -1} F_0^\prime(s) = \lim_{s\searrow -1} G_0^\prime(s) = - \infty \quad\text{and}\quad \lim_{s\nearrow 1} F_0^\prime(s) = \lim_{s\nearrow 1} G_0^\prime(s) = + \infty.
    \end{align*}
    Moreover, we assume that
    \begin{align*}
        F_0^{\prime\prime}(s) \geq c_F, \qquad G_0^{\prime\prime}(s) \geq c_G \qquad\text{for all~}s\in(-1,1).
    \end{align*}
    Without loss of generality, we assume that $F_0(0) = G_0(0) = 0$ and $F_0^\prime(0) = G_0^\prime(0) = 0$. Lastly, we require the singular part of the boundary potential to dominate the singular part of the bulk potential in the sense that there exist constants $\kappa_1, \kappa_2 > 0$ such that
    \begin{align}\label{Assumption:NSCH:DominationProperty}
        \abs{F_0^\prime(\alpha s)} \leq \kappa_1\abs{G_0^\prime(s)} + \kappa_2 \qquad\text{for all~}s\in(-1,1).
    \end{align}
\end{enumerate}

\subsection{Existence of weak solutions}

Let us introduce a suitable notion of a weak solution to \eqref{System:NSCH}.

\begin{definition}\label{NSCH:Definition:WeakSolution}
Suppose that the assumptions \ref{Assumption:NSCH:1}-\ref{Assumption:NSCH:Potentials} hold. Let $K,L\in[0,\infty]$, and let $(\bv_0,\bw_0)\in\mathbfcal{L}^2_\Div$ and $(\phi_0,\psi_0)\in\mathcal{H}^1_{K,\alpha}$ be initial data such that 
    \begin{subequations}\label{NSCH:cond:init}
        \begin{align}\label{NSCH:cond:init:int}
            \norm{\phi_0}_{L^\infty(\Om)} \leq 1 \quad\text{and}\quad \norm{\psi_0}_{L^\infty(\Ga)} \leq 1.
        \end{align}
        In addition, we further assume that
        \begin{align}\label{NSCH:cond:init:mean:L}
            \beta\mean{\phi_0}{\psi_0}, \,\mean{\phi_0}{\psi_0}\in(-1,1)  \qquad\text{if~} L\in[0,\infty),
        \end{align}
        and
        \begin{align}\label{NSCH:cond:init:mean:infty}
            \meano{\phi_0},\, \meang{\psi_0}\in(-1,1) \qquad\text{if~} L = \infty.
        \end{align}
    \end{subequations}
    The sextuplet $(\bv,\bw,\phi,\psi,\mu,\theta)$ is called a weak solution to system \eqref{System:NSCH} on $[0,\infty)$ if the following properties hold:
    \begin{enumerate}[label=\textnormal{(\roman*)}, ref=\thetheorem(\roman*)]
        \item \label{NSCH:WeakSolution:REG} The functions $\bv, \bw, \phi, \psi, \mu$ and $\theta$ have the regularities
        \begin{subequations}
            \begin{align}
                (\bv,\bw) &\in BC_w([0,\infty);\mathbfcal{L}^2_\Div)\cap L^2(0,\infty;\mathbfcal{H}^1_{0,\Div}), \label{NSCH:REG:VW}\\
                (\phi,\psi) &\in BC_w([0,\infty);\mathcal{H}_{K,\alpha}^1)\cap L^2_\uloc([0,\infty);\mathcal{H}^2), \label{NSCH:REG:PP}\\
                (F^\prime(\phi),G^\prime(\psi))&\in L^2_\uloc([0,\infty);\mathcal{L}^2),\label{NSCH:REG:POT}\\
                (\mu,\theta)&\in L^2_\uloc([0,\infty);\mathcal{H}_{L,\beta}^1), \label{NSCH:REG:MT}
            \end{align}
            and it holds
            \begin{align*}
            	\abs{\phi} < 1 \quad\text{a.e.~in~}Q \quad\text{and}\quad \abs{\psi} < 1 \quad\text{a.e.~on~}\Sigma.
            \end{align*}
    \end{subequations}
        \item \label{NSCH:WeakSolution:IC} The initial conditions are satisfied in the following sense:
        \begin{align}\label{NSCH:IC}
            \big(\bv,\phi\big)\vert_{t=0} = \big(\bv_0,\phi_0\big) \quad\text{in~}\Om, \quad \big(\bw,\psi\big)\vert_{t=0} = \big(\bw_0,\psi_0\big) \quad\text{on~}\Ga.
        \end{align}
        \item The variational formulations
        \begin{subequations}\label{NSCH:WeakSolution:WF}
            \begin{align}
                &-\int_0^\infty\intO\rho(\phi)\bv\cdot\delt\wv\dxt - \int_0^\infty\intG\sigma(\psi)\bw\cdot\delt\ww\dGt \nonumber \\
                &\qquad - \int_0^\infty\intO\Div(\rho(\phi)\bv\otimes\bv)\cdot\wv\dxt + \int_0^\infty\intG\Divg(\sigma(\psi)\bw\otimes\bw)\cdot\ww\dGt \nonumber \\
                &\qquad + \int_0^\infty\intO\big(\bv\otimes\J\big):\Grad\wv\dxt + \int_0^\infty\intG\big(\bw\otimes\K\big):\Gradg\ww\dGt \nonumber \\
                &\qquad + \int_0^\infty\intO2\nu_\Om(\phi)\D\bv:\D\wv\dxt + \int_0^\infty\intG2\nu_\Ga(\psi)\Dg\bw:\Dg\ww\dGt \nonumber \\
                &\qquad + \int_0^\infty\intG \gamma(\phi,\psi)\bw\cdot\ww\dGt \label{NSCH:WF:VW} \\
                &\quad = \int_0^\infty\intO\mu\Grad\phi\cdot\wv\dxt + \int_0^\infty\intG\theta\Gradg\psi\cdot\ww\dGt \nonumber \\
                &\qquad + \frac{\chi(L)}{2}\Big(\beta\frac{\tilde\sigma_2 - \tilde\sigma_1}{2} + \frac{\tilde\rho_2 - \tilde\rho_1}{2}\Big)\int_0^\infty\intG(\beta\theta - \mu)\bw\cdot\ww\dGt, \nonumber
            \end{align}
            and
            \begin{align}\label{NSCH:WF:PP}
                &-\int_0^\infty\intO\phi\,\delt\zeta\dxt - \int_0^\infty\intG\psi\,\delt\zeta_\Ga\dGt \nonumber \\
                &\qquad - \int_0^\infty\intO \phi\bv\cdot\Grad\zeta\dxt - \int_0^\infty\intG \psi\bw\cdot\Gradg\zeta_\Ga\dGt \\
                &\quad= - \int_0^\infty\intO m_\Om(\phi)\Grad\mu\cdot\Grad\zeta\dxt - \int_0^\infty\intG  m_\Ga(\psi)\Gradg\theta\cdot\Gradg\zeta_\Ga\dGt \nonumber \\
                &\qquad - \chi(L)\int_0^\infty\intG(\beta\theta-\mu)(\beta\zeta_\Ga - \zeta)\dGt \nonumber
            \end{align}
        \end{subequations}
        hold for all $(\wv,\ww)\in C_c^\infty(0,\infty;\mathbfcal{H}^2_{0,\Div})$ and all $(\zeta,\zeta_\Ga)\in C_c^\infty(0,\infty;\mathcal{H}^1_{L,\beta})$, where
        \begin{subequations}\label{NSCH:WeakSolution:MuTheta}
        \begin{alignat}{2}
        	&\mu = -\Lap\phi + F^\prime(\phi) &&\qquad\text{a.e.~in~}Q, \\
        	&\theta = -\Lapg\psi + G^\prime(\psi) + \alpha\deln\phi &&\qquad\text{a.e.~on~}\Sigma, \\
        	&K\deln\phi = \alpha\psi - \phi, \qquad K\in[0,\infty],
        	&&\qquad\text{a.e.~on~}\Sigma,
        \end{alignat}
        \end{subequations}
        and
        \begin{alignat*}{2}
        	\rho(\phi) &= \frac{\tilde\rho_2 - \tilde\rho_1}{2}\phi + \frac{\tilde\rho_1 + \tilde\rho_2}{2} &&\qquad\text{a.e.~in~}Q, \\
        	\sigma(\psi) &= \frac{\tilde\sigma_2 - \tilde\sigma_1}{2}\psi + \frac{\tilde\sigma_ 1 + \tilde\sigma_2}{2} &&\qquad\text{a.e.~on~}\Sigma, \\
        	\J &= - \frac{\tilde\rho_2 - \tilde\rho_1}{2}m_\Om(\phi)\Grad\mu &&\qquad\text{a.e.~in~}Q, \\
        	\K &= - \frac{\tilde\sigma_2 - \tilde\sigma_1}{2}m_\Ga(\psi)\Gradg\theta &&\qquad\text{a.e.~on~}\Sigma.
        \end{alignat*}
        \item \label{NSCH:WeakSolution:MCL} The functions $\phi$ and $\psi$ satisfy the mass conservation law
        \begin{align}\label{NSCH:WeakSolution:MassConservationLaw}
            \begin{dcases}
                \beta\intO \phi(t)\dx + \intG \psi(t)\dG = \beta\intO \phi_0 \dx + \intG \psi_0\dG &\textnormal{if~} L\in[0,\infty), \\
                \intO\phi(t)\dx = \intO\phi_0\dx \quad\textnormal{and}\quad \intG\psi(t)\dG = \intG\psi_0\dG &\textnormal{if~} L = \infty
            \end{dcases}
        \end{align}
        for all $t\in[0,\infty)$.
        \item \label{NSCH:WeakSolution:EI} The energy inequality
        \begin{align}\label{NSCH:WeakSolution:EnergyInequality}
            &E_{\mathrm{tot}}(\bv(t),\bw(t),\phi(t),\psi(t)) \nonumber \\
            &\qquad + \int_s^t\intO 2\nu_\Om(\phi)\abs{\D\bv}^2\dxtau + \int_s^t\intG 2\nu_\Ga(\psi)\abs{\Dg\bw}^2\dGtau + \int_s^t\intG \gamma(\phi,\psi)\abs{\bw}^2 \dGtau \\
            &\qquad + \int_s^t\intO m_\Om(\phi)\abs{\Grad\mu}^2\dxtau + \int_s^t\intG m_\Ga(\psi)\abs{\Gradg\theta}^2\dGtau + \chi(L)\int_s^t\intG (\beta\theta - \mu)^2\dGtau \nonumber \\
            &\quad\leq E_\tot(\bv(s),\bw(s),\phi(s),\psi(s)) \nonumber
        \end{align}
        holds for all $t\in[s,\infty)$ and almost all $s\in[0,\infty)$ including $s = 0$.
    \end{enumerate}
\end{definition}

We recall that the total energy associated with the system \eqref{System:NSCH} is given by
\begin{align*}
	E_\tot(\bv,\bw,\phi,\psi) &\coloneqq \intO \frac12\rho(\phi)\abs{\bv}^2\dx + \intG \frac12\sigma(\psi)\abs{\bw}^2\dG + \intO \frac12\abs{\Grad\phi}^2 + F(\phi)\dx \\
	&\quad + \intG \frac12\abs{\Gradg\psi}^2 + G(\psi)\dG + \chi(K)\intG \frac12(\alpha\psi - \phi)^2\dG.
\end{align*}

\begin{remark}
In view of the regularities of a weak solution $(\bv,\bw,\phi,\psi,\mu,\theta)$, it is straightforward to show with the weak formulation \eqref{NSCH:WF:PP} that 
\begin{align*}
	(\delt\phi,\delt\psi)\in L^2(0,\infty;(\mathcal{H}^1_{L,\beta})^\prime).
\end{align*}
Moreover, it holds that
\begin{align*}
	&\big\langle(\delt\phi,\delt\psi),(\zeta,\zeta_\Ga)\big\rangle_{\mathcal{H}^1_{L,\beta}}  - \intO\phi\bv\cdot\Grad\zeta\dx - \intG \psi\bw\cdot\Gradg\zeta_\Ga\dG \\
	&\quad = - \intO m_\Om(\phi)\Grad\mu\cdot\Grad\zeta\dx - \intG m_\Ga(\psi)\Gradg\theta\cdot\Gradg\zeta_\Ga\dG - \chi(L)\intG (\beta\theta - \mu)(\beta\zeta_\Ga - \zeta)\dG
\end{align*}
a.e.~on $(0,\infty)$ for all $(\zeta,\zeta_\Ga)\in\mathcal{H}^1_{L,\beta}$. In particular, $(\phi,\psi,\mu,\theta)$ is a weak solution to the convective bulk-surface Cahn--Hilliard equation with initial datum $(\phi_0,\psi_0)$ and velocity fields $(\bv,\bw)$ in the sense of \cite[Definition~3.1]{Knopf2025}.
\end{remark}

We are now in a position to state our main theorem of this manuscript, which establishes the existence of a global-in-time weak solution to system \eqref{System:NSCH}.

\begin{theorem}\label{NSCH:Theorem:ExistenceWeakSolutions}
Suppose that the assumptions \ref{Assumption:NSCH:1}-\ref{Assumption:NSCH:Potentials} hold.
Let $K,L\in[0,\infty]$, $(\bv_0,\bw_0)\in\mathbfcal{L}^2_\Div$, and let $(\phi_0,\psi_0)\in\mathcal{H}^1_{K,\alpha}$ satisfy \eqref{NSCH:cond:init}. If $L = 0$, we additionally assume that
    \begin{align}\label{Densities:Comp:Weak}
    	\beta(\tilde\sigma_2 - \tilde\sigma_1) = -(\tilde\rho_2 - \tilde\rho_1).
    \end{align}
    Then, there exists at least one global-in-time weak solution $(\bv,\bw,\phi,\psi,\mu,\theta)$ to system \eqref{System:NSCH} in the sense of Definition~\ref{NSCH:Definition:WeakSolution}.
    In addition, the following separate energy inequalities hold: for all $t\in[s,\infty)$ and almost every $s\in[0,\infty)$, including $s = 0$, it holds 
    \begin{align}\label{EnergyInequality:Kinetic}
        & E_\mathrm{kin}(\bv(t),\bw(t),\phi(t),\psi(t)) \nonumber \\%\intO\frac12\rho(\phi(t))\abs{\bv(t)}^2\dx + \intG\frac12\sigma(\psi(t))\abs{\bw(t))}^2\dG \nonumber \\
        &\qquad + \int_s^t\intO 2\nu_\Om(\phi)\abs{\D\bv}^2\dxtau + \int_s^t\intG 2\nu_\Ga(\psi)\abs{\Dg\bw}^2\dGtau + \int_s^t\intG \gamma(\phi,\psi)\abs{\bw}^2\dGtau \nonumber \\
        &\quad\leq E_\mathrm{kin}(\bv(s),\bw(s),\phi(s),\psi(s))%\intO\frac12\rho(\phi(s))\abs{\bv(s)}^2\dx + \intG\frac12\sigma(\psi(s))\abs{\bw(s))}^2\dG - \int_s^t\intO \phi\bv\cdot\Grad\mu\dxtau \nonumber \\
        - \int_s^t\intO \phi\bv\cdot\Grad\mu\dxtau- \int_s^t\intG \psi\bw\cdot\Gradg\theta\dGtau
    \end{align}
    as well as 
    \begin{align}\label{EnergyInequality:Free}
        &E_\free(\phi(t),\psi(t)) \nonumber \\
        &\qquad + \int_s^t\intO m_\Om(\phi)\abs{\Grad\mu}^2\dxtau + \int_s^t\intG m_\Ga(\psi)\abs{\Gradg\theta}^2\dGtau + \chi(L)\int_s^t\intG(\beta\theta-\mu)^2\dGtau \nonumber \\
        &\quad\leq E_\free(\phi(s),\psi(s)) + \int_s^t\intO \phi\bv\cdot\Grad\mu\dxtau + \int_s^t\intG \psi\bw\cdot\Gradg\theta\dGtau. 
    \end{align}
\end{theorem}

We remark the following regarding the compatibility assumption \eqref{Densities:Comp:Weak}.

\begin{remark} \label{REM:COMP}
	The compatibility condition \eqref{Densities:Comp:Weak} in the case $L = 0$ formally ensures that
        \begin{align*}%\label{Term:Comp:Weak}
        \frac12(\J\cdot\n)\bw + \frac12(\J_\Ga\cdot\n)\bw
        = \Big(\beta\frac{\tilde\sigma_2 - \tilde\sigma_1}{2} + \frac{\tilde\rho_2 - \tilde\rho_1}{2}\Big)m_\Om(\phi)\deln\mu\,\bw
        = 0 \qquad\text{on~}\Sigma.
        \end{align*}
    Consequently, the corresponding term in the weak formulation \eqref{NSCH:WF:VW} vanishes, which is crucial for treating the case $L = 0$. Condition \eqref{Densities:Comp:Weak} is related to the compatibility condition $\tilde\rho_1 = \tilde\rho_2$ imposed in \cite{Gal2024} for the analysis of the system introduced in \cite{Giorgini2023} in the case $L = 0$. From the viewpoint of mathematical analysis, the present model \eqref{System:NSCH} therefore has an advantage that it allows for fluids with unmatched densities even when $L = 0$.

    Although condition \eqref{Densities:Comp:Weak} is needed in our proof, the following argument indicates that is is also difficult to avoid when employing other standard existence methods for $L = 0$. One natural approach would be to pass to the asymptotic limit $L\rightarrow 0$, as previously done in \cite{Knopf2025a}. More precisely, for each $L\in(0,\infty)$, one could consider a corresponding weak solution
    \begin{align*}
        (\bv_L,\bw_L,\phi_L,\psi_L,\mu_L,\theta_L)
    \end{align*}
    and derive estimates uniform in $L$ from the energy inequality. The term in the weak formulation corresponding to the expression above would then be
    \begin{align*}%\label{Term:Comp:WeakF}
        \frac1L\Big(\beta\frac{\tilde\sigma_2-\tilde\sigma_1}{2} + \frac{\tilde\rho_2-\tilde\rho_1}{2}\Big)\int_\Sigma(\beta\theta_L-\mu_L)\bw_L\cdot\ww\dGt.
    \end{align*}
    However, the energy inequality only yields
    \begin{align*}
        \norm{\beta\theta_L-\mu_L}_{L^2(0,\infty;L^2(\Ga))} \leq C\sqrt{L}.
    \end{align*}
    Because of the prefactor $\frac1L$, this estimate is insufficient to pass to the limit $L\rightarrow 0$ in the preceding term. Thus, without the compatibility condition \eqref{Densities:Comp:Weak}, the standard asymptotic approach does not appear to provide the estimates required to treat $L = 0$.
\end{remark}

\section{Mosco convergence of the convex part of the free energy}
\label{Section:MoscoConvergence}

In this section, we establish the Mosco convergence of the convex part of the free bulk-surface energy as $K\rightarrow 0$ and $K\rightarrow\infty$. As a consequence, we obtain the corresponding graph convergence of its subdifferential, which will be used in the analysis of the asymptotic limits later on. To this end, let $K\in[0,\infty]$ be arbitrary, and recall that the potentials can be decomposed as $F(s) = F_0(s) - \frac{c_F}{2}s^2$ and $G(s) = G_0(s) - \frac{c_G}{2}s^2$ for $s\in[-1,1]$, see \ref{Assumption:NSCH:Potentials}. We then consider the functional $\widetilde{E}_\free^K: \mathcal{L}^2 \to (-\infty,\infty]$ defined as
\begin{align*}
    \widetilde{E}_\free^K(u,v) = 
    \begin{cases} 
    \displaystyle\intO \frac12\abs{\Grad u}^2 + F_0(u)\dx + \intG \frac12\abs{\Gradg v}^2 + G_0(v)\dG 
    &\smash{\raisebox{-1.6ex}
    {\text{for~}$(u,v)\in\mathrm{Dom}\,\widetilde E_\free^K$}} 
    \\
    \quad + \displaystyle\chi(K) \intG \frac12 (\alpha v - u)^2 \dG, 
    \\
    +\infty &\text{else},
    \end{cases}
\end{align*}
where
\begin{equation*}
    \mathrm{Dom}\,\widetilde{E}_\free^K = \big\{(u,v)\in\mathcal{H}^1_{K,\alpha}: \abs{u}\leq 1 
    \text{~a.e.~on~}\Om, \ \abs{v}\leq 1 \text{~a.e.~on~}\Ga\big\}.
\end{equation*}
In view of the assumptions on the singular parts $F_0$ and $G_0$, one readily checks that the functional $\widetilde E_\free^K$ is proper, lower semicontinuous, and convex, see also \cite[Lemma~4.1]{Giorgini2026}.

\subsection{Mosco convergence}

We start by recalling the definition of Mosco convergence, see \cite[Section~3.3, Definition~3.17]{Attouch1984}

\begin{definition}
    Let $H$ be a Hilbert space, and let $\Phi,\Phi_n:H\rightarrow [0,+\infty]$, $n\in\N$, be proper, lower semicontinuous and convex functionals on $H$. We say that $(\Phi_n)_{n\in\N}$ converges to $\Phi$ in the sense of Mosco if the following two conditions hold:
    \begin{enumerate}[label=\textnormal{\bfseries(M\arabic*)}]
        \item \label{Mosco1} For every sequence $(x_n)_{n\in\N}\subset H$ converging weakly to $x\in H$ it holds that
        \begin{align*}
            \Phi(x) \leq \liminf_{n\rightarrow\infty} \Phi_n(x_n).
        \end{align*}
        \item \label{Mosco2} For every $x\in H$ there exists a sequence $(x_n)_{n\in\N}\subset H$ converging strongly to $x$ such that
        \begin{align*}
            \Phi(x) \geq \limsup_{n\rightarrow\infty} \Phi_n(x_n).
        \end{align*}
    \end{enumerate}
\end{definition}

We are now able to prove the Mosco convergence of the functional $\widetilde E_\free^K$ as $K\rightarrow 0$ and $K\rightarrow\infty$, respectively. In this context, we also refer to \cite{Fukao2026}, where a similar energy functional was considered without additional singular nonlinear terms.
\begin{proposition}\label{Proposition:MoscoConvergence}
    The functional $\widetilde E_\free^K$ converges in the sense of Mosco to $\widetilde E_\free^0$ and $\widetilde E_\free^\infty$, respectively, as $K\rightarrow 0$ and $K\rightarrow\infty$, respectively.
\end{proposition}

\begin{proof}
    We first consider the case $K\rightarrow 0$. To this end, let $\{K_n\}_{n\in\N}\subset(0,\infty)$ be a sequence with $K_n\rightarrow 0$ as $n\rightarrow\infty$, and let $\{(u_n,v_n)\}_{n\in\N}\subset\mathcal{L}^2$ be a sequence converging weakly in $\mathcal{L}^2$ to $(u,v)\in\mathcal{L}^2$ as $n\rightarrow\infty$. Without loss of generality, we can assume that 
    \begin{align*}
        \liminf_{n\rightarrow\infty} \widetilde E_\free^{K_n}(u_n,v_n) < + \infty,
    \end{align*}
    as otherwise \ref{Mosco1} is trivially fulfilled. Then, going over to a subsequence, we have 
    \begin{align*}
        \lim_{k\rightarrow\infty} \widetilde E_\free^{K_{n_k}}(u_{n_k},v_{n_k}) = \liminf_{n\rightarrow\infty} \widetilde E_\free^{K_n}(u_n,v_n),
    \end{align*}
    as well as
    \begin{subequations}
        \begin{alignat}{2}
            (u_{n_k},v_{n_k}) &\rightarrow (u,v) &&\qquad\text{weakly in~}\mathcal{H}^1, \label{Prelim:Mosco:WeakConv} \\
            (u_{n_k},v_{n_k}) &\rightarrow (u,v) &&\qquad \text{strongly in~}\mathcal{L}^2 \label{Prelim:Mosco:StrongConv}
        \end{alignat}
    \end{subequations}
    as $k\rightarrow\infty$. Moreover, as the energy is bounded along this subsequence, it in particular holds that
    \begin{align}
        \sup_{k\in\N} \frac{1}{K_{n_k}}\intG \frac12(\alpha v_{n_k} - u_{n_k})^2\dG < +\infty.
    \end{align}
    Consequently, we readily infer that
    \begin{align*}
        \alpha v_{n_k} - u_{n_k} \rightarrow 0 \qquad\text{strongly in~}L^2(\Ga)
    \end{align*}
    as $k\rightarrow\infty$. Since we already know that \eqref{Prelim:Mosco:WeakConv} holds, the continuity of the trace operator with the above convergence entails that $u = \alpha v$ a.e. on $\Ga$. Additionally, by possibly going over to another subsequence, we readily infer from $\abs{u_{n_k}}\leq 1$ a.e.~in $\Om$ and $\abs{v_{n_k}}\leq 1$ a.e.~on $\Ga$ that $\abs{u}\leq 1$ a.e.~in $\Om$ and $\abs{v}\leq 1$ a.e.~on $\Ga$, respectively. Then, from the strong convergence \eqref{Prelim:Mosco:StrongConv} and the uniform continuity of $F_0$ and $G_0$, the continuity of the Nemytskii operator directly gives that
    \begin{subequations}
        \begin{alignat}{2}
            F_0(u_{n_k})&\rightarrow F_0(u) &&\qquad\text{strongly in~}L^1(\Omega), \label{Prelim:Mosco:Conv:Pot:F:K0}\\
            G_0(v_{n_k})&\rightarrow G_0(v) &&\qquad\text{strongly in~}L^1(\Ga)\label{Prelim:Mosco:Conv:Pot:G:K0}
        \end{alignat}
    \end{subequations}
    as $k\rightarrow\infty$. Thus, we deduce with \eqref{Prelim:Mosco:WeakConv} and \eqref{Prelim:Mosco:Conv:Pot:F:K0}-\eqref{Prelim:Mosco:Conv:Pot:G:K0}, and the weak lower semicontinuity of norms that
    \begin{align*}
        \widetilde E_\free^0(u,v) \leq \liminf_{k\rightarrow\infty} \widetilde E_\free^{K_{n_k}}(u_{n_k},v_{n_k}) = \lim_{k\rightarrow\infty} \widetilde E_\free^{K_{n_k}}(u_{n_k},v_{n_k}) = \liminf_{n\rightarrow\infty} \widetilde E_\free^{K_n}(u_n,v_n).
    \end{align*}
    
    For \ref{Mosco2}, let $(u,v)\in\mathcal{L}^2$. Without loss of generality, we can assume that $(u,v)\in\mathrm{Dom}\,\widetilde E_\free^0$, as otherwise $\widetilde E_\free^0(u,v) = +\infty$. We can thus choose the constant recovery sequence $(u_n,v_n)\coloneqq (u,v)\in \mathrm{Dom}\,\widetilde E_\free^0\subset \mathrm{Dom}\,\widetilde E_\free^{K_n}$ for all $n\in\N$, and readily obtain
    \begin{align*}
        \widetilde E_\free^0(u,v) = \limsup_{n\rightarrow\infty} \widetilde E_\free^{K_n}(u,v) = \limsup_{n\rightarrow\infty} \widetilde E_\free^{K_n}(u_n,v_n),
    \end{align*}
    which proves the Mosco convergence of $\widetilde E_\free^K$ as $K\rightarrow 0$ and finishes the proof in this case.

    Now, consider $K_n\rightarrow\infty$ as $n\rightarrow\infty$. Let $\{(u_n,v_n)\}_{n\in\N}\subset\mathcal{L}^2$ be a sequence converging weakly in $\mathcal{L}^2$ to $(u,v)\in\mathcal{L}^2$ as $n\rightarrow\infty$. Without loss of generality we can assume that 
    \begin{align*}
        \liminf_{n\rightarrow\infty} \widetilde E_\free^{K_n}(u_n,v_n) < + \infty,
    \end{align*}
    as otherwise \ref{Mosco1} is trivially fulfilled. Then, as for $K\rightarrow 0$, we find a subsequence such that
    \begin{align*}
        \lim_{k\rightarrow\infty} \widetilde E_\free^{K_{n_k}}(u_{n_k},v_{n_k}) = \liminf_{n\rightarrow\infty} \widetilde E_\free^{K_n}(u_n,v_n),
    \end{align*}
    and
    \begin{subequations}
        \begin{alignat}{2}
            (u_{n_k},v_{n_k}) &\rightarrow (u,v) &&\qquad\text{weakly in~}\mathcal{H}^1, \label{Prelim:Mosco:WeakConv:infty} \\
            (u_{n_k},v_{n_k}) &\rightarrow (u,v) &&\qquad \text{strongly in~}\mathcal{L}^2 \label{Prelim:Mosco:StrongConv:infty}
        \end{alignat}
    \end{subequations}
    as $k\rightarrow\infty$. Then, the same Nemytskii argument and weak lower semicontinuity give
    \begin{align*}
        \widetilde E_\free^\infty(u,v) \leq \liminf_{k\rightarrow\infty} \widetilde E_\free^\infty(u_{n_k},v_{n_k}) \leq \liminf_{k\rightarrow\infty} \widetilde E_\free^{K_{n_k}}(u_{n_k},v_{n_k}) = \liminf_{n\rightarrow\infty} \widetilde E_\free^{K_n}(u_n,v_n).
    \end{align*}
    For \ref{Mosco2}, let $(u,v)\in\mathcal{L}^2$. Without loss of generality we can again assume that $(u,v)\in\mathrm{Dom}\,\widetilde E_\free^\infty$, as otherwise $\widetilde E_\free^\infty(u,v) = +\infty$. We then choose again the constant sequence $(u_n,v_n)\coloneqq (u,v) \in\mathrm{Dom}\,\widetilde E_\free^\infty = \mathrm{Dom}\,\widetilde E_\free^{K_n}$ for all $n\in\N$, and observe that
    \begin{align*}
        \frac{1}{K_n}\intG \frac12(\alpha v_n - u_n)^2\dG = \frac{1}{K_n}\intG \frac12(\alpha v - u)^2\dG \rightarrow 0
    \end{align*}
    as $n\rightarrow\infty$, which entails that
    \begin{align*}
        \widetilde E_\free^\infty(u,v) = \limsup_{n\rightarrow\infty} \widetilde E_\free^{K_n}(u,v) = \limsup_{n\rightarrow\infty} \widetilde E_\free^{K_n}(u_n,v_n).
    \end{align*}
    This shows the Mosco convergence of $\widetilde E_\free^K$ as $K\rightarrow\infty$ and finishes the proof.
\end{proof}

\subsection{Subdifferentials and graph convergence}

We quickly recall the definition of the (convex) subdifferential $\del \widetilde{E}_\free^K$ of $\widetilde{E}_\free^K$. It can be formulated as an operator
$\del \widetilde{E}_\free^K:\mathcal{L}^2\rightarrow \mathcal{P}(\mathcal{L}^2)$, where
$\mathcal{P}(\mathcal{L}^2)$ denotes the power set of $\mathcal{L}^2$.
For any point $(u,v)\in\mathrm{Dom}\,\widetilde{E}_\free^K$, 
it holds $(\zeta,\zeta_\Ga)\in\del\widetilde E_\free^K(u,v)$ if and only if
\begin{align*}
 \bigscp{(\zeta,\zeta_\Ga)}{(u^\prime,v^\prime) - (u,v)}_{\mathcal{L}^2} \leq \widetilde E_\free^K(u^\prime,v^\prime) - \widetilde E_\free^K(u,v) \qquad\text{for all~}(u^\prime,v^\prime)\in\mathcal{L}^2.
\end{align*}
Moreover, the essential domain of $\del \widetilde{E}_\free^K$ is given by
\begin{equation*}
    D(\del \widetilde{E}_\free^K) 
    = \big\{ (u,v) \in \mathrm{Dom}\,\widetilde{E}_\free^K : 
    \del \widetilde{E}_\free^K(u,v) \neq \emptyset \big\}.
\end{equation*}
We also remind the reader of the following characterization of the subdifferential $\partial\widetilde E_\free^K$, which has been proven in \cite[Proposition~4.2]{Giorgini2026}.

\begin{proposition}\label{App:Proposition:Subdiff}
    The subdifferential $\partial \widetilde{E}_\free^K$ of $\widetilde{E}_\free^K$ is a maximal monotone operator on $\mathcal{L}^2$, whose essential domain is given by
    \begin{align*}
        D(\partial \widetilde{E}_\free^K) = \{(u, v)\in\mathcal{H}^2: (F_0^\prime(u), G_0^\prime(v))\in\mathcal{L}^2, \ K\deln u = \alpha v - u \ \text{a.e.~on }\Ga\}.
    \end{align*}
    For any $(u,v)\in D(\partial \widetilde{E}_\free^K)$, it holds
    \begin{align}
    \label{EQ:DELE}
        \partial \widetilde{E}_\free^K(u, v) = (-\Lap u + F_0^\prime(u), -\Lapg v + G_0^\prime(u) + \alpha\deln u).
    \end{align}
\end{proposition}

Next, we recall the notion of graph convergence, see \cite[Section~3.7, Definition~3.58]{Attouch1984}.

\begin{definition}
    Let $H$ be a Hilbert space, and let $A, A_n:H\rightarrow H$, $n\in\N$, be maximal monotone operators. We say that  $\{A_n\}_{n\in\N}$ converges to $A$ in the sense of graphs if for every $y\in Ax$, there exists a sequence $y_n\in A_nx_n$ such that $x_n\rightarrow x$ and $y_n\rightarrow y$ strongly in $H$ as $n\rightarrow\infty$.
\end{definition}

We conclude this section with the following result concerning the graph convergence of the subdifferential $\partial\widetilde E_\free^K$.

\begin{proposition}\label{Prelim:Proposition:GraphConvergence}
    Let $\{K_n\}_{n\in\N}\subset(0,\infty)$ be a sequence such that $K_n\rightarrow K$ as $n\rightarrow\infty$ where $K\in\{0,\infty\}$. Then $\partial\widetilde E_\free^{K_n}\rightarrow\partial\widetilde E_\free^K$ in the sense of graphs as $n\rightarrow\infty$.
    In particular, let $\{(u_n,v_n)\}_{n\in\N}\subset \mathrm{Dom}\, \widetilde E_\free^{K_n}$, and suppose there exists $(u,v), (f,f_\Ga)\in\mathcal{L}^2$ such that
    \begin{alignat*}{2}
        (u_n,v_n)&\rightarrow (u,v) &&\qquad\text{strongly in~}\mathcal{L}^2, \\
        \partial\widetilde E_\free^{K_n}(u_n,v_n) &\rightarrow (f,f_\Ga) &&\qquad\text{weakly in~}\mathcal{L}^2
    \end{alignat*}
    as $n\rightarrow\infty$. Then $(u,v)\in D(\partial\widetilde E_\free^K)$ and $\partial\widetilde E_\free^K(u,v) = (f,f_\Ga)$.
\end{proposition}

\begin{proof}
    In view of Proposition~\ref{Proposition:MoscoConvergence}, the asserted graph convergence follows from \cite[Section~3.8, Theorem~3.66]{Attouch1984}. The second assertion then follows by applying \cite[Section~3.7, Proposition~3.59]{Attouch1984}.
\end{proof}

\section{Existence of weak solutions}
\label{Section:Existence}

In this section, we present the proof of Theorem~\ref{NSCH:Theorem:ExistenceWeakSolutions}, which is based on an implicit time-discretization scheme. However, as discussed in the introduction, this approach cannot be applied directly for $K = 0$ because it requires an approximation of the initial data with higher regularity. Indeed, the core part of the proof requires $\mathcal{H}^2$-regular initial data, whereas we only assume that $(\phi_0,\psi_0)\in\mathcal{H}^1_{K,\alpha}$ and that \eqref{NSCH:cond:init} holds.

To overcome this difficulty, we employ Lemma~\ref{NSCH:Lemma:ApproxIC}, which provides a sequence of admissible initial data $\{(\phi_0^N,\psi_0^N)\}_{N\in\N}\subset\mathcal{H}^2$ such that $(\phi_0^N,\psi_0^N)\rightarrow(\phi_0,\psi_0)$ strongly in $\mathcal{H}^1$ as $N\rightarrow\infty$, and satisfying \eqref{NSCH:cond:init}. In particular, it holds
\begin{align}\label{nsch:remark}
	\abs{\phi_0^N} \leq 1 \quad\text{a.e.~in~}\Om \quad\text{and} \quad\abs{\psi_0^N} \leq 1 \quad\text{a.e.~on~}\Ga.
\end{align}
Nevertheless, in the case $K = 0$, i.e., when $\phi_0 = \alpha\psi_0$ a.e.~on $\Ga$, this relation is in general not preserved by this approximation, so that $\phi_0^N = \alpha\psi_0^N$ a.e.~on $\Ga$ cannot be guaranteed. 

A natural alternative would be to consider a coupled bulk-surface system instead of two uncoupled heat equations as in the proof of Lemma~\ref{NSCH:Lemma:ApproxIC}. For example, one could consider the system
\begin{alignat*}{2}
	\delt u - \Lap u &= 0 &&\qquad\text{in~}Q_T, \\
	\delt v - \Lapg v + \alpha\deln u &= 0 &&\qquad\text{on~}\Sigma_T, \\
	u &= \alpha v &&\qquad\text{on~}\Sigma_T, \\
	(u,v)\vert_{t=0} &= (\phi_0,\psi_0) &&\qquad\text{in~}\Om\times\Gamma.
\end{alignat*}
While this preserves the coupling condition, it only yields the estimates
\begin{align*}
	\abs{u} \leq 1 \quad\text{a.e.~in~}Q_T \quad\text{and}\quad \abs{\alpha}\abs{v} \leq 1 \quad\text{a.e.~on~}\Sigma_T.
\end{align*}
Since $\abs{\alpha} \leq 1$, the latter bound is insufficient to deduce $\abs{v} \leq 1$ a.e.~on $\Sigma_T$, which is essential for the analysis. We remark that such a system was previously considered in \cite{Gal2019} to approximate the initial data. However, in that work, the authors first studied the case of regular potentials, and therefore did not require the bound $\abs{v} \leq 1$ a.e.~on $\Ga$ at that stage of the argument.

For this reason, the proof proceeds in two steps. First, we establish the existence of weak solutions in the case $K\in(0,\infty]$ and $L\in[0,\infty]$ by means of an implicit time-discretization scheme. In a second step, we treat the case $K = 0$ by passing to the asymptotic limit $K\rightarrow 0$, following a similar strategy used in \cite{Knopf2025a}, where the authors considered the asymptotic limit $L\rightarrow 0$ to show the existence of a weak solution to \eqref{System:NSCH} for constant mobilities in the case $K\in[0,\infty]$ and $L = 0$. We also refer to \cite{Colli2019a} and \cite{Knopf2020}, where the asymptotic limit $K\rightarrow 0$ has been used to establish the existence of weak solutions to the limiting model for an Allen--Cahn equation with Allen--Cahn-type dynamic boundary condition and for a Cahn--Hilliard equation with Cahn--Hilliard-type dynamic boundary condition, respectively.

\subsection{Existence of weak solutions for \texorpdfstring{$K\in(0,\infty]$}{K>0}}

In this section, we prove the following theorem, which provides the existence of a weak solution to \eqref{System:NSCH} in the case $K\in(0,\infty]$ and $L\in[0,\infty]$.

\begin{theorem}\label{NSCH:Theorem:ExistenceWeakSolutions:K>0}
Suppose that the assumptions \ref{Assumption:NSCH:1}-\ref{Assumption:NSCH:Potentials} hold.
Let $K\in(0,\infty]$, $L\in[0,\infty]$, $(\bv_0,\bw_0)\in\mathbfcal{L}^2_\Div$, and let $(\phi_0,\psi_0)\in\mathcal{H}^1_{K,\alpha}$ satisfy \eqref{NSCH:cond:init}. If $L = 0$, we additionally assume that
    \begin{align}\label{Dens:Cond}
    	\beta(\tilde\sigma_2 - \tilde\sigma_1) = -(\tilde\rho_2 - \tilde\rho_1).
    \end{align}
    Then, there exists at least one global-in-time weak solution $(\bv,\bw,\phi,\psi,\mu,\theta)$ to system \eqref{System:NSCH} in the sense of Definition~\ref{NSCH:Definition:WeakSolution}. Additionally, the following energy inequalities hold: for any $t\in[s,\infty)$ and a.e. any $s\in[0,\infty)$, with $s = 0$ included, it holds
    \begin{align}
        &E_\mathrm{kin}(\bv(t),\bw(t),\phi(t),\psi(t)) \nonumber \\
        &\qquad + \int_s^t\intO 2\nu_\Om(\phi)\abs{\D\bv}^2\dxtau + \int_s^t\intG 2\nu_\Ga(\psi)\abs{\Dg\bw}^2\dGtau + \int_s^t\intG \gamma(\phi,\psi)\abs{\bw}^2\dGtau \nonumber \\
        &\quad\leq E_\mathrm{kin}(\bv(s),\bw(s),\phi(s),\psi(s)) - \int_s^t\intO \phi\bv\cdot\Grad\mu\dxtau - \int_s^t\intG \psi\bw\cdot\Gradg\theta\dGtau \label{EnergyInequality:Kinetic:K>0}
    \end{align}
    as well as
    \begin{align}
        &E_\free(\phi(t),\psi(t)) \nonumber \\
        &\qquad + \int_s^t\intO m_\Om(\phi)\abs{\Grad\mu}^2\dxtau + \int_s^t\intG m_\Ga(\psi)\abs{\Gradg\theta}^2\dGtau + \chi(L)\int_s^t\intG(\beta\theta-\mu)^2\dGtau \nonumber \\
        &\quad\leq E_\free(\phi(s),\psi(s)) + \int_s^t\intO \phi\bv\cdot\Grad\mu\dxtau + \int_s^t\intG \psi\bw\cdot\Gradg\theta\dGtau. \label{EnergyInequality:Free:K>0}
    \end{align}
\end{theorem}

Before presenting the proof of Theorem~\ref{NSCH:Theorem:ExistenceWeakSolutions:K>0}, we set up a time-discretization scheme and show the existence of approximate solutions to the time-discrete system.

\noindent
\textbf{An implicit time-discretization scheme}

We fix $N\in\N$ and denote the time-step size by $h = \frac1N$. For $k\in\N_0$, let $(\bv_k,\bw_k)\in\mathbfcal{L}^2_\Div$ and $(\phi_k,\psi_k)\in\mathcal{H}^1$ with $(F_0^\prime(\phi_k),G_0^\prime(\psi_k))\in\mathcal{L}^2$ be given, and set $(\rho_k,\sigma_k) = (\rho(\phi_k),\sigma(\psi_k))$. Then, we construct 
\begin{align*}
	(\bv,\bw,\phi,\psi,\mu,\theta) = (\bv_{k+1},\bw_{k+1},\phi_{k+1},\psi_{k+1},\mu_{k+1},\theta_{k+1})
\end{align*}
as a solution of the nonlinear system \eqref{NSCH:Discr:System}, where
\begin{align*}
    \J &= \J_{k+1} \coloneqq -\frac{\tilde\rho_2-\tilde\rho_1}{2}m_\Om(\phi_k)\Grad\mu_{k+1} = -\frac{\tilde\rho_2-\tilde\rho_1}{2}m_\Om(\phi_k)\Grad\mu, \\
    \K &= \K_{k+1} \coloneqq -\frac{\tilde\sigma_2-\tilde\sigma_1}{2}m_\Ga(\psi_k)\Gradg\theta_{k+1} = -\frac{\tilde\sigma_2-\tilde\sigma_1}{2}m_\Ga(\psi_k)\Gradg\theta, \\
    \J_\Ga &= \J_{\Ga,k+1} \coloneqq -\beta\frac{\tilde\sigma_2-\tilde\sigma_1}{2}m_\Om(\phi_k)\Grad\mu_{k+1} = -\beta\frac{\tilde\sigma_2-\tilde\sigma_1}{2}m_\Om(\phi_k)\Grad\mu.
\end{align*}
Our aim is to find $(\bv,\bw,\phi,\psi,\mu,\theta)$ with $(\bv,\bw)\in\mathbfcal{H}^1_{0,\Div}$, $(\phi,\psi)\in\mathrm{dom}\,(\partial\widetilde E_\free)$ and $(\mu,\theta)\in\mathcal{H}^2_{L,\beta}$ such that
\begin{subequations}\label{NSCH:Discr:System}
    \begin{align}\label{NSCH:VW:Discr}
        &\Big(\frac{\rho\bv - \rho_k\bv_k}{h},\wv\Big)_{\mathbf{L}^2(\Om)} + \Big(\frac{\sigma\bw - \sigma_k\bw_k}{h},\ww\Big)_{\mathbf{L}^2(\Ga)} \nonumber \\
        &\qquad + \big(\Div(\rho_k\bv\otimes\bv),\wv\big)_{\mathbf{L}^2(\Om)} + \big(\Divg(\sigma_k\bw\otimes\bw),\ww\big)_{\mathbf{L}^2(\Ga)}  \nonumber \\
        &\qquad + \left(2\nu_\Om(\phi_k)\D\bv,\D\wv\right)_{\mathbf{L}^2(\Om)} + \big(2\nu_\Ga(\psi_k)\Dg\bw,\Dg\ww\big)_{\mathbf{L}^2(\Ga)} + \big(\gamma(\phi_k,\psi_k)\bw,\ww\big)_{\mathbf{L}^2(\Ga)}\\
        &\qquad + \big(\Div(\bv\otimes\J),\wv\big)_{\mathbf{L}^2(\Om)}+ \big(\Divg(\bw\otimes\K),\ww\big)_{\mathbf{L}^2(\Ga)} \nonumber \\
        &\quad = \big(\mu\Grad\phi,\wv\big)_{\mathbf{L}^2(\Om)} + \big(\theta\Gradg\psi,\ww\big)_{\mathbf{L}^2(\Ga)} + \frac12\big((\J\cdot\n)\bw,\ww\big)_{\mathbf{L}^2(\Ga)} + \frac12\big((\J_\Ga\cdot\n)\bw,\ww\big)_{\mathbf{L}^2(\Ga)} \nonumber
    \end{align}
    for all $(\wv,\ww)\in\mathbfcal{H}_{0,\Div}^2$, as well as
    \begin{alignat}{2}
        &\frac{\phi - \phi_k}{h} + \bv\cdot\Grad\phi_k = \Div(m_\Om(\phi_k)\Grad\mu) &&\qquad\text{a.e.~in~}\Om, \label{NSCH:Phi:Discr}\\
        &\mu + \frac{c_F}{2}(\phi + \phi_k) = - \Lap\phi + F_0^\prime(\phi) &&\qquad\text{a.e.~in~}\Om, \label{NSCH:Mu:Discr}\\
        &\frac{\psi - \psi_k}{h} + \bw\cdot\Gradg\psi_k = \Divg(m_\Ga(\psi_k)\Gradg\theta) - \beta m_\Om(\phi_k)\deln\mu &&\qquad\text{a.e.~on~}\Ga, \label{NSCH:Psi:Discr}\\
        &\theta + \frac{c_G}{2}(\psi + \psi_k) = -\Lapg\psi + G_0^\prime(\psi) + \alpha\deln\phi &&\qquad\text{a.e.~on~}\Ga, \label{NSCH:Theta:Discr} \\
        &K\deln\phi = \alpha\psi - \phi, \qquad\qquad\;\,\, K\in(0,\infty], &&\qquad\text{a.e.~on~}\Ga, \label{NSCH:Phi:Discr:Bc} \\
        &Lm_\Om(\phi_k)\deln\phi = \beta\theta - \mu , \qquad L\in[0,\infty], &&\qquad\text{a.e.~on~}\Ga. \label{NSCH:Mu:Discr:Bc}
    \end{alignat}
\end{subequations}
Here, and in the following, we use for brevity the notation
\begin{align*}
	\rho \coloneqq \rho(\phi), \qquad \sigma \coloneqq \sigma(\psi),
\end{align*}
and the space
\begin{align*}
	\mathcal{H}^2_{L,\beta} = \begin{cases}
		\{(\zeta,\zeta_\Ga)\in\mathcal{H}^2: \zeta = \beta\zeta_\Ga \, \text{on~}\Ga\} &\text{if~}L = 0, \\
		\{(\zeta,\zeta_\Ga)\in\mathcal{H}^2: Lm_\Om(\phi_k)\deln\zeta = \beta\zeta_\Ga - \zeta \, \text{on~}\Ga\} &\text{if~}L\in(0,\infty), \\
		\{(\zeta,\zeta_\Ga)\in\mathcal{H}^2: \deln\zeta = 0 \, \text{on~}\Ga\} &\text{if~} L = \infty.
	\end{cases}
\end{align*}

\begin{remark}
    Multiplying \eqref{NSCH:Phi:Discr} with $-\frac{\tilde{\rho}_2 - \tilde{\rho}_1}{2}$ and \eqref{NSCH:Psi:Discr} with $-\frac{\tilde{\sigma}_2 - \tilde{\sigma}_1}{2}$ leads to
    \begin{subequations}
    \begin{alignat*}{2}
        -\frac{\rho - \rho_k}{h} - \bv\cdot\Grad\rho_k &= \Div\,\J &&\qquad\text{a.e.~in~}\Om, \\
        -\frac{\sigma - \sigma_k}{h} - \bw\cdot\Gradg\sigma_k &= \Divg\,\K - \J_\Ga\cdot\n &&\qquad\text{a.e.~on~}\Ga.
    \end{alignat*}
    \end{subequations}
    Using the identities $\Div(\bv\otimes\J) = (\Div\,\J)\bv + (\J\cdot\Grad)\bv$ and $\Divg(\bw\otimes\K) = (\Divg\,\K)\bw + (\K\cdot\Gradg)\bw$, we obtain the following equivalent version of \eqref{NSCH:VW:Discr}:
    \begin{align}\label{NSCH:VW:Discr:Equiv}
        &\Big(\frac{\rho\bv - \rho_k\bv_k}{h},\wv\Big)_{\mathbf{L}^2(\Om)} + \Big(\frac{\sigma\bw - \sigma_k\bw_k}{h},\ww\Big)_{\mathbf{L}^2(\Ga)} \nonumber \\
        &\qquad + \big(\Div(\rho_k\bv\otimes\bv),\wv\big)_{\mathbf{L}^2(\Om)} + \big(\Divg(\sigma_k\bw\otimes\bw),\ww\big)_{\mathbf{L}^2(\Ga)} \nonumber \\
        &\qquad + \big(2\nu_\Om(\phi_k)\D\bv,\D\wv\big)_{\mathbf{L}^2(\Om)} + \big(2\nu_\Ga(\psi_k)\Dg\bw,\Dg\ww\big)_{\mathbf{L}^2(\Ga)} + \big(\gamma(\phi_k,\psi_k)\bw,\ww\big)_{\mathbf{L}^2(\Ga)} \nonumber \\
        &\qquad + \frac12\bigg(\Big(\Div\,\J - \frac{\rho-\rho_k}{h} - \bv\cdot\Grad\rho_k\Big)\bv,\wv\bigg)_{\mathbf{L}^2(\Om)} \\
        &\qquad + \frac12\bigg(\Big(\Divg\,\K - \frac{\sigma-\sigma_k}{h} - \bw\cdot\Gradg\sigma_k\big)\bw,\ww\bigg)_{\mathbf{L}^2(\Ga)} \nonumber \\
        &\quad = \big(\mu\Grad\phi,\wv\big)_{\mathbf{L}^2(\Om)} + \big(\theta\Gradg\psi,\ww\big)_{\mathbf{L}^2(\Ga)} + \frac12\big((\J\cdot\n)\bw,\ww\big)_{\mathbf{L}^2(\Ga)} \nonumber 
    \end{align}
    for all $(\wv,\ww)\in\mathbfcal{H}^2_{0,\Div}$.
\end{remark}

\begin{remark}
    Integrating \eqref{NSCH:Phi:Discr} and \eqref{NSCH:Psi:Discr} with respect to the spatial variable, using the fact that $(\bv,\bw)\in\mathbfcal{H}^1_{0,\Div}$ and \eqref{NSCH:Mu:Discr:Bc}, we find that
    \begin{align*}
        \begin{dcases}
            \beta\intO \phi\dx + \intG \psi\dG = \beta\intO \phi_k \dx + \intG \psi_k\dG &\textnormal{if } L\in[0,\infty), \\
            \intO\phi\dx = \intO\phi_k\dx \quad\textnormal{and}\quad \intG\psi\dG = \intG\psi_k\dG &\textnormal{if } L = \infty
        \end{dcases}
    \end{align*}
    for all $k\in\N$. In particular, this observation implies that
    \begin{align*}
        \begin{dcases}
            \beta\intO \phi_k\dx + \intG \psi_k\dG = \beta\intO \phi_0 \dx + \intG \psi_0\dG &\textnormal{if~} L\in[0,\infty), \\
            \intO\phi_k\dx = \intO\phi_0\dx \quad\textnormal{and}\quad \intG\psi_k\dG = \intG\psi_0\dG &\textnormal{if~} L = \infty
        \end{dcases}
    \end{align*}
    for all $k\in\N_0$. As a result, the mass conservation property is preserved at the discrete level.
\end{remark}

\begin{lemma}\label{NSCH:Lemma:H^2+Subdiff}
    Let $(\phi_k,\psi_k)\in\mathcal{H}^2$ with 
    \begin{align}\label{NSCH:Lemma:H^2+Subdiff:<1}
    	\abs{\phi_k}\leq 1\quad\text{a.e.~in~}\Om \quad\text{and}\quad \abs{\psi_k} \leq 1 \quad\text{a.e.~on~}\Ga,
    \end{align}
    and further assume 
    \begin{subequations}
        \begin{align*}
            \beta\mean{\phi_k}{\psi_k}, \,\mean{\phi_k}{\psi_k}\in(-1,1) \qquad\text{if~}L\in[0,\infty),
        \end{align*}
        and
        \begin{align*}
            \meano{\phi_k}, \, \meang{\psi_k}\in(-1,1) \qquad\text{if~} L = \infty.
        \end{align*}
    \end{subequations}
    Let $(\phi,\psi,\mu,\theta)\in\mathrm{dom}\,(\partial \widetilde E_\free)\times\mathcal{H}^1$ be a quadruple solving $\big(\eqref{NSCH:Mu:Discr},\eqref{NSCH:Theta:Discr},\eqref{NSCH:Phi:Discr:Bc}\big)$ with
    \begin{align}\label{NSCH:MCL:Discrete}
        \begin{dcases}
            \mean{\phi}{\psi} = \mean{\phi_k}{\psi_k} &\textnormal{if~} L\in[0,\infty), \\
            \meano{\phi} = \meano{\phi_k} \quad\textnormal{and}\quad \meang{\psi} = \meang{\psi_k} &\textnormal{if~} L = \infty.
        \end{dcases}
    \end{align}    
	Then, there exists a constant $C_p = C(p,\mean{\phi_k}{\psi_k}) > 0$ such that
    \begin{align}
        \norm{(\phi,\psi)}_{\mathcal{W}^{2,p}} + \norm{(F_0^\prime(\phi),G_0^\prime(\psi))}_{\mathcal{L}^p} + \abs{\mean{\mu}{\theta}} &\leq C_p\big( 1 + \norm{(\mu,\theta)}_{L,\beta} \big), \label{NSCH:Lemma:Est:1}\\
        \norm{\partial \widetilde E_\free(\phi,\psi)}_{\mathcal{L}^p} &\leq C_p\big( 1 + \norm{(\mu,\theta)}_{\mathcal{L}^p}\big) \label{NSCH:Lemma:Est:2}
    \end{align}
    for $p = 6$ if $d = 3$ and any $p\in[2,\infty)$ if $d = 2$.
\end{lemma}

\begin{proof}
	Defining $(f,f_\Ga) \coloneqq (\mu + \frac{c_F}{2}(\phi + \phi_k), \theta + \frac{c_G}{2}(\phi + \psi_k)$, we notice that $(\phi,\psi)$ is a solution to
	\begin{alignat*}{2}
		-\Lap\phi + F_0^\prime(\phi) &= f &&\qquad\text{in~}\Om, \\
		-\Lapg\psi + G_0^\prime(\psi) + \alpha\deln\phi &= f_\Ga &&\qquad\text{on~}\Ga, \\
		K\deln\phi &= \alpha\psi - \phi &&\qquad\text{on~}\Ga.
	\end{alignat*}
	Thus, since $(f,f_\Ga)\in\mathcal{L}^p$, an application of regularity theory for systems with bulk-surface coupling and logarithmic nonlinearity from \cite[Proposition~5.1]{Giorgini2026} readily yields
	\begin{align*}
		\begin{split}
			\norm{(\phi,\psi)}_{\mathcal{W}^{2,p}} + \norm{(F_0^\prime(\phi),G_0^\prime(\psi))}_{\mathcal{L}^p} &\leq C_p\big(1 + \norm{(f,f_\Ga)}_{\mathcal{L}^p}\big) \\
			&\leq C_p\big(1 + \norm{(\mu,\theta)}_{\mathcal{L}^p}\big).
		\end{split}
	\end{align*}
	Here, we additionally exploited the assumption \eqref{NSCH:Lemma:H^2+Subdiff:<1} as well as $\abs{\phi}\leq 1$ a.e.~in $\Om$ and $\abs{\psi} \leq 1$ a.e.~on $\Ga$, which follows from $(\phi,\psi)\in\mathrm{dom}\,(\partial\widetilde E_\free)$. Next, an argument based on the Miranville--Zelik inequality (see \cite[Appendix A.1]{Miranville2004} or \cite[p.~908]{Gilardi2009}) gives
	\begin{align*}
		\abs{\mean{\mu}{\theta}} \leq C\big(1 + \norm{(\mu,\theta)}_{L,\beta}\big).
	\end{align*}
    As the necessary estimates are standard, we omit the details and refer, for instance, to \cite[Proof of Theorem~3.4, Step~4]{Knopf2025}. Thus, an application of the bulk-surface Poincar\'{e} inequality entails
	\begin{align*}
		\norm{(\phi,\psi)}_{\mathcal{W}^{2,p}} + \norm{(F_0^\prime(\phi),G_0^\prime(\psi))}_{\mathcal{L}^p} \leq C_p\big(1 + \norm{(\mu,\theta)}_{L,\beta}\big).
	\end{align*}
	Finally, the estimate \eqref{NSCH:Lemma:Est:2} on the subdifferential follows immediately in view of the identity
	\begin{align*}
		\del 	\widetilde E_\free(\phi,\psi) = \Big(\mu + \frac{c_F}{2}(\phi + \phi_k), \theta + \frac{c_G}{2}(\psi + \psi_k)\Big),
	\end{align*}
	and the above mentioned estimates on $(\phi,\psi)$ and $(\phi_k,\psi_k)$.
\end{proof}

\noindent
\textbf{Existence of solutions of the time-discrete scheme}

Next, we prove that solutions to the discrete system \eqref{NSCH:VW:Discr}-\eqref{NSCH:Mu:Discr:Bc}, if they exist, satisfy a discrete energy inequality that will provide the necessary a-priori bounds.

\begin{lemma}\label{NSCH:Lemma:DiscrEnergyInequality}
    Let $(\bv_k,\bw_k)\in\mathbfcal{L}^2_\Div$ and $(\phi_k,\psi_k)\in\mathcal{H}^2$ be given, and set $\rho_k = \rho(\phi_k)$, $\sigma_k = \sigma(\psi_k)$. Then, a solution $(\bv,\bw,\phi,\psi,\mu,\theta)\in\mathbfcal{H}^1_{0,\Div}\times\mathrm{dom}\,(\partial \widetilde E_\free)\times\mathcal{H}^2_{L,\beta}$ to the time-discrete system \eqref{NSCH:VW:Discr}-\eqref{NSCH:Mu:Discr:Bc} satisfies the following energy inequality
    \begin{align}\label{NSCH:EnergyEstimate:Discr}
        &E_{\mathrm{tot}}(\bv,\bw,\phi,\psi) + \intO \frac12\rho_k\abs{\bv - \bv_k}^2\dx + \intG \frac12\sigma_k\abs{\bw - \bw_k}^2\dG + \intO \frac12\abs{\Grad\phi - \Grad\phi_k}^2\dx \nonumber \\
        &\qquad + \intG \frac12\abs{\Gradg\psi - \Gradg\psi_k}^2\dG + \chi(K)\intG \frac12\big(\alpha(\psi - \psi_k) - (\phi - \phi_k)\big)^2\dG \\
        &\qquad + h\intO 2\nu_\Om(\phi_k)\abs{\D\bv}^2\dx + h\intG 2\nu_\Ga(\psi_k)\abs{\Dg\bw}^2\dG + h\intG\gamma(\phi_k,\psi_k)\abs{\bw}^2\dG \nonumber \\
        &\qquad + h\intO m_\Om(\phi_k)\abs{\Grad\mu}^2\dx + h\intG m_\Ga(\psi_k)\abs{\Gradg\theta}^2\dG + h\chi(L)\intG (\beta\theta - \mu)^2\dG \nonumber \\
        &\leq E_{\mathrm{tot}}(\bv_k,\bw_k,\phi_k,\psi_k). \nonumber
    \end{align}
\end{lemma}

\begin{proof}
    The idea of this proof is to use suitable test functions for the time-discrete system.  However, before testing \eqref{NSCH:VW:Discr:Equiv} with $(\bv,\bw)$, we recall the identities
    \begin{align*}
        &\intO \Big(\Div(\rho_k\bv\otimes\bv) - \frac12(\bv\cdot\Grad\rho_k)\bv\Big)\cdot\bv\dx = \intO \frac12\Div\big(\rho_k\bv\abs{\bv}^2\big)\dx = 0, \\
        &\intO\Big(\frac12(\Div\,\J)\bv + (\J\cdot\Grad)\bv\Big)\cdot\bv \dx = \intO \frac12\Div\big(\J\abs{\bv}^2\big)\dx = \frac12\intG (\J\cdot\n)\abs{\bv}^2\dG,
    \end{align*}
    which have been established in \cite[Lemma 4.3]{Abels2013}. Similarly, we derive
    \begin{align*}
        &\intG \Big(\Divg(\sigma\bw\otimes\bw) - \frac12(\bv\cdot\Gradg\sigma_k)\bw\Big)\cdot\bv\dG = \intG \frac12\Divg\big(\sigma_k\bw\abs{\bw}^2\big)\dG = 0, \\
        &\intG \Big(\frac12(\Divg\,\K)\bw + (\K\cdot\Gradg)\bw\Big)\cdot\bw\dG = \intG \frac12\Divg\big(\K\abs{\bw}^2\big)\dG = 0.
    \end{align*}
    Furthermore, the algebraic identity
    \begin{align*}
        \mathbf{a}\cdot(\mathbf{a}-\mathbf{b}) = \frac{\abs{\mathbf{a}}^2}{2} - \frac{\abs{\mathbf{b}}^2}{2} + \frac{\abs{\mathbf{a} - \mathbf{b}}^2}{2} \qquad\text{for~}\mathbf{a},\mathbf{b}\in\R^d
    \end{align*}
    yields that
    \begin{align*}
        \frac1h(\rho\bv - \rho_k\bv_k)\cdot\bv &= \frac{1}{2h}\Big(\rho\abs{\bv}^2 - \rho_k\abs{\bv_k}^2\Big) + \frac{1}{2h} (\rho - \rho_k)\abs{\bv}^2 + \frac{1}{2h}\rho_k\abs{\bv - \bv_k}^2, \\
        \frac1h(\sigma\bv - \sigma_k\bw_k)\cdot\bw &= \frac{1}{2h}\Big(\rho\abs{\bw}^2 - \sigma_k\abs{\bw_k}^2\Big) + \frac{1}{2h}(\sigma - \sigma_k)\abs{\bw}^2 + \frac{1}{2h} \sigma_k\abs{\bw - \bw_k}^2.
    \end{align*}
    Now, testing \eqref{NSCH:VW:Discr:Equiv} with $(\bv,\bw)$ and using the equations above, leads us to 
    \begin{align*}
        0 &= \intO \frac{1}{2h}\Big(\rho{\abs{\bv}^2 - \rho_k\abs{\bv_k}}^2\Big)\dx + \intG \frac{1}{2h}\Big(\sigma\abs{\bw}^2 - \sigma_k\abs{\bw_k}^2\Big)\dG + \intO \frac{1}{2h} \rho_k\abs{\bv - \bv_k}^2\dx \nonumber \\
        &\quad + \intG \frac{1}{2h}\sigma_k\abs{\bw - \bw_k}^2\dG + \intO 2\nu_\Om(\phi_k)\abs{\D\bv}^2\dx + \intG 2\nu_\Ga(\psi_k)\abs{\Dg\bw}^2\dG \\
        &\quad + \intG \gamma(\phi_k,\psi_k)\abs{\bw}^2\dG - \intO \mu\Grad\phi_k\cdot\bv \dx - \intG \theta\Gradg\psi_k\cdot\bw \dG. \nonumber
    \end{align*}
    Next, taking $\mu$ and $\theta$ to test with \eqref{NSCH:Phi:Discr} and \eqref{NSCH:Psi:Discr}, respectively, further gives
    \begin{align*}
        0 &= \intO \frac{1}{h}\big(\phi - \phi_k\big)\mu\dx + \intG \frac{1}{h}\big(\psi - \psi_k\big)\theta\dG + \intO \bv\cdot\Grad\phi_k\mu\dx + \intG \bw\cdot\Gradg\psi_k\theta\dG \nonumber \\
        &\quad + \intO m_\Om(\phi_k)\abs{\Grad\mu}^2\dx + \intG m_\Ga(\psi_k)\abs{\Gradg\theta}^2 \dG + \chi(L)\intG (\beta\theta - \mu)^2\dG,
    \end{align*}
    while testing \eqref{NSCH:Mu:Discr} and \eqref{NSCH:Theta:Discr} with $\frac1h\big(\phi - \phi_k\big)$ and $\frac1h\big(\psi - \psi_k\big)$, respectively, results in
    \begin{align*}
        0 &= \frac1h\intO\Grad\phi\cdot(\Grad\phi - \Grad\phi_k)\dx + \frac1h\intG\Gradg\psi\cdot(\Gradg\psi - \Gradg\psi_k)\dG \nonumber \\
        &\quad + \frac1h\chi(K)\intG (\alpha\psi - \phi)\big(\alpha(\psi - \psi_k) - (\phi - \phi_k)\big)\dG + \frac1h\intO F_0^\prime(\phi)(\phi - \phi_k)\dx \nonumber \\
        &\quad + \frac1h\intG G_0^\prime(\psi)(\psi - \psi_k)\dG - \frac1h \intO \mu(\phi - \phi_k)\dx - \frac1h\intG \theta(\psi - \psi_k)\dG \\
        &\quad - \frac1h\intO\frac{c_F}{2}\big(\phi^2 - \phi_k^2\big)\dx - \frac1h\intG \frac{c_G}{2}\big(\psi^2 - \psi_k^2\big)\dG. \nonumber 
    \end{align*}
    Taking into account the convexity of $F_0$ and $G_0$ to estimate
    \begin{align*}
        F_0^\prime(\phi)(\phi - \phi_k) \geq F_0(\phi) - F_0(\phi_k), \qquad G_0^\prime(\psi)(\psi - \psi_k) \geq G_0(\psi) - G_0(\psi_k),
    \end{align*}
    we sum these identities and deduce
    \begin{align*}
        0 &\geq \frac1h\intO \frac12\Big(\rho{\abs{\bv}^2 - \rho_k\abs{\bv_k}}^2\Big)\dx + \frac1h\intG \frac12\Big(\sigma\abs{\bw}^2 - \sigma_k\abs{\bw_k}^2\Big)\dG \nonumber \\
        &\quad + \frac1h\intO \frac12 \rho_k\abs{\bv - \bv_k}^2\dx + \frac1h\intG \frac12\sigma_k\abs{\bw - \bw_k}^2\dG + \intO 2\nu_\Om(\phi_k)\abs{\D\bv}^2\dx \nonumber \\
        &\quad + \intG 2\nu_\Ga(\psi_k)\abs{\Dg\bw}^2\dG + \intG \gamma(\phi_k,\psi_k)\abs{\bw}^2\dG + \intO m_\Om(\phi_k)\abs{\Grad\mu}^2\dx \\
        &\quad + \intG m_\Ga(\psi_k)\abs{\Gradg\theta}^2\dG + \chi(L)\intG(\beta\theta - \mu)^2\dG + \frac1h\intO F_0(\phi) - F_0(\phi_k)\dx \nonumber \\
        &\quad - \frac1h\intO \frac{c_F}{2}\big(\phi^2 - \phi_k^2\big)\dx + \frac1h\intG G_0(\psi) - G_0(\psi_k)\dG - \frac1h\intG\frac{c_G}{2}\big(\psi^2 - \psi_k^2\big)\dG \nonumber \\
        &\quad + \frac1h \intO \frac12\abs{\Grad\phi - \Grad\phi_k}^2\dx + \frac1h\intG\frac12\abs{\Gradg\psi - \Gradg\psi_k}^2\dG \nonumber \\
        &\quad + \frac1h\chi(K)\intG\frac12\big(\alpha(\psi - \psi_k) - (\phi - \phi_k)\big)^2\dG + \frac1h\intO\frac12\abs{\Grad\phi}^2 - \frac12\abs{\Grad\phi_k}^2\dx \nonumber \\
        &\quad + \frac1h\intG\frac12\abs{\Gradg\psi}^2 - \frac12\abs{\Gradg\psi_k}^2\dG + \frac1h\chi(K)\intG\frac12(\alpha\psi - \phi)^2 - \frac12(\alpha\psi_k - \phi_k)^2\dG. \nonumber
    \end{align*}
    By definition of $E_{\mathrm{tot}}(\bv,\bw,\phi,\psi)$, we have therefore established the time-discrete energy inequality \eqref{NSCH:EnergyEstimate:Discr} as claimed.
\end{proof}

In the next step, we show that the time-discrete system \eqref{NSCH:VW:Discr}-\eqref{NSCH:Mu:Discr:Bc} admits at least one solution.

\begin{lemma}\label{NSCH:Lemma:Existence:Discrete}
    Let $(\bv_k,\bw_k)\in\mathbfcal{L}^2_\Div$ and $(\phi_k,\psi_k)\in\mathcal{H}^2_{K,\alpha}$ be given, and set $(\rho_k,\sigma_k) = (\rho(\phi_k,),\sigma(\psi_k))$. Then, there exists a solution $(\bv,\bw,\phi,\psi,\mu,\theta)\in\mathbfcal{H}^1_{0,\Div}\times\mathrm{dom}\,(\partial \widetilde E_\free)\times\mathcal{H}^2_{L,\beta}$ to the time-discrete system \eqref{NSCH:VW:Discr}-\eqref{NSCH:Mu:Discr:Bc}.
\end{lemma}

\begin{proof}
    The central concept of this proof is based on the application of the Leray--Schauder principle. To this end, we define two nonlinear operators $\mathcal{M}_k,\mathcal{F}_k:X_{K,L}\rightarrow Y$, where
    \begin{align*}
        X_{K,L} \coloneqq \mathbfcal{H}^1_{0,\Div}\times\mathrm{dom}\,(\partial \widetilde E_\free)\times\mathcal{H}^2_{L,\beta}, \qquad Y \coloneqq (\mathbfcal{H}^1_{0,\Div})^\prime\times\mathcal{L}^2\times\mathcal{L}^2.
    \end{align*}
    These are constructed in a way such that a sextuple $\mathbf{z} = (\bv,\bw,\phi,\psi,\mu,\theta)\in X_{K,L}$ solves the system \eqref{NSCH:VW:Discr}-\eqref{NSCH:Mu:Discr:Bc} if and only if the identity
    \begin{align*}
        \mathcal{M}_k(\mathbf{z}) = \mathcal{F}_k(\mathbf{z})
    \end{align*}
    holds. For $\mathbf{z} = (\bv,\bw,\phi,\psi,\mu,\theta)\in X_{K,L}$, we define
    \begin{align*}
        \mathcal{M}_k(\mathbf{z}) \coloneqq \begin{pmatrix}
            \mathcal{T}_k(\bv,\bw) \\
            \begin{pmatrix}
                \phi \\
                \psi
            \end{pmatrix}
            + \del \widetilde E_\free(\phi,\psi) \\
            \mathfrak{S}_{L,\beta}\begin{pmatrix}
                \mu \\
                \theta
            \end{pmatrix}
        \end{pmatrix},
    \end{align*}
	where
    \begin{align*}
        \big\langle\mathcal{T}_k(\bv,\bw),(\wv,\ww)\big\rangle &\coloneqq \intO 2\nu_\Om(\phi_k)\D\bv:\D\wv\dx + \intG 2\nu_\Ga(\psi_k)\Dg\bw:\Dg\ww\dG \\
        &\quad + \intG \gamma(\phi_k,\psi_k)\bw\cdot\ww\dG
    \end{align*}
    for all $(\bv,\bw), (\wv,\ww)\in\mathbfcal{H}^1_{0,\Div}$, and 
    \begin{align*}
        \mathfrak{S}_{L,\beta}\begin{pmatrix}
            \mu \\ \theta
        \end{pmatrix}
        \coloneqq
        \begin{pmatrix}
            -\Div(m_\Om(\phi_k)\Grad\mu) + \intO\mu\dx \\
            -\Divg(m_\Ga(\psi_k)\Gradg\theta) + \beta m_\Om(\phi_k)\deln\mu + \intG \theta\dG
        \end{pmatrix}
    \end{align*}
    for all $(\mu,\theta)\in\mathcal{H}^2_{L,\beta}$. Next, for any $\mathbf{z} = (\bv,\bw,\phi,\psi,\mu,\theta)\in X_{K,L}$, we define
    \begin{align*}
        \mathcal{F}_k(\mathbf{z}) \coloneqq \begin{pmatrix}
            \mathbfcal{S} \\
            \phi + \mu + \frac{c_F}{2}(\phi + \phi_k) \\
            \psi + \theta + \frac{c_G}{2}(\psi + \psi_k) \\
            -\frac{\phi - \phi_k}{h} - \bv\cdot\Grad\phi_k + \intO \mu\dx \\
            -\frac{\psi - \psi_k}{h} - \bw\cdot\Gradg\psi_k + \intG \theta\dG
        \end{pmatrix},
    \end{align*}
    where $\mathbfcal{S}$ in the first line of $\mathcal{F}_k$ is defined as
    \begin{align*}
    	\mathbfcal{S} &\coloneqq -\frac{\rho\bv-\rho_k\bv_k}{h} - \frac{\sigma\bw-\sigma_k\bw_k}{h} - \Div(\rho_k\bv\otimes\bv) - \Divg(\sigma_k\bw\otimes\bw) - (\J\cdot\Grad)\bv \\
    	&\quad - (\K\cdot\Gradg)\bw - \frac12\Big(\Div\,\J - \frac{\rho-\rho_k}{2} - \bv\cdot\Grad\rho_k\Big)\bv \\
    	&\quad - \frac12\Big(\Divg\,\K - \frac{\sigma-\sigma_k}{2} - \bw\cdot\Gradg\sigma_k\Big)\bw + \mu\Grad\phi + \theta\Gradg\psi + \frac12(\J\cdot\n)\bw.
    \end{align*}
    Using a similar argument as in \cite[Section~3.1]{Gal2019}, one can show that $\mathcal{T}_k:\mathbfcal{H}^1_{0,\Div}\rightarrow(\mathbfcal{H}^1_{0,\Div})^\prime$ is invertible with continuous inverse $\mathcal{T}_k^{-1}$ (see also \cite[Lemma~3.1]{Gal2024}). Moreover, from Lemma~\ref{NSCH:Proposition:Elliptic:MuTheta}, we infer that $\mathfrak{S}_{L,\beta}:\mathcal{H}_{L,\beta}^2\rightarrow\mathcal{L}^2$ is invertible with continuous inverse $\mathfrak{S}_{L,\beta}^{-1}$. In what follows, we examine the second line of the $\mathcal{M}_k$. Since $\partial\widetilde E_\free$ is a maximal monotone operator according to Proposition~\ref{App:Proposition:Subdiff}, the operator
    \begin{align*}
        I + \partial \widetilde E_\free:\mathrm{dom}\,(\partial \widetilde E_\free)\rightarrow\mathcal{L}^2
    \end{align*}
    is invertible. The inverse operator, considered as a mapping
    \begin{align*}
        (I + \partial \widetilde E_\free)^{-1}:\mathcal{L}^2\rightarrow\mathcal{H}^{2-s}, \qquad s\in(0,\tfrac14),
    \end{align*}
    is continuous and compact. Indeed, let $(f_l,f_{\Ga l})\rightarrow (f,f_\Ga)$ strongly in $\mathcal{L}^2$ as $l\rightarrow\infty$ with $(f_l,f_{\Ga l}) = (u_l,v_l) + \partial \widetilde E_\free(u_l,v_l)$ for $l\in\N$ and $(f,f_\Ga) = (u,v) + \partial \widetilde E_\free(u,v)$ be given. Then we have $(u_l,v_l)\rightarrow (u,v)$ strongly in $\mathcal{H}^1$ as $l\rightarrow\infty$ since
    \begin{align}\label{NSCH:u_lv_l:Conv:L2}
        &\norm{(u_l - u, v_l - v)}_{\mathcal{L}^2}^2 + \norm{(u_l - u, v_l - v)}_{K,\alpha}^2 \nonumber \\
        &\quad\leq \norm{(u_l - u, v_l - v)}_{\mathcal{L}^2}^2 + \bigscp{\partial \widetilde E_\free(u_l,v_l) - \partial \widetilde E_\free(u,v)}{(u,v) - (u_l,v_l)}_{\mathcal{L}^2} \nonumber \\
        &\quad\leq \norm{(u_l,v_l) + \partial \widetilde E_\free(u_l,v_l) - ((u,v) + \partial \widetilde E_\free(u,v))}_{\mathcal{L}^2}\norm{(u,v) - (u_l,v_l)}_{\mathcal{L}^2} \nonumber \\
        &\quad\leq \frac12\norm{(f_l,f_{\Ga l}) - (f,f_\Ga)}_{\mathcal{L}^2}^2 + \frac12\norm{(u_l,v_l) - (u,v)}_{\mathcal{L}^2}^2
    \end{align}
    for all $l\in\N$. Moreover, exploiting again the regularity theory for system with bulk-surface coupling and logarithmic nonlinearity from \cite[Proposition~5.1]{Giorgini2026}, we have
    \begin{align}\label{NSCH:u_lv_l:Est:H2}
    	\norm{(u_l,v_l)}_{\mathcal{H}^2} \leq C\big(1 + \norm{\del\widetilde E_\free(u_l,v_l)}_{\mathcal{L}^2}\big) \leq C\big(\norm{(u_l,v_l)}_{\mathcal{L}^2} + \norm{(f_l,f_{\Ga l})}_{\mathcal{L}^2}\big)
    \end{align}
    for all $l\in\N$, and due to the strong convergence of $\{(f_l,f_{\Ga l})\}_{l\in\N}$ in $\mathcal{L}^2$ and the estimate \eqref{NSCH:u_lv_l:Conv:L2}, we immediately obtain that $\{(u_l,v_l)\}_{l\in\N}$ is bounded in $\mathcal{H}^2$. Thus, exploiting the compact embedding $\mathcal{H}^2\emb\mathcal{H}^{2-s}$ yields that $(u_l,v_l)\rightarrow (u,v)$ strongly in $\mathcal{H}^{2-s}$ as $l\rightarrow\infty$ for any $s\in(0,\frac14)$. 
    
    Altogether we obtain that $\mathcal{M}_k:X_{K,L}\rightarrow Y$ is invertible with inverse $\mathcal{M}_k^{-1}:Y\rightarrow X_{K,L}$. Note that $X_{K,L}$ is not a Banach space since $\mathrm{dom}\,(\partial \widetilde E_\free)$ includes inequality constraints.
    To get a continuous and even compact operator, for any $s\in(0,\frac14)$, we introduce the Banach spaces
    \begin{align*}
        \widetilde{X}_{K,L} \coloneqq \mathbfcal{H}^1_{0,\Div}\times\mathcal{H}^{2-s}\times\mathcal{H}^2_{L,\beta}, \ s\in\big(0,\tfrac14\big), \qquad  \widetilde{Y} \coloneqq \mathbfcal{L}^{\frac32}\times\mathcal{H}^1\times\mathcal{W}^{1,\frac32}.
    \end{align*}
    In this way, we have the continuous embedding $X_{K,L}\emb\widetilde{X}$ and the compact embedding $\widetilde{Y}\emb Y$. Eventually, $\mathcal{M}_k^{-1}:\widetilde{Y}\rightarrow\widetilde{X}_{K,L}$ is a compact operator. 
    
    We proceed by showing that $\mathcal{F}_k:\widetilde{X}_{K,L}\rightarrow\widetilde{Y}$ is continuous and maps bounded sets into bounded sets. In the following, we denote with $C_k$ constants that may depend on the parameters of the system and on $k\in\N$. Then, it was shown in \cite[(i)-(vi), pp. 467]{Abels2013} that the following estimates
    \begin{alignat*}{2}
        &\norm{\rho\bv}_{\mathbf{L}^{\frac32}(\Om)}\leq C\norm{\bv}_{\mathbf{H}^1(\Om)}\big(1+\norm{\phi}_{L^2(\Om)}\big), &&\qquad \norm{\Div(\rho_k\bv\otimes\bv)}_{\mathbf{L}^{\frac32}(\Om)} \leq C_k\norm{\bv}_{\mathbf{H}^1(\Om)}^2, \\
        &\norm{\mu\Grad\phi_k}_{\mathbf{L}^{\frac32}(\Om)}\leq C_k\norm{\mu}_{L^2(\Om)}, &&\qquad\norm{(\Div\J)\bv}_{\mathbf{L}^{\frac32}(\Om)} \leq C_k\norm{\bv}_{\mathbf{H}^1(\Om)}\norm{\mu}_{H^2(\Om)}, \\
        &\norm{(\J\cdot\Grad)\bv}_{\mathbf{L}^{\frac32}(\Om)} \leq C\norm{\bv}_{\mathbf{H}^1(\Om)}\norm{\mu}_{H^2(\Om)}, &&\qquad\norm{\bv\cdot\Grad\phi_k}_{W^{1,\frac32}(\Om)}\leq C_k\norm{\bv}_{\mathbf{H}^1(\Om)},
    \end{alignat*}
    hold. As $\Ga$ is a $(d-1)$-dimensional submanifold on $\R^d$, we can repeat the argument presented therein and show that analogous estimates hold for the corresponding boundary terms. Lastly, by the trace theorem, we have
    \begin{align*}
    	\norm{\tfrac12(\J\cdot\n)\bw}_{\mathbf{L}^{\frac32}(\Ga)} \leq C\norm{\mu}_{H^2(\Om)}\norm{\bw}_{\mathbf{H}^1(\Ga)}.
    \end{align*}
We thus deduce that
\begin{align}\label{F_k:bounded}
	\norm{\mathcal{F}_k(\mathbf z)}_{\widetilde Y} \leq C_k\Big(1 + \norm{\mathbf z}_{\widetilde X_{K,L}}^2\Big)
\end{align}    
for all $\mathbf z\in\widetilde X_{K,L}$. In particular, $\mathcal{F}_k:\widetilde{X}_{K,L}\rightarrow\widetilde{Y}$ maps bounded sets into bounded sets.
Now, we aim to apply the Leray--Schauer principle (see, e.g., \cite[Theorem 6.A]{Zeidler1986}). To this end, we notice that finding a solution $\mathbf{z}\in X_{K,L}$ to the time-discrete problem \eqref{NSCH:VW:Discr}-\eqref{NSCH:Mu:Discr:Bc}, i.e., a solution to $\mathcal{M}_k(\mathbf{z}) = \mathcal{F}_k(\mathbf{z})$, is equivalent to solving the equation
\begin{align*}
    \mathcal{K}_k(\mathbf{f}) = \mathbf{f}, \qquad\text{where~}\mathcal{K}_k \coloneqq\mathcal{F}_k\circ\mathcal{M}_k^{-1}, \ \mathbf{f} = \mathcal{M}_k(\mathbf{z}).
\end{align*}
Hence, our goal is to find a fixed point of $\mathcal{K}_k$ in $\widetilde{Y}$. It follows from the considerations above that the mapping $\mathcal{K}_k:\widetilde{Y}\rightarrow\widetilde{Y}$ is compact. To apply the Leray--Schauder principle, it remains to show that the set
\begin{align}\label{NSCH:LeraySchauder}
\big\{\f\in\widetilde Y \;:\; \f = \lambda\mathcal{K}_k(\f) \ \text{for some~} 0 \leq \lambda \leq 1\big\}
\end{align}
is bounded. So, let us consider $\mathbf{f}\in\widetilde{Y}$ and $0\leq\lambda\leq 1$ satisfying $\mathbf{f} = \lambda\mathcal{K}_k(\mathbf{f})$. Defining $\mathbf{z} \coloneqq \mathcal{M}_k^{-1}(\mathbf{f})\in X_{K,L}$ leads to the equation
\begin{align}\label{NSCH:Eq:LeraySchauder}
    \mathcal{M}_k(\mathbf{z}) = \lambda\mathcal{F}_k(\mathbf{z}),
\end{align}
which is equivalent to the weak formulation
\begin{align}\label{NSCH:WF:vw:lambda}
    &\intO 2\nu_\Om(\phi_k)\D\bv:\D\wv\dx + \intG 2\nu_\Ga(\psi_k)\Dg\bw:\Dg\ww\dG + \intG \gamma(\phi_k,\psi_k)\bw\cdot\ww\dG \nonumber \\
    &\qquad + \lambda\intO \frac{\rho\bv - \rho_k\bv_k}{h}\cdot\wv\dx + \lambda\intG \frac{\sigma\bw - \sigma_k\bw_k}{h}\cdot\ww\dG + \lambda\intO \Div(\rho_k\bv\otimes\bv)\cdot\wv\dx \nonumber \\
    &\qquad + \lambda\intG \Divg(\sigma_k\bw\otimes\bw)\cdot\ww\dG + \lambda\intO (\J\cdot\Grad)\bv\cdot\wv\dx + \lambda\intG (\K\cdot\Gradg)\bw\cdot\ww\dG \nonumber \\
    &\qquad + \lambda\intO \frac12\Big(\Div\,\J - \frac{\rho - \rho_k}{2} - \bv\cdot\Grad\rho_k\Big)\bv\cdot\wv\dx \\
    &\qquad + \lambda\intG \frac12\Big(\Divg\,\K - \frac{\sigma - \sigma_k}{2} - \bw\cdot\Gradg\sigma_k\Big)\bw\cdot\ww\dG \nonumber \\
    &\quad = \lambda\intO \mu\Grad\phi_k\cdot\wv \dx + \lambda\intG \theta\Gradg\psi_k\cdot\ww\dG + \lambda\intG \frac12(\J\cdot\n)\bw\cdot\ww\dG \nonumber
\end{align}
for all $(\wv,\ww)\in\mathbfcal{H}^1_{0,\Div}$, and the pointwise identities
\begin{subequations}\label{NSCH:WF:pp:mt:lambda}
\begin{alignat}{2}
        &\lambda\phi + \lambda\mu + \lambda \frac{c_F}{2}\big(\phi+\phi_k\big) = \phi -\Lap\phi + F_0^\prime(\phi_0) &&\qquad\text{a.e.~in~}\Om, \label{NSCH:WF:phi:lambda}\\
        &\lambda\psi + \lambda\theta + \lambda \frac{c_G}{2}\big(\psi+\psi_k\big) = \psi - \Lapg\psi + G_0^\prime(\psi_0) + \alpha\deln\phi &&\qquad\text{a.e.~on~}\Ga, \label{NSCH:WF:psi:lambda}\\
        &-\Div(m_\Om(\phi_k)\Grad\mu) + \intO \mu\dx && \nonumber \\
    	&\quad = -\lambda\frac{\phi-\phi_k}{h} - \lambda\bv\cdot\Grad\phi_k + \lambda\intO\mu\dx &&\qquad\text{a.e.~in~}\Om, \label{NSCH:WF:mu:lambda} \\
    	&-\Divg(m_\Ga(\psi_k)\Gradg\theta) + \beta m_\Om(\phi_k)\deln\mu + \intG \theta\dG &&  \nonumber \\
    	&\quad = -\lambda\frac{\psi-\psi_k}{h} - \lambda\bw\cdot\Gradg\psi_k + \lambda\intG\theta\dG &&\qquad\text{a.e.~on~}\Ga, \label{NSCH:WF:theta:lambda} \\
        &K\deln\phi = \alpha\psi - \phi, \qquad\qquad\;\,\, K\in(0,\infty], &&\qquad\text{a.e.~on~}\Ga, \label{NSCH:WF:phi:bc:lambda} \\
        &Lm_\Om(\phi_k)\deln\mu = \beta\theta - \mu, \qquad L\in[0,\infty], &&\qquad\text{a.e.~on~}\Ga. \label{NSCH:WF:mu:bc:lambda}
\end{alignat}
\end{subequations}
Proceeding analogously to the proof of Lemma~\ref{NSCH:Lemma:DiscrEnergyInequality}, we now test \eqref{NSCH:WF:vw:lambda} with $(\bv,\bw)$, \eqref{NSCH:WF:phi:lambda} with $\frac1h(\phi-\phi_k)$, \eqref{NSCH:WF:psi:lambda} with $\frac1h(\psi-\psi_k)$, \eqref{NSCH:WF:mu:lambda} with $\mu$, and \eqref{NSCH:WF:theta:lambda} with $\theta$. This yields 
\begin{align*}
    &\lambda\intO \frac{1}{2h}\big(\rho\abs{\bv}^2 - \rho_k\abs{\bv_k}^2\big)\dx + \lambda\intG \frac{1}{2h}\big(\sigma\abs{\bw}^2 - \sigma_k\abs{\bw_k}^2\big)\dG + \lambda\intO \frac{1}{2h}\rho_k\abs{\bv-\bv_k}^2\dx \nonumber \\
    &\qquad + \lambda\intG\frac{1}{2h}\sigma_k\abs{\bw-\bw_k}^2\dG + \intO 2\nu_\Om(\phi_k)\abs{\D\bv}^2\dx + \intG 2\nu_\Ga(\psi_k)\abs{\Dg\bw}^2\dG \nonumber \\
    &\qquad + \intG \gamma(\phi_k,\psi_k)\abs{\bw}^2\dG + \intO m_\Om(\phi_k)\abs{\Grad\mu}^2\dx + \intG m_\Ga(\psi_k)\abs{\Gradg\theta}^2\dG \nonumber \\
    &\qquad + \chi(L)\intG(\beta\theta-\mu)^2\dG + (1-\lambda)\frac1h\intO\phi(\phi-\phi_k)\dx + (1-\lambda)\frac1h\intG\psi(\psi-\psi_k)\dG \nonumber \\
    &\qquad + \frac1h\intO \Grad\phi\cdot(\Grad\phi-\Grad\phi_k)\dx + \frac1h\intG\Gradg\psi\cdot(\Gradg\psi-\Gradg\psi_k)\dG \\
    &\qquad + \frac1h\chi(K)\intG (\alpha\psi - \phi)\big(\alpha(\psi-\psi_k) - (\phi-\phi_k)\big)\dG + \frac1h\intO F_0^\prime(\phi)(\phi-\phi_k)\dx \nonumber \\
    &\qquad + (1-\lambda)\left(\intO\mu\dx\right)^2 + \frac1h\intG G_0^\prime(\psi)(\psi-\psi_k)\dG + (1-\lambda)\left(\intG\theta\dG\right)^2 \nonumber \\
    &\quad = \lambda \intO\frac{c_F}{2h} \big(\phi^2-\phi_k^2\big)\dx + \lambda\intG \frac{c_G}{2h}\big(\psi^2-\psi_k^2\big)\dG. \nonumber
\end{align*}
Next, note that due to $\mathbf{z} = (\bv,\bw,\phi,\psi,\mu,\theta) = \mathcal{M}_k^{-1}(\mathbf{f})\in X_{K,L}$, it holds that $(\phi,\psi)\in\mathrm{dom}\,(\partial \widetilde E_\free)$, and therefore $\abs{\phi}\leq 1$ a.e.~in $\Om$ as well as $\abs{\psi}\leq 1$ a.e.~on $\Ga$. Consequently, it holds that $\rho\geq 0$ a.e.~in $\Om$ and $\sigma\geq 0 $ a.e.~on $\Ga$. Furthermore, we recall the convexity of $F_0$ and $G_0$, respectively, as well as the identity
\begin{align*}
	\mathbf{a}\cdot(\mathbf{a} - \mathbf{b}) = \frac12\abs{\mathbf{a}}^2 - \frac12\abs{\mathbf{b}}^2 + \frac12\abs{\mathbf{a}+\mathbf{b}}^2 \qquad\text{for all~}\mathbf{a},\mathbf{b}\in\R^d.
\end{align*} 
Then, dropping any non-essential non-negative term on the left-hand side, we deduce the following inequality
\begin{align}
    &h\intO 2\nu_\Om(\phi_k)\abs{\D\bv}^2\dx + h\intG 2\nu_\Ga(\psi_k)\abs{\Dg\bw}^2\dG + h\intG \gamma(\phi_k,\psi_k)\abs{\bw}^2\dG \nonumber \\
    &\qquad + h\intO m_\Om(\phi_k)\abs{\Grad\mu}^2\dx + h\intG m_\Ga(\psi_k)\abs{\Gradg\theta}^2\dG + h\chi(L)\intG (\beta\theta-\mu)^2\dG \nonumber \\
    &\qquad + \intO \frac12\abs{\Grad\phi}^2\dx + \intG \frac12\abs{\Gradg\psi}^2\dG + \chi(K)\intG \frac12 (\alpha\psi-\phi)^2\dG \nonumber \\
    &\qquad + (1-\lambda)h\left(\intO\mu\dx\right)^2 + (1-\lambda)h\left(\intG\theta\dG\right)^2 \nonumber \\
    &\quad\leq \intO \frac12 \rho_k\abs{\bv_k}^2\dx + \intG \frac12\sigma_k\abs{\bw_k}^2\dG + \intO \frac12\abs{\Grad\phi_k}^2\dx + \intG \frac12\abs{\Gradg\psi_k}^2\dG \nonumber \\
    &\qquad + \chi(K)\intG \frac12(\alpha\psi_k-\phi_k)^2\dG + \intO F_0(\phi_k)\dx + \intG G_0(\psi_k)\dG \nonumber \\
    &\qquad + \lambda \intO \frac{c_F}{2}\phi_k^2\dx + \lambda \intG \frac{c_G}{2}\psi_k^2\dG + \Bigabs{\intO F_0(\phi)\dx} + \Bigabs{\intG G_0(\psi)\dG} \nonumber \\
    &\quad\leq C_k, 
\end{align}
for a constant $C_k$ that may depend on $k$ and $h$ but is independent of $\lambda$, and  whose value may change between the lines.
Thus, recalling again that $\abs{\phi}, \abs{\phi_k} < 1$ a.e.~in $\Om$ and $\abs{\psi}, \abs{\psi_k} < 1$ a.e.~on $\Ga$, and estimating $\lambda$ by $1$, we employ the bulk-surface Poincar\'{e} inequality and obtain
\begin{align}\label{NSCH:Est:Uniform:Lambda:1}
	&\norm{(\bv,\bw)}_{\mathbfcal{H}^1} + \norm{(\mu,\theta)}_{L,\beta} + \norm{(\phi,\psi)}_{\mathcal{H}^1} + (1-\lambda)^{\frac12}\Bigabs{\intO\mu\dx} + (1-\lambda)^{\frac12}\Bigabs{\intG\theta\dG} \nonumber \\
    &\quad\leq C_k.
\end{align}
In a subsequent step, we intend to control the $\mathcal{L}^2$-norm of $(\mu,\theta)$, for which we consider two distinct cases. For $\lambda\in[0,\tfrac12)$, a direct use of \eqref{NSCH:Est:Uniform:Lambda:1} implies
\begin{align}\label{NSCH:Est:Mean:mt:lambda:1}
    \Bigabs{\intO\mu\dx} + \Bigabs{\intG\theta\dG} \leq C_k,
\end{align}
which enables us to get a control on $\norm{(\mu,\theta)}_{\mathcal{L}^2}$ in view of the bulk-surface Poincar\'{e} inequality. On the other hand, for $\lambda\in[\tfrac12,1]$, we proceed as in Lemma~\ref{NSCH:Lemma:H^2+Subdiff}, with \eqref{NSCH:Mu:Discr} and \eqref{NSCH:Theta:Discr} replaced by \eqref{NSCH:WF:mu:lambda} and \eqref{NSCH:WF:theta:lambda}, to get the estimate
\begin{align}\label{NSCH:Est:Mean:mt:lambda:2}
    \frac12\Bigabs{\intO\mu\dx} + \frac12\Bigabs{\intG\theta\dG} &\leq \lambda\Bigabs{\intO\mu\dx} + \lambda\Bigabs{\intG\theta\dG} \leq C_k.
\end{align}
Thus, employing once again the bulk-surface Poincar\'{e} inequality, we obtain from \eqref{NSCH:Est:Uniform:Lambda:1} and \eqref{NSCH:Est:Mean:mt:lambda:1}-\eqref{NSCH:Est:Mean:mt:lambda:2} the estimate
\begin{align}\label{NSCH:Est:mt:H^1:k}
	\norm{(\mu,\theta)}_{\mathcal{H}^1} \leq C_k.
\end{align}
Then, repeating the arguments from the proof of Lemma~\ref{NSCH:Lemma:H^2+Subdiff} and utilizing \eqref{NSCH:Est:Uniform:Lambda:1}, we find
\begin{align}\label{NSCH:Est:pp:H^2:k}
    \norm{(\phi,\psi)}_{\mathcal{H}^2} + \norm{\partial\widetilde E_\free(\phi,\psi)}_{\mathcal{L}^2} \leq C_k.
\end{align}
Moreover, incorporating Lemma~\ref{NSCH:Proposition:Elliptic:MuTheta}, where now
\begin{align*}
	(f,f_\Ga) = \Big(-\lambda\frac{\phi-\phi_k}{h} - \lambda\bv\cdot\Grad\phi_k + \lambda\intO\mu\dx, -\lambda\frac{\psi-\psi_k}{h} - \lambda\bw\cdot\Gradg\psi_k + \lambda\intG\theta\dG\Big),
\end{align*} 
implies in combination with \eqref{NSCH:Est:Uniform:Lambda:1} and \eqref{NSCH:Est:pp:H^2:k} that
\begin{align}\label{NSCH:Est:mt:H^2:k}
    \norm{(\mu,\theta)}_{\mathcal{H}^2} &\leq C\norm{(f,f_\Ga)}_{\mathcal{L}^2} 
    \leq C_k\Big(1 + \norm{(\bv,\bw)}_{\mathbfcal{L}^4}\norm{(\Grad\phi_k,\Gradg\psi_k)}_{\mathbfcal{L}^4}\Big) 
    \leq C_k. 
\end{align}
Ultimately, combining \eqref{NSCH:Est:Uniform:Lambda:1}, \eqref{NSCH:Est:pp:H^2:k} and \eqref{NSCH:Est:mt:H^2:k}, we obtain that $\mathbf{z} = (\bv,\bw,\phi,\psi,\mu,\theta)\in\widetilde X_{K,L} = \mathbfcal{H}^1_{0,\Div}\times\mathcal{H}^{2-s}\times\mathcal{H}^2_{L,\beta}$ satisfies the estimate
\begin{align*}
    \norm{\mathbf{z}}_{\widetilde X_{K,L}} \leq C_k,
\end{align*}
where
\begin{align*}
    \norm{\mathbf{z}}_{\widetilde X_{K,L}} = \norm{(\bv,\bw)}_{\mathbfcal{H}^1} + \norm{(\phi,\psi)}_{\mathcal{H}^{2-s}} + \norm{(\mu,\theta)}_{\mathcal{H}^2}.
\end{align*}
Eventually, since $\mathbf{f}\in\widetilde{Y}$ fulfills $\mathbf{f} = \lambda\mathcal{K}_k(\mathbf{f}) = \lambda\mathcal{F}_k(\mathbf{z})$, and $\mathcal{F}_k:\widetilde X_{K,L}\rightarrow\widetilde{Y}$ maps bounded sets into bounded sets by \eqref{F_k:bounded}, we conclude
\begin{align*}
    \norm{\mathbf{f}}_{\widetilde{Y}} = \norm{\lambda\mathcal{F}_k(\mathbf{z})}_{\widetilde{Y}} \leq C_k\big(1 + \norm{\mathbf{z}}_{\widetilde{X}_{K,L}}^2\big) \leq C_k,
\end{align*}
where the constant $C_k > 0$ is independent of $\lambda\in[0,1]$ and $\mathbf{f}$. The latter shows that the set defined in \eqref{NSCH:LeraySchauder} is bounded. Consequently, we are in a position to apply the Leray--Schauder fixed-point theorem, which entails the existence of a fixed point $\f\in\widetilde Y\emb Y$ of the operator $\mathcal{K}_k$. Thus, $\mathbf{z} = \mathcal{M}_k^{-1}(\mathbf{f})\in\widetilde X_{K,L}$ is a solution to $\mathcal{M}_k(\mathbf{z}) = \mathcal{F}_k(\mathbf{z})$. We have therefore shown, for any $k\in\N$, the existence of a solution $\mathbf{z} = (\bv,\bw,\phi,\psi,\mu,\theta) = (\bv_{k+1},\bw_{k+1},\phi_{k+1},\psi_{k+1},\mu_{k+1},\theta_{k+1})\in X_{K,L}$ to the time-discrete system \eqref{NSCH:VW:Discr}-\eqref{NSCH:Mu:Discr:Bc}, which completes the proof of Lemma~\ref{NSCH:Lemma:Existence:Discrete}.
\end{proof}

We are now in a position to present the proof of Theorem~\ref{NSCH:Theorem:ExistenceWeakSolutions:K>0}.

\begin{proof}[Proof of Theorem~\ref{NSCH:Theorem:ExistenceWeakSolutions:K>0}]
    Let $N\in\N$ and let $h = \frac1N$ be the time-step size. We consider the initial data $(\bv_0,\bw_0)\in\mathbfcal{L}^2_\Div$ and $(\phi_0,\psi_0)\in\mathcal{H}^1$ satisfying \eqref{NSCH:cond:init}. Then, according to Lemma~\ref{NSCH:Lemma:ApproxIC}, there exists a sequence $\{(\widetilde\phi_0^N,\widetilde\psi_0^N)\}_{N\in\N}\subset\mathcal{H}^2$ such that $(\widetilde\phi_0^N,\widetilde\psi_0^N)\rightarrow(\phi_0,\psi_0)$ strongly in $\mathcal{H}^1$ as $N\rightarrow\infty$ such that 
    \begin{subequations}\label{NSCH:Cond:Init:Approx}
    \begin{align}\label{NSCH:Cond:Init:Approx:Int}
    	\abs{\widetilde\phi_0^N} \leq 1 \quad\text{a.e.~in~}\Om \quad\text{and}\quad \abs{\widetilde\psi_0^N} \leq 1 \quad\text{a.e.~on~}\Ga
    \end{align} 
    for all $N\in\N$. In particular, in view of the strong convergence, there exists $m\in[0,1)$ and $\overline N\in\N$ such that, for all $N\geq\overline N$, it holds
    \begin{align}
        \abs{\beta\bigmean{\widetilde\phi_0^N}{\widetilde\psi_0^N}} \leq \overline m < 1, \qquad \abs{\bigmean{\widetilde\phi_0^N}{\widetilde\psi_0^N}} \leq \overline m < 1
    \end{align}
    if $L\in[0,\infty)$, and
    \begin{align}
        \abs{\meano{\widetilde\phi_0^N}} \leq\overline m < 1, \qquad\abs{\meang{\widetilde\psi_0^N}} \leq\overline m < 1 
    \end{align}
    \end{subequations}
    if $L = \infty$. Thus, redefining $(\phi_0^N,\psi_0^N) = (\widetilde\phi_0^{N+\overline N},\widetilde\psi_0^{N+\overline N})$ for $N\in\N$, we infer that $(\phi_0^N,\psi_0^N)$ is an admissible pair of initial data for all $N\in\N$. Then, according to Lemma~\ref{NSCH:Lemma:DiscrEnergyInequality} and Lemma~\ref{NSCH:Lemma:Existence:Discrete}, for any $k\in\N$, we can choose iteratively a solution $(\bv_{k+1},\bw_{k+1},\phi_{k+1},\psi_{k+1},\mu_{k+1},\theta_{k+1})$  to the time-discrete system \eqref{NSCH:VW:Discr}-\eqref{NSCH:Mu:Discr:Bc} with initial data $(\bv_0,\bw_0,\phi_0^N,\psi_0^N)$, which additionally satisfies the discrete energy inequality \eqref{NSCH:EnergyEstimate:Discr}.

	As in \cite[Section~5]{Abels2013}, we define $f^N(t)$ on $[-h,\infty)$ through 
    \begin{align*}
    	f^N(t) \coloneqq f_k \qquad\text{for~}t\in[(k-1)h,kh), \ k\in\N_0,
    \end{align*}
    where $f\in\{\bv,\bw,\phi,\psi,\mu,\theta\}$. Moreover, we set
	\begin{align*}
		\rho^N \coloneqq \rho(\phi^N), \qquad \sigma^N \coloneqq \sigma(\psi^N).
	\end{align*}	    
    With these definitions at hand, it in particular holds that
    \begin{align*}
        f^N((k-1)h) = f_k, \qquad f^N(kh) = f_{k+1}, \qquad f^N(t) = f_{k+1} \qquad\text{for~}t\in[kh,(k+1)h).
    \end{align*}
    Moreover, we define $f_h(t) \coloneqq f(t-h)$ and
    \begin{alignat*}{2}
        (\Lap_h^+f)(t) &\coloneqq f(t+h) - f(t), \qquad \delthp f(t) &&\coloneqq \tfrac1h(\Lap_h^+f)(t), \\
        (\Lap_h^-f)(t) &\coloneqq f(t) - f(t-h), \qquad \delthm f(t) &&\coloneqq \tfrac1h(\Lap_h^-f)(t).
    \end{alignat*}
    Then, for $(\wv,\ww)\in C_0^\infty(0,\infty;\mathbfcal{H}^2_{0,\Div})$, we take $(\tv,\tw) = \int_{kh}^{(k+1)h}(\wv,\ww)\dt$ as a test function  in \eqref{NSCH:VW:Discr} and sum over $k\in\N$. This yields
    \begin{align}\label{NSCH:WF:vw:N}
        &-\int_Q\rho^N\bv^N\cdot\delthp\wv\dxt - \int_\Sigma\sigma^N\bw^N\cdot\delthp\ww\dGt \nonumber \\
        &\qquad + \int_Q\Div(\rho_h^N\bv^N\otimes\bv^N)\cdot\wv\dxt + \int_\Sigma\Divg(\sigma_h^N\bw^N\otimes\bw^N)\cdot\ww\dGt \nonumber \\
        &\qquad + \int_Q 2\nu_\Om(\phi_h^N)\D\bv^N:\D\wv\dxt + \int_\Sigma 2\nu_\Ga(\psi_h^N)\Dg\bw^N:\Dg\ww\dGt \nonumber \\
        &\qquad + \int_\Sigma \gamma(\phi_h^N,\psi_h^N)\bw^N\cdot\ww\dGt - \int_Q(\bv^N\otimes\J^N):\Grad\wv\dxt \\
        &\qquad  - \int_\Sigma(\bw^N\otimes\K^N):\Gradg\ww\dGt \nonumber \\
        &\quad = \int_Q\mu^N\Grad\phi_h^N\cdot\wv\dxt + \int_\Sigma\theta^N\Gradg\psi_h^N\cdot\ww\dGt \nonumber \\
        &\qquad + \frac12\int_\Sigma(\J^N-\J_\Ga^N)\cdot\n\;\bw^N\cdot\ww\dGt, \nonumber
    \end{align}
    where we used the notation
    \begin{align*}
        \J^N &\coloneqq -\frac{\tilde\rho_2 - \tilde\rho_1}{2}m_\Om(\phi_h^N)\Grad\mu^N, \qquad \K^N \coloneqq -\frac{\tilde\sigma_2 - \tilde\sigma_1}{2}m_\Ga(\psi_h^N)\Gradg\theta^N, \\
        \J_\Ga^N &\coloneqq -\beta\frac{\tilde\sigma_2 - \tilde\sigma_1}{2}m_\Om(\phi_h^N)\Grad\mu^N.
    \end{align*}
    In \eqref{NSCH:WF:vw:N}, we additionally made use of the following integration by parts formula
    \begin{align*}
        &\int_Q\delthm(\rho^N\bv^N)\cdot\wv\dxt + \int_\Sigma\delthm(\sigma^N\bw^N)\cdot\ww\dGt \\
        &\qquad= -\int_Q\rho^N\bv^N\cdot\delthp\wv\dxt - \int_\Sigma\sigma^N\bw^N\cdot\delthp\ww\dGt.
    \end{align*}
    In a similar manner, we derive from \eqref{NSCH:Phi:Discr},\eqref{NSCH:Psi:Discr} and \eqref{NSCH:Phi:Discr:Bc} that
    \begin{align}\label{NSCH:WF:pp:N}
        \begin{split}
            &\int_{Q}\delthm\phi^N\,\zeta\dxt + \int_{\Sigma}\delthm\psi^N\,\zeta_\Ga\dGt \\
            &\quad= \int_{Q}\phi_h^N\bv^N\cdot\Grad\zeta\dxt + \int_{\Sigma}\psi_h^N\bw^N\cdot\Gradg\zeta_\Ga\dGt \\
            &\qquad - \int_{Q} m_\Om(\phi_h^N)\Grad\mu^N\cdot\Grad\zeta\dxt - \int_{\Sigma} m_\Ga(\psi_h^N)\Gradg\theta^N\cdot\Gradg\zeta_\Ga\dGt \\
            &\qquad - \chi(L)\int_{\Sigma}(\beta\theta^N - \mu^N)(\beta\zeta_\Ga - \zeta)\dGt
        \end{split}
    \end{align}
    for all $(\zeta,\zeta_\Ga)\in C_0^\infty(0,\infty;\mathcal{H}^1_{L,\beta})$, while we infer from \eqref{NSCH:Mu:Discr}, \eqref{NSCH:Theta:Discr} and \eqref{NSCH:Mu:Discr:Bc} that
    \begin{subequations}\label{NSCH:WF:mt:N}
        \begin{alignat}{2}
            &\mu^N + \frac{c_F}{2}(\phi + \phi^N) = -\Lap\phi^N + F_0^\prime(\phi^N) &&\qquad\text{a.e.~in~}Q, \label{NSCH:WF:mu:N}\\
            &\theta^N + \frac{c_G}{2}(\psi + \psi^N) = -\Lapg\psi^N + G_0^\prime(\psi^N) + \alpha\deln\phi^N &&\qquad\text{a.e.~on~}\Sigma, \label{NSCH:WF:theta:N}\\
            &K\deln\phi^N = \alpha\psi^N - \phi^N, \qquad K\in(0,\infty], &&\qquad\text{a.e.~on~}\Sigma. \label{NSCH:WF:phi:bc:N}
        \end{alignat}
    \end{subequations}

    Let $E^N(t)$ be the piecewise linear interpolant of $E_{\mathrm{tot}}(\bv_k,\bw_k,\phi_k,\psi_k)$ at $t_k = kh$, i.e., $E^N(t)$ is defined as
    \begin{align*}
        E^N(t) \coloneqq \frac{(k+1)h - t}{h} E_{\mathrm{tot}}(\bv_k,\bw_k,\phi_k,\psi_k) + \frac{t-kh}{h}E_{\mathrm{tot}}(\bv_{k+1},\bw_{k+1},\phi_{k+1},\psi_{k+1})
    \end{align*}
    for $t\in (t_k,t_{k+1})$. For all $t\in (t_k,t_{k+1})$, $k\in\N$, we also define the dissipation
    \begin{align*}
        D^N(t) &\coloneqq \intO 2\nu_\Om(\phi_k)\abs{\D\bv_{k+1}}^2\dx + \intG 2\nu_\Ga(\psi_k)\abs{\Dg\bw_{k+1}}^2\dG + \intG \gamma(\phi_k,\psi_k)\abs{\bw_{k+1}}^2\dG \\
        &\quad + \intO m_\Om(\phi_k)\abs{\Grad\mu_{k+1}}\dx + \intG m_\Ga(\psi_k)\abs{\Gradg\theta_{k+1}}^2\dG + \chi(L) \intG (\beta\theta_{k+1} - \mu_{k+1})^2\dG.
    \end{align*}
    From the discrete energy estimate \eqref{NSCH:EnergyEstimate:Discr} we obtain
    \begin{align}\label{NSCH:Est:E^N}
        -\ddt E^N(t) = \frac{E_{\mathrm{tot}}(\bv_k,\bw_k,\phi_k,\psi_k) - E_{\mathrm{tot}}(\bv_{k+1},\bw_{k+1},\phi_{k+1},\psi_{k+1})}{h} \geq D^N(t)
    \end{align}
    for all $t\in(t_k,t_{k+1})$, $k\in\N$. Then, integrating \eqref{NSCH:Est:E^N} in time directly leads to
    \begin{align}\label{NSCH:EnergyEstimate:Cont}
        \begin{split}
            &E_{\mathrm{tot}}\big(\bv^N(t),\bw^N(t),\phi^N(t),\psi^N(t)\big) + \int_s^t\intO 2\nu_\Om(\phi_h^N)\abs{\D\bv^N}^2\dxtau \\
            &\qquad + \int_s^t\intG 2\nu_\Ga(\psi_h^N)\abs{\Dg\bw^N}^2\dGtau + \int_s^t\intG \gamma(\phi_h^N,\psi_h^N)\abs{\bw^N}^2\dGtau \\
            &\qquad + \int_s^t\intO m_\Om(\phi_h^N)\abs{\Grad\mu^N}^2\dxtau + \int_s^t\intG m_\Ga(\psi_h^N)\abs{\Gradg\theta^N}^2\dGtau \\
            &\qquad + \chi(L) \int_s^t\intG (\beta\theta^N - \mu^N)^2\dGtau \\
            &\quad \leq E_{\mathrm{tot}}\big(\bv^N(s),\bw^N(s),\phi^N(s),\psi^N(s)\big)
        \end{split}
    \end{align}
    for all $0 \leq s \leq t < \infty$ with $s,t\in h\N$. Now, by \eqref{NSCH:Cond:Init:Approx:Int}, we have that $F_0(\phi_0^N)$ and $G_0(\psi_0^N)$ are uniformly bounded in $L^1(\Om)$ and $L^1(\Ga)$, respectively. Then, invoking the bulk-surface Korn inequality and, recalling that the discrete solution $(\phi_k,\psi_k)$ satisfies the mass conservation law \eqref{NSCH:MCL:Discrete}, the bulk-surface Poincar\'{e} inequality, we find a constant $C > 0$, independent of $N\in\N$, such that
    \begin{align}
    	&\norm{(\bv^N,\bw^N)}_{L^\infty(0,\infty;\mathbfcal{L}^2)} + \norm{(\bv^N,\bw^N)}_{L^2(0,\infty;\mathbfcal{H}^1)} \leq C, \label{NSCH:Est:vw:Uniform:N:1} \\
    	&\norm{(\phi^N,\psi^N)}_{L^\infty(0,\infty;\mathcal{H}^1)} \leq C, \label{NSCH:Est:pp:Uniform:N:1} \\
    	&\norm{(\Grad\mu^N,\Gradg\theta^N)}_{L^2(0,\infty;\mathbfcal{L}^2)} + \chi(L)^{\frac12}\norm{\beta\theta^N-\mu^N}_{L^2(0,\infty;L^2(\Ga))} \leq C. \label{NSCH:Est:mt:Uniform:N:1}
	\end{align}
    Furthermore, in view Lemma~\ref{NSCH:Lemma:H^2+Subdiff}, we find that
    \begin{align*}
        \int_0^T\abs{\mean{\mu^N}{\theta^N}}\dt \leq C\Big(T + \int_0^T \norm{(\mu^N,\theta^N)}_{L,\beta}\dt\Big)
    \end{align*}
    for all $T > 0$. Thus, employing again the bulk-surface Poincar\'{e} inequality, for any $T > 0$, there exists a constant $C = C(T)$, independent of $N$, such that
    \begin{align}\label{NSCH:Est:mt:Uniform:N:2}
        \norm{(\mu^N,\theta^N)}_{L^2(0,T;\mathcal{H}^1)} \leq C .
    \end{align}
    Using these uniform bounds, we can apply the Banach--Alaoglu theorem to conclude the existence of a sextuplet $(\bv,\bw,\phi,\psi,\mu,\theta)$ such that
    \begin{subequations}
    \begin{alignat}{2}
        (\bv^N,\bw^N) &\rightarrow (\bv,\bw) &&\qquad\text{weakly in~} L^2(0,\infty;\mathbfcal{H}^1), \label{NSCH:Conv:vw:Weakly} \\
        & &&\qquad\text{weakly-star in~} L^\infty(0,\infty;\mathbfcal{L}^2), \\
        (\phi^N,\psi^N) &\rightarrow (\phi,\psi) &&\qquad\text{weakly-star in~} L^\infty(0,\infty;\mathcal{H}^1), \label{NSCH:Conv:pp:Weakly} \\
        (\Grad\mu^N,\Gradg\theta^N) &\rightarrow (\Grad\mu,\Gradg\theta) &&\qquad\text{weakly in~} L^2(0,\infty;\mathbfcal{L}^2), \\
        \beta\theta^N - \mu^N &\rightarrow \beta\theta - \mu &&\qquad\text{weakly in~} L^2(0,\infty; L^2(\Ga)), \label{NSCH:Conv:bt-m:Weakly} \\
        (\mu^N,\theta^N) &\rightarrow (\mu,\theta) &&\qquad\text{weakly in~} L^2(0,T;\mathcal{H}^1), \label{NSCH:Conv:mt:Weakly}
    \end{alignat}
    \end{subequations}
    for all $T > 0$, along a non-relabeled subsequence $N\rightarrow\infty$. To pass to the limit in the nonlinearities, we need additional strong convergences of $(\phi^N,\psi^N)$ and $(\bv^N,\bw^N)$. 

    To this end, let $\widetilde{\phi}^N \coloneqq \tfrac1h\mathbf{1}_{[0,h]}\ast_t\phi^N$ be the piecewise linear interpolant of $\phi^N(t_k)$ where $t_k = kh$ and $k\in\N$. Here, the convolution is only taken with respect to the time variable. Similarly, we define $\widetilde{\psi}^N \coloneqq \tfrac1h\mathbf{1}_{[0,h]}\ast_t\phi^N$. Then,
    \begin{align*}
        \delt\widetilde{\phi}^N = \delthm\phi^N, \qquad \delt\tilde{\psi}^N = \delthm\psi^N,
    \end{align*}
    and
    \begin{align}\label{NSCH:Est:pp:N:Diff}
        \norm{(\widetilde{\phi}^N - \phi^N, \widetilde{\psi}^N - \psi^N)}_{(\mathcal{H}^1_{L,\beta})^\prime} \leq h\norm{(\delt\widetilde{\phi}^N, \delt\widetilde{\psi}^N)}_{(\mathcal{H}^1_{L,\beta})^\prime}.
    \end{align}
    Due to the aforementioned uniform bounds \eqref{NSCH:Est:vw:Uniform:N:1}-\eqref{NSCH:Est:mt:Uniform:N:1}, we infer, using a comparison argument in \eqref{NSCH:WF:pp:N}, that
    \begin{align}\label{NSCH:Est:delt:pp:tilde:N}
        \norm{(\delt\widetilde{\phi}^N,\delt\widetilde{\psi}^N)}_{L^2(0,\infty;(\mathcal{H}^1_{L,\beta})^\prime)} \leq C.
    \end{align}
    Next, note that the boundedness of $(\phi^N,\psi^N)$ in $L^\infty(0,\infty;\mathcal{H}^1)$ implies the boundedness of $(\widetilde{\phi}^N,\widetilde{\psi}^N)$ in the same space. Combining this fact with the estimate \eqref{NSCH:Est:delt:pp:tilde:N}, we get, with the help of the Aubin--Lions--Simon lemma, for any $T > 0$, the strong convergence
    \begin{align}\label{NSCH:Conv:pp:Strongly}
        (\widetilde \phi^N, \widetilde \psi^N) \rightarrow (\widetilde \phi, \widetilde \psi) \quad\text{strongly in~} C([0,T];\mathcal{H}^{1-s}) \text{~for all~}s\in(0,\tfrac12)
    \end{align}
    along a subsequence $N\rightarrow\infty$ for some $(\widetilde \phi, \widetilde \psi)\in L^\infty(0,\infty;\mathcal{H}^1)$. In particular, it holds that $\widetilde \phi^N\rightarrow\widetilde \phi$ a.e.~in $Q$ as well as $\widetilde \psi^N\rightarrow\widetilde \psi$ a.e.~on $\Sigma$. Furthermore, the estimates \eqref{NSCH:Est:pp:N:Diff} and \eqref{NSCH:Est:delt:pp:tilde:N} entail that
    \begin{align*}
        (\widetilde \phi^N - \phi^N, \widetilde \psi^N - \psi^N) \rightarrow (0,0) \qquad\text{strongly in~} L^2(0,\infty;(\mathcal{H}^1_{L,\beta})^\prime)
    \end{align*}
    as $N\rightarrow\infty$, and consequently, $(\widetilde \phi,\widetilde \psi) = (\phi,\psi)$. Moreover, the convergences \eqref{NSCH:Conv:pp:Weakly} and \eqref{NSCH:Conv:pp:Strongly} combined with Lemma~\ref{Prelim:Lemma:BC_w} yield
    \begin{align*}
        (\phi,\psi) \in BC_w([0,\infty);\mathcal{H}^1).
    \end{align*}
    Then, to verify the initial conditions for $\phi$ and $\psi$, we first notice that, in view of \eqref{NSCH:Conv:pp:Strongly}, we have 
\begin{align*}
	(\widetilde\phi^N(0),\widetilde\psi^N(0))\rightarrow (\phi(0),\psi(0))\qquad\text{strongly in~}\mathcal{L}^2
\end{align*}
as $N\rightarrow\infty$. Besides, it holds that
\begin{align*}
	(\widetilde\phi^N(0),\widetilde\psi^N(0)) = (\phi_0^N,\psi_0^N)\rightarrow (\phi_0,\psi_0) \qquad\text{strongly in~}\mathcal{H}^1
\end{align*}
as $N\rightarrow\infty$. We deduce that $(\phi,\psi)\vert_{t=0} = (\phi_0,\psi_0)$ a.e.~in $\Om\times\Ga$.

    To pass to the limit in \eqref{NSCH:WF:mu:N}-\eqref{NSCH:WF:theta:N}, we first point out that the right-hand side is given by $\partial \widetilde E_\free(\phi^N,\psi^N)$, which is bounded according to Lemma~\ref{NSCH:Lemma:H^2+Subdiff} in $L^2(0,T;\mathcal{L}^2)$ for all $T > 0$. Furthermore, the left-hand side converges weakly in $L^2(0,T;\mathcal{L}^2)$ as $N\rightarrow\infty$ to
    \begin{align*}
        (f,f_\Ga) \coloneqq (\mu + c_F\phi, \theta + c_G\psi)
    \end{align*}
    for all $T > 0$, which means that 
    \begin{align*}
        \partial \widetilde E_\free(\phi^N,\psi^N) \rightarrow (f,f_\Ga) \qquad\text{weakly in~} L^2(0,T;\mathcal{L}^2)
    \end{align*}
    as $N\rightarrow\infty$ for all $T > 0$.
    Consequently, we infer together with the strong convergence \eqref{NSCH:Conv:pp:Strongly} that
    \begin{align*}
        \bigang{\partial \widetilde E_\free(\phi^N,\psi^N)}{(\phi^N,\psi^N)}_{\mathcal{L}^2} \rightarrow \bigang{(\mu + c_F\phi,\theta + c_G\psi)}{(\phi,\psi)}_{\mathcal{L}^2}
    \end{align*}
    as $N\rightarrow\infty$. With $\partial\widetilde E_\free$ being a maximal monotone operator, we infer that $(\phi,\psi)\in\mathrm{dom}\,(\partial\widetilde E_\free)$ and $\partial\widetilde E_\free(\phi,\psi) = (f,f_\Ga)$ by \cite[Proposition~2.1]{Barbu2010}, thereby establishing \eqref{NSCH:WeakSolution:MuTheta}. This identity additionally entails that $\partial\widetilde E_\free(\phi,\psi)\in L^2_{\mathrm{uloc}}([0,\infty);\mathcal{H}^1)$, and thus, with another application of \cite[Proposition~5.1]{Giorgini2026} implies the desired regularities
    \begin{align*}
    	(\phi,\psi)\in L^2_{\mathrm{uloc}}([0,\infty);\mathcal{W}^{2,p}), \qquad (F_0^\prime(\phi),G_0^\prime(\psi))\in L^2_{\mathrm{uloc}}([0,\infty);\mathcal{L}^p),
    \end{align*}
    for $p = 6$ if $d = 3$ and any $p\in[2,\infty)$ if $d = 2$.
    
    To get a better strong convergence result for $(\phi^N,\psi^N)$,
    we first recall that by Lemma~\ref{NSCH:Lemma:H^2+Subdiff}, $\{(\phi^N,\psi^N)\}_{N\in\N}$ is bounded in $L^2(0,T;\mathcal{H}^2)$ for any $T > 0$, which, together with \eqref{NSCH:Conv:pp:Strongly} and suitable interpolation inequalities further implies that, up to a subsequence,
    \begin{align}\label{NSCH:Conv:pp:Strong:H2-s}
    	(\phi^N,\psi^N) \rightarrow (\phi,\psi) \qquad\text{strongly in~} L^2(0,T;\mathcal{H}^{2-s}) \quad\text{for all~}s\in(0,1)
    \end{align}
    as $N\rightarrow\infty$ for any $T > 0$.
    
    In the next step, we show the strong convergence of $\{(\bv^N,\bw^N)\}_{N\in\N}$ in $L^2(0,T;\mathbfcal{L}^2)$ for any $T > 0$. To this end, we denote with $\widetilde{\rho\bv}^N$ and $\widetilde{\sigma\bw}^N$ the piecewise linear interpolant of $(\rho^N\bv^N)(t_k)$ and $(\sigma^N\bw^N)(t_k)$, respectively, where $t_k = kh$, $k\in\N$, and let $T > 0$. As previously seen, it then holds that
    \begin{align*}
    	\delt(\widetilde{\rho\bv}^N) = \delthm(\rho^N\bv^N), \qquad \delt(\widetilde{\sigma\bw}^N) = \delthm(\sigma^N\bw^N).
    \end{align*}
    Now, proceeding similarly to the line of argument in \cite[Section~5]{Abels2013}, we use the uniform estimates \eqref{NSCH:Est:vw:Uniform:N:1}-\eqref{NSCH:Est:mt:Uniform:N:1} along with standard interpolation results to show that
\begin{align}\label{NSCH:Est:Vel:AGD}
    \{(\rho_h^N\bv^N\otimes\bv^N, \sigma_h^N\bw^N\otimes\bw^N)\}_{N\in\N}
        &\ \text{~is uniformly bounded in~} \ L^2(0,T;\mathbfcal{L}^{\frac32}), 
        \nonumber \\
    \{(\bv^N\otimes\Grad\mu^N, \bw^N\otimes\Gradg\theta^N)\}_{N\in\N} 
        &\ \text{~is uniformly bounded in~} \ L^{\frac87}(0,T;\mathbfcal{L}^{\frac43}), 
       \nonumber \\
    \{(\mu^N\Grad\phi_h^N, \theta^N\Gradg\psi_h^N)\}_{N\in\N}
        &\ \text{~is uniformly bounded in~} \ L^2(0,T;\mathbfcal{L}^{\frac32}), 
        \\
    \{(\beta\theta^N - \mu^N)\bw^N\}_{N\in\N}
        &\ \text{~is uniformly bounded in~} \ L^{\frac87}(0,T;\mathbf{L}^{\frac43}(\Ga)). \nonumber
\end{align}
Moreover, via the trace embedding $H^1(\Omega) \emb L^2(\Gamma)$, we infer from \eqref{NSCH:Est:vw:Uniform:N:1} that
\begin{align}\label{NSCH:Est:Vel:AGD:2}
    \begin{split}
        \{(\D\bv^N,\Dg\bw^N)\}_{N\in\N}
        &\ \text{~is uniformly bounded in~} L^2(0,T;\mathbfcal{L}^2), 
        \\
        \{\bw^N\}_{N\in\N}
        &\ \text{~is uniformly bounded in~} L^2(0,T;\mathbf{L}^2(\Ga)).
    \end{split}
\end{align}
Regarding the last term on the left-hand side of \eqref{NSCH:WF:vw:N}, we note that, if $L\in(0,\infty]$, it holds that
\begin{align*}
	&\frac12\intS(\J^N-\J_\Ga^N)\cdot\n\,\bw^N\cdot\ww\dGt \\
	&\quad = \frac12\Big(\beta\frac{\tilde\sigma_2 - \tilde\sigma_1}{2} + \frac{\tilde\rho_2 - \tilde\rho_1}{2}\Big)\intS m_\Om(\phi_h^N)\deln\mu^N\bw^N\cdot\ww\dGt \\
	&\quad = \frac{\chi(L)}{2}\Big(\beta\frac{\tilde\sigma_2 - \tilde\sigma_1}{2} + \frac{\tilde\rho_2 - \tilde\rho_1}{2}\Big)\intS (\beta\theta^N-\mu^N)\bw^N\cdot\ww\dGt,
\end{align*}
whereas in the case $L = 0$, we have
\begin{align*}
	\frac12\intS(\J^N-\J_\Ga^N)\cdot\n\,\bw^N\cdot\ww\dGt = 0
\end{align*}
for all $\bw\in C_c^\infty(0,\infty;\mathbf{H}^2_\Div(\Ga))$ as $\beta(\tilde\sigma_2-\tilde\sigma_1)=-(\tilde\rho_2-\tilde\rho_1)$ by assumption \eqref{Densities:Comp:Weak}. Thus, due to this observation and the uniform bounds \eqref{NSCH:Est:Vel:AGD} and \eqref{NSCH:Est:Vel:AGD:2}, any $(\wv,\ww)\in L^8(0,T;\mathbfcal{W}^{1,4}_{0,\Div})$ is an admissible test function in \eqref{NSCH:WF:vw:N}. Hence, by a comparison argument, we infer that $\{\delt\mathbfcal{P}_\Div(\widetilde{\rho\bv}^N,\widetilde{\sigma\bw}^N)\}_{N\in\N}$ is bounded in $L^{\frac87}(0,T;(\mathbfcal{W}^{1,4}_{0,\Div})^\prime)$, where $\mathbfcal{P}_\Div = (\mathbf P_\Div^\Om, \mathbf P_\Div^\Ga)$. In addition, as we know that $\{\mathbfcal{P}_\Div(\widetilde{\rho\bv}^N,\widetilde{\sigma\bw}^N)\}_{N\in\N}$ is bounded in $L^2(0,T;\mathbfcal{H}^1)$, we deduce from the Aubin--Lions--Simon lemma the strong convergence
\begin{align}\label{NSCH:Conv:P_Div:N:L2L2}
	\mathbfcal{P}_\Div(\widetilde{\rho\bv}^N,\widetilde{\sigma\bw}^N) \rightarrow (\bv_\ast,\bw_\ast) \qquad\text{strongly in~} L^2(0,T;\mathbfcal{L}^2)
\end{align}
as $N\rightarrow\infty$, up to a subsequence, for some $(\bv_\ast,\bw_\ast)\in L^\infty(0,T;\mathbfcal{L}^2)$. Then, exploiting the weak continuity of $\mathbfcal{P}_\Div$ we conclude from the weak convergence $(\widetilde{\rho\bv}^N,\widetilde{\sigma\bw}^N)\rightarrow (\rho\bv,\sigma\bw)$ in $L^2(0,T;\mathbfcal{L}^2)$ as $N\rightarrow\infty$ that $(\bv_\ast,\bw_\ast) = \mathbfcal{P}_\Div(\rho\bv,\sigma\bw)$.

Next, due to \eqref{NSCH:Conv:vw:Weakly} and \eqref{NSCH:Conv:P_Div:N:L2L2}, we have
\begin{align*}
	&\int_{Q_T} \rho^N\abs{\bv^N}^2\dxt + \int_{\Sigma_T} \sigma^N\abs{\bw^N}^2\dGt \\
	&\quad = \int_{Q_T} \mathbf{P}_\Div^\Om(\rho^N\bv^N)\cdot\bv^N\dxt + \int_{\Sigma_T} \mathbf{P}_\Div^\Ga(\sigma^N\bw^N)\cdot\bw^N\dGt \\
	&\quad\longrightarrow \int_{Q_T} \mathbf{P}_\Div^\Om(\rho\bv)\cdot\bv\dxt + \int_{\Sigma_T} \mathbf{P}_\Div^\Ga(\sigma\bw)\cdot\bw\dGt \\
	&\quad = \int_{Q_T} \rho\abs{\bv}^2\dxt + \int_{\Sigma_T} \sigma\abs{\bw}^2\dGt
\end{align*}
as $N\rightarrow\infty$. Moreover, for all $(\tv,\tw)\in L^2(0,T;\mathbfcal{L}^2)$ it holds that
\begin{align*}
	&\int_{Q_T} \big((\rho^N)^{\frac12}\bv^N - \rho(\phi)^{\frac12}\bv\big)\cdot\wv\dxt + \int_{\Sigma_T}\big((\sigma^N)^{\frac12}\bw^N - \sigma^{\frac12}(\psi)\bw\big)\cdot\ww\dGt \\
	&\quad = \int_{Q_T} \big((\rho^N)^{\frac12} - \rho^\frac12\big)\wv\cdot\bv\dxt + \int_{Q_T} \rho^{\frac12}(\bv^N - \bv)\cdot\wv\dxt \\
	&\qquad + \int_{\Sigma_T} \big((\sigma^N)^{\frac12} - \sigma^{\frac12}\big)\ww\cdot\bw^N\dGt + \int_{\Sigma_T}\sigma^{\frac12}(\bw^N - \bw)\cdot\ww\dGt \\
	&\quad\longrightarrow 0
\end{align*}
as $N\rightarrow\infty$. Here, we have used the strong convergence 
\begin{align*}
	\big((\rho^N)^{\frac12}\wv,(\sigma^N)^{\frac12}\ww\big) \rightarrow \big(\rho^{\frac12}\wv,\sigma^{\frac12}\ww\big) \qquad\text{strongly in~} L^2(0,T;\mathbfcal{L}^2)
\end{align*}
as $N\rightarrow\infty$ in combination with the weak convergence \eqref{NSCH:Conv:vw:Weakly}. Consequently, as we have weak convergence and the convergence of norms, we readily conclude that
\begin{align*}
	\big((\rho^N)^{\frac12}\bv^N, (\sigma^N)^{\frac12}\bw^N\big) \rightarrow \big(\rho^{\frac12}\bv,\sigma^{\frac12}\bw) \qquad\text{strongly in~} L^2(0,T;\mathbfcal{L}^2)
\end{align*}
as $N\rightarrow\infty$. Lastly, since $\rho^N\rightarrow\rho$ a.e.~in $Q$, and $\rho \geq \rho_\ast > 0$, we use Lebesgue's dominated convergence theorem to deduce that
\begin{align}\label{NSCH:Conv:v:Strong:L2}
	\bv^N = (\rho^N)^{-\frac12}\big((\rho^N)^{\frac12}\bv^N\big) \rightarrow \bv \qquad\text{strongly in~} L^2(0,T;\mathbf{L}^2(\Om))
\end{align}
as $N\rightarrow\infty$. Similarly, we find that
\begin{align}\label{NSCH:Conv:w:Strong:L2}
	\bw^N = (\sigma^N)^{-\frac12}\big((\sigma^N)^{\frac12}\bw^N\big) \rightarrow \bw \qquad\text{strongly in~} L^2(0,T;\mathbf{L}^2(\Ga))
\end{align}
as $N\rightarrow\infty$.

Now, to pass to the limit in \eqref{NSCH:WF:vw:N}, \eqref{NSCH:WF:pp:N} and \eqref{NSCH:WF:mt:N}, we first notice that
\begin{align}\label{NSCH:CONV:pp:Strong:L4H1}
	(\phi^N,\psi^N) \rightarrow (\phi,\psi) \qquad\text{strongly in~} L^4(0,T;\mathcal{H}^1)
\end{align}
as $N\rightarrow\infty$, which follows by interpolation using \eqref{NSCH:Est:pp:Uniform:N:1} and \eqref{NSCH:Conv:pp:Strong:H2-s}. We thus readily get
\begin{align*}
	&\int_{Q_T}\mu^N\Grad\phi_h^N\cdot\wv\dxt + \int_{\Sigma_T}\theta^N\Gradg\psi_h^N\cdot\ww\dGt \\
	&\quad = -\int_{Q_T}\big(\Grad\mu^N\phi_h^N\big)\cdot\wv\dxt - \int_{\Sigma_T}\big(\Gradg\theta^N\psi_h^N\big)\cdot\ww\dGt \\
	&\quad\longrightarrow -\int_{Q_T}\big(\Grad\mu\,\phi\big)\cdot\wv\dxt - \int_{\Sigma_T}\big(\Gradg\theta\,\psi\big)\cdot\ww\dGt \\
	&\quad = \int_{Q_T}\mu\Grad\phi\cdot\wv\dxt + \int_{\Sigma_T}\theta\Gradg\psi\cdot\ww\dGt
\end{align*}
as $N\rightarrow\infty$ for all $(\wv,\ww)\in C_c^\infty(0,T;\mathbfcal{H}^2_{0,\Div})$ . Then, \eqref{NSCH:Conv:bt-m:Weakly} together with the strong convergence \eqref{NSCH:Conv:w:Strong:L2} readily implies in the case $L\in(0,\infty]$ that
\begin{align*}
	&\frac12\int_{\Sigma_T}\big(\J^N - \J_\Ga^N\big)\cdot\bw^N\cdot\ww\dGt \\
	&\quad = \frac{\chi(L)}{2}\Big(\beta\frac{\tilde\sigma_2 - \tilde\sigma_1}{2} + \frac{\tilde\rho_2 - \tilde\rho_1}{2}\Big)\int_{\Sigma_T} (\beta\theta^N-\mu^N)\bw^N\cdot\ww\dGt \\
	&\rightarrow \frac{\chi(L)}{2}\Big(\beta\frac{\tilde\sigma_2 - \tilde\sigma_1}{2} + \frac{\tilde\rho_2 - \tilde\rho_1}{2}\Big)\int_{\Sigma_T} (\beta\theta-\mu)\bw\cdot\ww\dGt
\end{align*}
as $N\rightarrow\infty$ for all $(\wv,\ww)\in C_c^\infty(0,T;\mathbfcal{H}^2_{0,\Div})$. In the case $L = 0$, we have already seen that this term disappears in view of the assumption \eqref{Densities:Comp:Weak}. Then, recalling the convergences \eqref{NSCH:Conv:vw:Weakly}-\eqref{NSCH:Conv:mt:Weakly}, \eqref{NSCH:Conv:pp:Strong:H2-s} and \eqref{NSCH:Conv:v:Strong:L2}-\eqref{NSCH:Conv:w:Strong:L2}, as well as the uniform bounds \eqref{NSCH:Est:Vel:AGD}-\eqref{NSCH:Est:Vel:AGD:2}, we can pass to the limit $N\rightarrow\infty$ in \eqref{NSCH:WF:vw:N} and \eqref{NSCH:WF:pp:N} to get \eqref{NSCH:WF:VW} and \eqref{NSCH:WF:PP}.

Moreover, by arguing as in \cite[Section~6.4]{Knopf2025a}, we can show that
\begin{align*}
    (\bv,\bw)\in BC_w([0,\infty);\mathbfcal{L}^2),
\end{align*}
and that the initial conditions $\bv\vert_{t=0} = \bv_0$ a.e.~in $\Om$ and $\bw\vert_{t=0} = \bw_0$ a.e.~on $\Ga$ are fulfilled.

Finally, we verify the energy inequality \eqref{NSCH:WeakSolution:EnergyInequality} as well as the separate energy inequalities \eqref{EnergyInequality:Kinetic:K>0} and \eqref{EnergyInequality:Free:K>0}. To show \eqref{NSCH:WeakSolution:EnergyInequality}, let $\pi\in W^{1,1}(0,\infty)$ with $\pi\geq 0$, and multiply \eqref{NSCH:Est:E^N} with $\pi$, integrate in time, and apply integration by parts. We find
\begin{align}\label{NSCH:Est:E^N:2}
	E_{\mathrm{tot}}(\bv_0,\bw_0,\phi_0^N,\psi_0^N)\pi(0) + \int_0^\infty E^N(t)\pi^\prime(t)\dt \geq \int_0^\infty D^N(t)\pi(t)\dt.
\end{align}
Due to the strong convergences \eqref{NSCH:Conv:pp:Strong:H2-s},\eqref{NSCH:Conv:v:Strong:L2}, and \eqref{NSCH:Conv:w:Strong:L2} we readily find that
\begin{alignat*}{2}
	(\bv^N(t),\bw^N(t)) &\rightarrow (\bv(t),\bw(t)) &&\qquad\text{strongly in~}\mathbfcal{L}^2, \\
	(\phi^N(t),\psi^N(t)) &\rightarrow (\phi(t),\psi(t)) &&\qquad\text{strongly in~}C(\overline\Om)\times C(\Ga)
\end{alignat*}
for a.e.~$t\in(0,\infty)$ as $N\rightarrow\infty$ up to a subsequence extraction. These convergences readily entail that
\begin{align*}
	E^N(t) \rightarrow E_{\mathrm{tot}}(\bv(t),\bw(t),\phi(t),\psi(t))
\end{align*}
for a.e.~$t\in(0,\infty)$ as $N\rightarrow\infty$. On the other hand, by lower semicontinuity of norms and the almost everywhere convergence of $(\phi^N,\psi^N)$ to $(\phi,\psi)$, we find
\begin{align*}
	\liminf_{N\rightarrow\infty} \int_0^\infty D^N(t)\pi(t)\dt \geq \int_0^\infty D(t)\pi(t)\dt,
\end{align*}
where
\begin{align*}
	D(t) &\coloneqq \intO 2\nu_\Om(\phi(t))\abs{\D\bv(t)}^2\dx + \intG 2\nu_\Ga(\psi(t))\abs{\Dg\bw(t)}^2\dG + \intG \gamma(\phi(t),\psi(t))\abs{\bw(t)}^2\dG \\
	&\quad + \intO m_\Om(\phi(t))\abs{\Grad\mu(t)}^2\dx + \intG m_\Ga(\psi(t))\abs{\Gradg\theta(t)}^2\dG + \chi(L)\intG (\beta\theta(t) - \mu(t))^2\dG.
\end{align*}
We also refer to the proof of \cite[Theorem~4.1]{Knopf2024} or to \cite[Proof of Theorem~3.3, Step~4]{Giorgini2023}, where similar arguments were carried out in detail to show an energy inequality.
Hence, passing to the limit $N\rightarrow\infty$ in \eqref{NSCH:Est:E^N:2}, we deduce
\begin{align*}
	E_{\mathrm{tot}}(\bv_0,\bw_0,\phi_0,\psi_0)\pi(0) + \int_0^\infty E_{\mathrm{tot}}(\bv(t),\bw(t),\phi(t),\psi(t))\pi^\prime(t)\dt \geq \int_0^\infty D(t)\pi(t)\dt.
\end{align*}
Since $\pi\in W^{1,1}(0,\infty)$ with $\pi \geq 0$ was arbitrary, we can apply Lemma~\ref{NSCH:Prelim:Lemma:Energy}, and conclude that the energy inequality \eqref{NSCH:WeakSolution:EnergyInequality} holds true.

To show that the two separate energy inequalities \eqref{EnergyInequality:Kinetic:K>0} and \eqref{EnergyInequality:Free:K>0}, we start by noticing that the approximate solutions satisfy
\begin{align}\label{EI:DISCR:KIN}
    &E_\mathrm{kin}(\bv^N(t),\bw^N(t),\phi^N(t),\psi^N(t)) \nonumber \\
    &\qquad + \int_s^t\intO2\nu_\Om(\phi_h^N)\abs{\D\bv^N}^2\dxtau + \int_s^t\intG 2\nu_\Ga(\psi_h^N)\abs{\Dg\bw^N}^2\dGtau \nonumber \\
    &\qquad + \int_s^t\intG \gamma(\phi_h^N,\psi_h^N)\abs{\bw^N}^2\dGtau \nonumber \\
    &\quad\leq E_\mathrm{kin}(\bv^N(s),\bw^N(s),\phi^N(s),\psi^N(s)) - \int_s^t\intO \phi_h^N\bv^N\cdot\Grad\mu^N\dxtau \nonumber \\
    &\qquad - \int_s^t\intG\psi_h^N\bw^N\cdot\Gradg\theta^N\dGtau
\end{align}
as well as
\begin{align}\label{EI:DISCR:FREE}
    &E_\free(\phi^N(t),\psi^N(t)) \nonumber \\
    &\qquad + \int_s^t\intO m_\Om(\phi_h^N)\abs{\Grad\mu^N}^2\dxtau + \int_s^t\intG m_\Ga(\psi_h^N)\abs{\Gradg\theta^N}^2\dGtau \nonumber \\
    &\qquad + \chi(L)\int_s^t\intG (\beta\theta^N-\mu^N)^2\dGtau \nonumber \\
    &\quad\leq E_\free(\phi^N(s),\psi^N(s)) + \int_s^t\intO \phi_h^N\bv^N\cdot\Grad\mu^N\dxtau + \int_s^t\intG \psi_h^N\bw^N\cdot\Gradg\theta^N\dGtau
\end{align}
for all $t\geq 0$ and a.e~$s\in[0,\infty)$ with $s = 0$ included. In view of the convergences above, the energies at the terminal time are weakly lower semicontinuous with respect to $N\rightarrow\infty$, while the energies at the initial time $s$ converge for almost every $s\in(0,\infty)$. Together with the weak lower semicontinuity of the dissipation terms and the convergence of the convective coupling terms, this allows us to pass to the limit $N\rightarrow\infty$ in \eqref{EI:DISCR:KIN} and \eqref{EI:DISCR:FREE}. We thereby obtain 
\begin{align}\label{EI:DISCR:KIN:AE}
    &E_\mathrm{kin}(\bv(t),\bw(t),\phi(t),\psi(t)) \nonumber \\
    &\qquad + \int_s^t\intO2\nu_\Om(\phi)\abs{\D\bv}^2\dxtau + \int_s^t\intG 2\nu_\Ga(\psi)\abs{\Dg\bw}^2\dGtau + \int_s^t\intG \gamma(\phi,\psi)\abs{\bw}^2\dGtau \nonumber \\
    &\quad\leq E_\mathrm{kin}(\bv(s),\bw(s),\phi(s),\psi(s)) - \int_s^t\intO \phi\bv\cdot\Grad\mu\dxtau - \int_s^t\intG\psi\bw\cdot\Gradg\theta\dGtau
\end{align}
as well as
\begin{align}\label{EI:DISCR:FREE:AE}
    &E_\free(\phi(t),\psi(t)) \nonumber \\
    &\qquad + \int_s^t\intO m_\Om(\phi)\abs{\Grad\mu}^2\dxtau + \int_s^t\intG m_\Ga(\psi)\abs{\Gradg\theta}^2\dGtau \nonumber + \chi(L)\int_s^t\intG (\beta\theta-\mu)^2\dGtau \nonumber \\
    &\quad\leq E_\free(\phi(s),\psi(s)) + \int_s^t\intO \phi\bv\cdot\Grad\mu\dxtau + \int_s^t\intG \psi\bw\cdot\Gradg\theta\dGtau
\end{align}
for a.e. pair $0 < s < t < \infty$. Next, the temporal regularity of $(\bv,\bw,\phi,\psi)$ implies that the mappings
\begin{align*}
    t\mapsto E_\mathrm{kin}(\bv(t),\bw(t),\phi(t),\psi(t)), \qquad t\mapsto E_\free(\phi(t),\psi(t))
\end{align*}
are lower semicontinuous on $[0,\infty)$. Since all time-integral terms in \eqref{EI:DISCR:KIN:AE} and \eqref{EI:DISCR:FREE:AE} depend continuously on their upper integration limit, these inequalities extend from almost every terminal time $t$ to every $t\in[s,\infty)$. We conclude that these energy inequalities hold for all $t\in[s,\infty)$ and a.e.~$s\in(0,\infty)$. Finally, to see that $s = 0$ is admissible, we can simply set $s = 0$ first in \eqref{EI:DISCR:KIN} and \eqref{EI:DISCR:FREE}, and then pass to the limit $N\rightarrow\infty$ in an analogous way, by exploiting the convergence of the initial kinetic and free energy, respectively. We also refer to \cite[Theorem~3.8, 5.7.7]{Abels2025}, where a similar argument was carried out in detail. For the limit in \eqref{EI:DISCR:FREE}, see also \cite[Theorem~3.2]{Knopf2024}. This finishes the proof.
\end{proof}

\subsection{Asymptotic limit \texorpdfstring{$K\rightarrow 0$}{K->0}}
\label{Section:NSCH:K=0}
Lastly, in this section, we finish the proof of Theorem~\ref{NSCH:Theorem:ExistenceWeakSolutions} by proving the existence of a weak solution to \eqref{System:NSCH} in the case $K = 0$.

\begin{theorem}\label{NSCH:Theorem:Existence:K=0}
Suppose that the assumptions \ref{Assumption:NSCH:1}-\ref{Assumption:NSCH:Potentials} hold. Let $L\in[0,\infty]$, $(\bv_0,\bw_0)\in\mathbfcal{L}^2_\Div$, and let $(\phi_0,\psi_0)\in\mathcal{H}^^1$ satisfy $\phi_0 = \alpha\psi_0$ a.e.~on $\Ga$ as well as \eqref{NSCH:cond:init}.
In the case $L = 0$, we additionally assume that
\begin{align}
	\beta(\tilde\sigma_2 - \tilde\sigma_1) = -(\tilde\rho_2 - \tilde\rho_1).
\end{align}
For any $K\in(0,\infty)$, let $(\bv_K,\bw_K,\phi_K,\psi_K,\mu_K,\theta_K)$ denote a weak solution to system \eqref{System:NSCH} in the sense of Definition~\ref{NSCH:Definition:WeakSolution} with initial data $(\bv_0,\bw_0,\phi_0,\psi_0)$, whose existence is guaranteed by Theorem~\ref{NSCH:Theorem:ExistenceWeakSolutions:K>0}. Then, there exists a sextuplet $(\bv_\ast,\bw_\ast,\phi_\ast,\psi_\ast,\mu_\ast,\theta_\ast)$ such that, for any $T > 0$,
	\begin{alignat*}{2}
		(\bv_K,\bw_K) &\rightarrow (\bv_\ast,\bw_\ast) &&\qquad\text{weakly-star in~} L^\infty(0,\infty;\mathbfcal{L}^2), \\
		& &&\qquad\text{weakly in~} L^2(0,\infty;\mathbfcal{H}^1_{0,\Div}), \\
		& &&\qquad\text{strongly in~} L^2(0,T;\mathbfcal{L}^2), \\
		(\phi_K,\psi_K) &\rightarrow (\phi_\ast,\psi_\ast) &&\qquad\text{weakly-star in~} L^\infty(0,\infty;\mathcal{H}^1), \\
		& &&\qquad\text{strongly in~} C([0,T];\mathcal{H}^{1-s}) \text{~for all~}s\in(0,\tfrac12), \\
		(\mu_K,\theta_K) &\rightarrow (\mu_\ast,\theta_\ast) &&\qquad\text{weakly in~} L^2(0,T;\mathcal{H}^1),
	\end{alignat*}
	as $K\rightarrow 0$, along a non-relabeled subsequence, with
	\begin{align}\label{Est:Bc:K->0}
		\norm{\alpha\psi_K - \phi_K}_{L^\infty(0,\infty;L^2(\Ga))} \leq \sqrt{K},
	\end{align}
	and the sextuplet $(\bv_\ast,\bw_\ast,\phi_\ast,\psi_\ast,\mu_\ast,\theta_\ast)$ is a weak solution  to system \eqref{System:NSCH} in the sense of Definition~\ref{NSCH:Definition:WeakSolution} originating from $(\bv_0,\bw_0,\phi_0,\psi_0)$ for $K = 0$. Additionally, it holds that
    \begin{align*}
        (\phi_\ast,\psi_\ast)\in L^2_\uloc([0,\infty);\mathcal{W}^{2,p}), \qquad (F^\prime(\phi_\ast), G^\prime(\psi_\ast))\in L^2_\uloc([0,\infty);\mathcal{L}^p)
    \end{align*}
    for $p = 6$ if $d = 3$ and any $p\in[2,\infty)$ if $d = 2$.
\end{theorem}

\begin{remark}
\begin{enumerate}[label=\textnormal{(\alph*)},leftmargin=*]
    \item As the right-hand side in \eqref{Est:Bc:K->0} tends to zero as $K\rightarrow 0$, this explains why the Dirichlet type boundary condition $\phi_\ast = \alpha\psi_\ast$ a.e.~on $\Sigma$ appears in the limit model corresponding to $K = 0$. 
	\item The result of Theorem~\ref{NSCH:Theorem:Existence:K=0} remains valid if we replace the initial data $(\phi_0,\psi_0)\in\mathcal{H}^1$ with $\phi_0 = \alpha\psi_0$ a.e.~on $\Ga$ by a sequence $(\phi_{0K},\psi_{0K})\in\mathcal{H}^1$ satisfying
	\begin{align*}
		E_{\mathrm{free}}^K(\phi_{0K},\psi_{0K}) \rightarrow E_{\mathrm{free}}^0(\phi_0,\psi_0) \qquad\text{as~}K\rightarrow 0
	\end{align*}
	for some pair $(\phi_0,\psi_0)\in\mathcal{H}^1$, see also \cite[Theorem~2.3]{Knopf2020}. In this case, one can show that $\phi_0 = \alpha\psi_0$ a.e.~on $\Ga$, and that $(\phi_\ast,\psi_\ast)\vert_{t=0} = (\phi_0,\psi_0)$ a.e.~in $\Om\times\Ga$.
\end{enumerate}
\end{remark}

\begin{proof}
    We consider an arbitrary sequence $\{K_m\}_{m\in\N}\subset(0,\infty)$ such that $K_m\rightarrow0$ as $m\rightarrow\infty$. Let $\{(\bv_{K_m},\bw_{K_m},\phi_{K_m},\psi_{K_m},\mu_{K_m},\theta_{K_m})\}_{m\in\N}$ be a sequence of corresponding weak solutions  
	originating from the initial data $(\bv_0,\bw_0,\phi_0,\psi_0)$. In the following, we use the letter $C$ to denote generic positive constants independent of $K_m$ and $m$. 
    
    Let now $m\in\N$ be arbitrary. We first note that, since the initial datum $(\phi_0,\psi_0)$ satisfies the trace condition $\phi_0 = \alpha\psi_0$ a.e. on $\Ga$, the total energy of the initial data is bounded, i.e.,
    \begin{align*}
    	E_\mathrm{tot}^{K_m}(\bv_0,\bw_0,\phi_0,\psi_0) \leq C,
    \end{align*}
    as the corresponding term containing $K_m$ vanishes. Then, in view of the energy inequality \eqref{NSCH:WeakSolution:EnergyInequality}, we directly obtain that
   	\begin{align}\label{NSCH:K0:Est:Uniform}
   		&\norm{(\bv_{K_m},\bw_{K_m})}_{L^\infty(0,\infty;\mathbfcal{L}^2)} + \norm{(\bv_{K_m},\bw_{K_m})}_{L^2(0,\infty;\mathbfcal{H}^1)} \nonumber \\
   		&\quad + \norm{(\phi_{K_m},\psi_{K_m})}_{L^\infty(0,\infty;\mathcal{H}^1)} + \norm{(\Grad\mu_{K_m},\Gradg\theta_{K_m})}_{L^2(0,\infty;\mathbfcal{L}^2)} \nonumber \\
   		&\quad + \chi(L)^{\frac12}\norm{\beta\theta_{K_m} - \mu_{K_m}}_{L^2(0,\infty;L^2(\Ga))} \leq C,
   	\end{align}
   	as well as
   	\begin{align}\label{NSCH:K0:Est:Uniform:ap-p}
   		\norm{\alpha\psi_{K_m} - \phi_{K_m}}_{L^2(0,\infty;L^2(\Ga))} \leq C\sqrt{K_m}.
   	\end{align}
   	Then, the estimate
   	\begin{align}\label{NSCH:K0:Est:Uniform:High}
   		&\norm{(\phi_{K_m},\psi_{K_m})}_{L^2_{\mathrm{uloc}}([0,\infty);\mathcal{W}^{2,6})} + \norm{(F^\prime(\phi_{K_m}),G^\prime(\psi_{K_m}))}_{L^2_{\mathrm{uloc}}([0,\infty);\mathcal{L}^6)} \nonumber \\
   		&\quad + \norm{(\mu_{K_m},\theta_{K_m})}_{L^2_{\mathrm{uloc}}([0,\infty);\mathcal{L}^2)} \leq C
   	\end{align}
   	can be derived as in the proof of Theorem~\ref{NSCH:Theorem:ExistenceWeakSolutions:K>0}. Moreover, from a comparison argument, we can show, as in the proof of Theorem~\ref{NSCH:Theorem:ExistenceWeakSolutions:K>0}, that
   	\begin{align*}
   		\norm{(\delt\phi_{K_m},\delt\psi_{K_m})}_{L^2(0,\infty;(\mathcal{H}^1_{L,\beta})^\prime)} \leq C.
   	\end{align*}
   	Consequently, by the Banach--Alaoglu theorem and the Aubin--Lions--Simon lemma, there exists a sextuplet $(\bv_\ast,\bw_\ast,\phi_\ast,\psi_\ast,\mu_\ast,\theta_\ast)$ such that, for any $T > 0$,
    \begin{subequations}
	\begin{alignat}{2}
		(\bv_{K_m},\bw_{K_m}) &\rightarrow (\bv_\ast,\bw_\ast) &&\qquad\text{weakly-star in~} L^\infty(0,\infty;\mathbfcal{L}^2), \\
		& &&\qquad\text{weakly in~} L^2(0,\infty;\mathbfcal{H}^1_{0,\Div}), \\
		(\phi_{K_m},\psi_{K_m}) &\rightarrow (\phi_\ast,\psi_\ast) &&\qquad\text{weakly-star in~} L^\infty(0,\infty;\mathcal{H}^1), \\
		& &&\qquad\text{strongly in~} C([0,T];\mathcal{H}^{1-s}) \text{~for all~}s\in(0,\tfrac12), \label{NSCH:K0:Conv:pp:Strong} \\
		(\mu_{K_m},\theta_{K_m}) &\rightarrow (\mu_\ast,\theta_\ast) &&\qquad\text{weakly in~} L^2(0,T;\mathcal{H}^1), \label{NSCH:K0:Conv:MT:Weak}
	\end{alignat}
    \end{subequations}
	as $m\rightarrow\infty$ along a non-relabeled subsequence.
	Additionally, due to \eqref{NSCH:K0:Conv:pp:Strong} and the trace theorem, we infer
	\begin{align*}
		\alpha\psi_{K_m} - \phi_{K_m} \rightarrow \alpha\psi_\ast - \phi_\ast \qquad\text{strongly in~} C([0,T];L^2(\Ga))
	\end{align*}
	as $m\rightarrow\infty$ for any $T > 0$, which, in combination with \eqref{NSCH:K0:Est:Uniform:ap-p}, readily entails $\phi_\ast = \alpha\psi_\ast$ a.e.~on $\Sigma$. Then, as in the proof of Theorem~\ref{NSCH:Theorem:ExistenceWeakSolutions:K>0}, we can show that for any $T > 0$
    \begin{align*}
        (\bv_{K_m},\bw_{K_m})\rightarrow (\bv_\ast,\bw_\ast) \qquad\text{strongly in~} L^2(0,T;\mathbfcal{L}^2)
    \end{align*}
    as $m\rightarrow\infty$ along a non-relabeled subsequence, as well as
    \begin{align*}
        (\bv_\ast,\bw_\ast)\in BC_w([0,\infty);\mathbfcal{L}^2_\Div).
    \end{align*}
    Then, from these convergences, it is straightforward to pass to the limit $m\rightarrow\infty$ in the weak formulations \eqref{NSCH:WF:VW} and \eqref{NSCH:WF:PP}. To pass to the limit in \eqref{NSCH:WeakSolution:MuTheta}, we make use of the Mosco convergence of the convex part of the free energy established in Proposition~\ref{Proposition:MoscoConvergence}. To this end, we first notice that
    \begin{align*}
        (\mu_{K_m}+c_F\phi_{K_m},\theta_{K_m}+c_G\psi_{K_m}) = \partial \widetilde E_\free^{K_m}(\phi_{K_m},\psi_{K_m})
    \end{align*}
    for all $m\in\N$. Then, recalling the convergences \eqref{NSCH:K0:Conv:pp:Strong}-\eqref{NSCH:K0:Conv:MT:Weak}, we have
    \begin{align*}
        (\phi_{K_m},\psi_{K_m})\rightarrow(\phi_\ast,\psi_\ast) \qquad\text{strongly in~} L^2(0,T;\mathcal{L}^2),
    \end{align*}
    as well as
    \begin{align*}
        (\mu_{K_m}+c_F\phi_{K_m},\theta_{K_m}+c_G\psi_{K_m}) \rightarrow (\mu_\ast+c_F\phi_\ast,\theta_\ast+c_G\psi_\ast) \qquad\text{weakly in~}L^2(0,T;\mathcal{L}^2)
    \end{align*}
    as $m\rightarrow\infty$. We are thus in the setting of Proposition~\ref{Prelim:Proposition:GraphConvergence} and deduce that $(\phi_\ast,\psi_\ast)\in D(\partial\widetilde E_\free^0)$ a.e.~on $[0,\infty)$ and
    \begin{align*}
        (\mu_\ast + c_F\phi_\ast,\theta_\ast + c_G\psi_\ast) = \partial\widetilde E_\free^0(\phi_\ast,\psi_\ast).
    \end{align*}
    In particular, it holds
    \begin{alignat*}{2}
        \mu_\ast &= -\Lap\phi_\ast + F^\prime(\phi_\ast) &&\qquad\text{a.e.~in~}Q, \\
        \theta_\ast &= -\Lapg\psi_\ast + G^\prime(\psi_\ast) + \alpha\deln\phi_\ast &&\qquad\text{a.e.~on~}\Sigma, \\
        \phi_\ast &= \alpha\psi_\ast &&\qquad\text{a.e.~on~}\Sigma.
    \end{alignat*}
    With this, we immediately obtain
    \begin{align*}
        (\phi_\ast,\psi_\ast)\in L^2_\uloc([0,\infty);\mathcal{W}^{2,p}), \qquad (F^\prime(\phi_\ast),G^\prime(\psi_\ast))\in L^2_\uloc([0,\infty);\mathcal{L}^p)
    \end{align*}
    for $p = 6$ if $d = 3$ and any $p\in[2,\infty)$ if $d = 2$. Lastly, to go over the limit in the separate energy inequalities \eqref{EnergyInequality:Kinetic} and \eqref{EnergyInequality:Free}, we can simply argue as in the proof of Theorem~\ref{NSCH:Theorem:ExistenceWeakSolutions:K>0}. In this way, we also obtain the energy inequality \eqref{NSCH:WeakSolution:EnergyInequality} for the total energy. This finishes the proof.
\end{proof}

\section{The high friction limit}
\label{Seciton:HighFrictionLimit}
In this last section, we investigate the high friction limit by sending $\gamma\rightarrow\infty$. In the same procedure, we additionally let $K\rightarrow\infty$ and $L\rightarrow\infty$. 

We then have the following theorem.

\begin{theorem}\label{Theorem:HighFriction}
Suppose the assumptions \ref{Assumption:NSCH:1}-\ref{Assumption:NSCH:Potentials} hold. Let $K,L\in(0,\infty)$, $\gamma\in(0,\infty)$, and let initial data $\bv_0\in\mathbf{L}^2_\Div(\Om)$ and $(\phi_0,\psi_0)\in\mathcal{H}^1$ be given such that
\begin{align*}
    \abs{\phi_0}\leq1\quad\text{a.e.~in }\Om, \qquad \abs{\psi_0}\leq1\quad\text{a.e.~on }\Ga,
\end{align*}
and
\begin{align*}
    \meano{\phi_0},\meang{\psi_0}\in(-1,1).
\end{align*}
Consider a weak solution $(\bv_{\gamma,K,L},\bw_{\gamma,K,L},\phi_{\gamma,K,L},\psi_{\gamma,K,L},\mu_{\gamma,K,L},\theta_{\gamma,K,L})$ to system \eqref{System:NSCH} originating from $(\bv_0,\mathbf 0,\phi_0,\psi_0)$, which exists according to Theorem~\ref{NSCH:Theorem:ExistenceWeakSolutions}.
Then, there exist functions $\bv, \phi, \mu, \psi$ and $\theta$ such that, for any $T > 0$, it holds
\begin{subequations}
\begin{alignat}{2}
    (\bv_{\gamma,K,L},\bw_{\gamma,K,L}) &\rightarrow (\bv,\mathbf 0) &&\qquad\text{weakly-star in~}L^\infty(0,\infty;\mathbfcal{L}^2), \\
    & &&\qquad\text{weakly in~}L^2(0,\infty;\mathbfcal{H}^1_{0,\Div}), \\
    & &&\qquad\text{strongly in~} L^2(0,T;\mathbfcal{L}^2), \\
    (\phi_{\gamma,K,L},\psi_{\gamma,K,L}) &\rightarrow (\phi,\psi) &&\qquad\text{weakly-star in~}L^\infty(0,\infty;\mathcal{H}^1), \\
    & &&\qquad\text{strongly in~}C([0,T];\mathcal{H}^{1-s})\ \text{for all~}s\in(0,\tfrac12), \\
    (\mu_{\gamma,K,L},\theta_{\gamma,K,L}) &\rightarrow(\mu,\theta) &&\qquad\text{weakly in~}L^2(0,T;\mathcal{H}^1)
\end{alignat}
\end{subequations}
as $\gamma,K,L\rightarrow\infty$ along a non-relabeled subsequence. Moreover, $(\bv,\phi,\mu)$ is a weak solution to the AGG-model \eqref{System:AGG} corresponding to the initial data $(\bv_0,\phi_0)$ in the sense of \cite{Abels2013}, and $(\psi,\theta)$ solves a surface Cahn--Hilliard equation with initial data $\psi_0$.
    
\end{theorem}

\begin{proof}
Consider $(K_m)_{m\in\N}$, $(L_m)_{m\in\N}\subset(0,\infty)$ and $(\gamma_m)_{m\in\N}\subset(0,\infty)$ with $K_m, L_m, \gamma_m \rightarrow\infty$ as $m\rightarrow\infty$. Without loss of generality we assume that $K_m, L_m, \gamma_m \geq 1$. In the following, we denote with the letter $C$ positive constants that are independent of $K_m, L_m, \gamma_m$ and $m$.

For any $m\in\N$ let $(\bv_m,\bw_m,\phi_m,\psi_m,\mu_m,\theta_m)$ be a corresponding weak solution originating from the initial data $(\bv_{0,m},\bw_{0,m}) = (\bv_0,\mathbf 0)$ and $(\phi_{0,m},\psi_{0,m}) = (\phi_0,\psi_0)$ which exists according to Theorem~\ref{NSCH:Theorem:ExistenceWeakSolutions}. The energy inequality \eqref{NSCH:WeakSolution:EnergyInequality} then reads as
\begin{align}\label{HighFriction:Energy}
    &E_\tot^{K_m}(\bv_m(t),\bw_m(t),\phi_m(t),\psi_m(t)) \nonumber \\
    &\qquad + \int_0^t\intO 2\nu_\Om(\phi_m)\abs{\D\bv_m}^2\dxs + \int_0^t\intG 2\nu_\Ga(\psi_m)\abs{\Dg\bw_m}^2\dGs + \int_0^t\intG \gamma_m\abs{\bw_m}^2\dGs \nonumber \\
    &\qquad + \int_0^t\intO m_\Om(\phi_m)\abs{\Grad\mu_m}^2\dxs + \int_0^t\intG m_\Ga(\psi_m)\abs{\Gradg\theta_m}^2\dGs \nonumber \\
    &\qquad + \frac{1}{L_m}\int_0^t\intG (\beta\theta_m-\mu_m)^2\dGs \nonumber \\
    &\quad \leq E_\tot^{K_m}(\bv_0,\mathbf 0,\phi_0,\psi_0)
\end{align}
for all $t\geq 0$. In view of the regularity assumptions on the initial data, and as $K_m\geq 1$, we find a constant $C > 0$ such that
\begin{align}\label{HighFriction:ICEnergy}
    E_\tot^{K_m}(\bv_0,\mathbf 0,\phi_0,\psi_0) \leq C\big(1 + K_m^{-1}\big) \leq C
\end{align}
for all $m\in\N$. Then, since $\abs{\phi_m}\leq 1$ a.e. in $Q$ and $\abs{\psi_m} \leq 1$ a.e. on $\Sigma$, we employ the bulk-surface Korn inequality and readily obtain from \eqref{HighFriction:Energy} and \eqref{HighFriction:ICEnergy} that
\begin{align}\label{Asymp:EnergyInequality}
    &\norm{(\bv_m,\bw_m)}_{L^\infty(0,\infty;\mathbfcal{L}^2)} + \norm{(\bv_m,\bw_m)}_{L^2(0,\infty;\mathbfcal{H}^1)} \nonumber \\
    &\quad + \norm{(\phi_m,\psi_m)}_{L^\infty(0,\infty;\mathcal{H}^1)} + \norm{(\Grad\mu_m,\Gradg\theta_m)}_{L^2(0,\infty;\mathbfcal{L}^2)} \leq C.
\end{align}
Moreover, we have
\begin{align}
    \norm{\bw_m}_{L^2(0,\infty;\mathbf L^2(\Ga))} &\leq \frac{C}{\sqrt{\gamma_m}} \label{Asymp:w:gamma}, \\
    \frac{1}{L_m}\norm{\beta\theta_m-\mu_m}_{L^2(0,\infty;L^2(\Ga))} &\leq \frac{C}{\sqrt{L_m}}. \label{Asymp:mt:boundary}
\end{align}
Then, the estimate
\begin{align}\label{NSCH:GKL0:Est:Uniform:High}
    \begin{split}
        \norm{(\mu_m,\theta_m)}_{L^2_{\mathrm{uloc}}([0,\infty);\mathcal{L}^2)} \leq C
    \end{split}
\end{align}
can be derived as in the proof of Theorem~\ref{NSCH:Theorem:ExistenceWeakSolutions:K>0}. Moreover, from a comparison argument, we can show, as in the proof of Theorem~\ref{NSCH:Theorem:ExistenceWeakSolutions:K>0}, that
\begin{align}\label{Asymp:pp:delt}
    \norm{(\delt\phi_m,\delt\psi_m)}_{L^2(0,\infty;(\mathcal{H}^1_{L,\beta})^\prime)} \leq C.
\end{align}
In view of the uniform estimates \eqref{Asymp:EnergyInequality}, \eqref{NSCH:GKL0:Est:Uniform:High} and \eqref{Asymp:pp:delt}, we infer from the Banach--Alaoglu theorem and the Aubin--Lions--Simon lemma the existence of functions $(\bv,\bw,\phi,\psi,\mu,\theta)$, such that for any $T > 0$,
\begin{subequations}
    \begin{alignat}{2}
        (\bv_m,\bw_m) &\rightarrow (\bv,\bw) &&\qquad\text{weakly-star in~}L^\infty(0,\infty;\mathbfcal{L}^2), \label{HF:Conv:VW} \\
        & &&\qquad\text{weakly in~}L^2(0,\infty;\mathbfcal{H}^1_{0,\Div}), \label{HF:Conv:vw:L2H1}\\
        (\phi_m,\psi_m) &\rightarrow (\phi,\psi) &&\qquad\text{weakly-star in~}L^\infty(0,\infty;\mathcal{H}^1), \\
        & &&\qquad\text{strongly in~} C([0,T];\mathcal{H}^{1-s}) \ \text{for all~}s\in(0,\tfrac12), \label{HF:Conv:pp:Strong} \\
        (\mu_m,\theta_m) &\rightarrow(\mu,\theta) &&\qquad\text{weakly in~}L^2(0,T;\mathcal{H}^1) \label{HF:CONV:MT:WEAK}
    \end{alignat}
\end{subequations}
along a non-relabeled subsequence $m\rightarrow\infty$. Then, similarly to the proof of Theorem~\ref{NSCH:Theorem:ExistenceWeakSolutions:K>0}, we can show that for any $T > 0$
\begin{align}\label{HF:Conv:vw:Strong}
    (\bv_m,\bw_m) \rightarrow (\bv,\bw) \qquad\text{strongly in~}L^2(0,T;\mathbfcal{L}^2)
\end{align}
along a non-relabeled subsequence $m\rightarrow\infty$, and
\begin{align*}
    (\bv,\bw)\in BC_w([0,\infty);\mathbfcal{L}^2_\Div).
\end{align*}
Thus, we immediately infer from \eqref{Asymp:w:gamma} that $\bw = \mathbf 0$ a.e.~on $\Sigma$ and therefore $\bv\vert_\Ga = \mathbf 0$ a.e. on $\Sigma$. Next, we go over to the limit in the weak formulation. To this end let $\wv\in\mathbf C_{0,\sigma}^\infty(\Om\times(0,\infty))$ and take $(\wv,\mathbf 0)$ as a test function in \eqref{NSCH:WF:VW}. Then the weak formulation reads as
\begin{align*}
    &-\int_0^\infty\intO\rho(\phi_m)\bv_m\cdot\delt\wv\dxt - \int_0^\infty\intO \Div(\rho(\phi_m)\bv_m\otimes\bv_m)\cdot\wv\dxt \\
    &\qquad + \int_0^\infty\intO (\bv_m\otimes\J_m):\Grad\wv\dxt + \int_0^\infty\intO 2\nu_\Om(\phi_m)\D\bv_m:\D\wv\dxt \\
    &\quad = \int_0^\infty\intO \mu_m\Grad\phi_m\cdot\wv\dxt
\end{align*}
where $\J_m = -\frac{\tilde\rho_1-\tilde\rho_2}{2}m_\Om(\phi_m)\Grad\mu_m$. Letting $m\rightarrow\infty$, we readily infer from \eqref{HF:Conv:VW}-\eqref{HF:CONV:MT:WEAK} and \eqref{HF:Conv:vw:Strong} that
\begin{align*}
    &-\int_0^\infty\intO\rho(\phi)\bv\cdot\delt\wv\dxt - \int_0^\infty\intO\Div(\rho(\phi)\bv\otimes\bv)\cdot\wv\dxt \\
    &\qquad + \int_0^\infty\intO (\bv\otimes\J):\Grad\wv\dxt + \int_0^\infty\intO 2\nu_\Om(\phi)\D\bv:\D\wv\dxt \\
    &\quad = \int_0^\infty\intO \mu\Grad\phi\cdot\wv\dxt
\end{align*}
for all $\wv\in C_{0,\sigma}^\infty(\Om\times(0,\infty))$. Next, we have
\begin{align}\label{HiFr:WF:PP:m}
    \begin{split}
    &-\int_0^\infty\intO \phi_m\,\delt\zeta\dxt -\int_0^\infty\intG \psi_m\,\delt\zeta_\Ga\dGt \\
    &\qquad - \int_0^\infty\intO\phi_m\bv_m\cdot\Grad\zeta\dxt - \int_0^\infty\intG\psi_m\bw_m\cdot\Gradg\zeta_\Ga\dGt \\
    &\quad = -\int_0^\infty\intO m_\Om(\phi_m)\Grad\mu_m\cdot\Grad\zeta\dxt - \int_0^\infty\intG m_\Ga(\psi_m)\Gradg\theta_m\cdot\Gradg\zeta_\Ga\dGt \\
    &\qquad - \frac{1}{L_m}\int_0^\infty\intG (\beta\theta_m-\mu_m)(\beta\zeta_\Ga-\zeta)\dGt
    \end{split}
\end{align}
for all $(\zeta,\zeta_\Ga)\in C_c^\infty(0,\infty;\mathcal{H}^1)$, where the chemical potentials are characterized by
\begin{align}\label{HF:MT:M:SUBDIFF}
    (\mu_m+c_F\phi_m,\theta_m+c_G\psi_m) = \partial\widetilde E_\free^{K_m}(\phi_m,\psi_m).
\end{align}
Arguing again as in the proof of Theorem~\ref{NSCH:Theorem:ExistenceWeakSolutions:K>0}, we can let $m\rightarrow\infty$ in \eqref{HiFr:WF:PP:m} and readily deduce that
\begin{align*}
    &-\int_0^\infty\intO \phi\,\delt\zeta\dxt -\int_0^\infty\intG \psi\,\delt\zeta_\Ga\dGt - \int_0^\infty\intO\phi\bv\cdot\Grad\zeta\dxt \\
    &\quad = -\int_0^\infty\intO m_\Om(\phi)\Grad\mu\cdot\Grad\zeta\dxt - \int_0^\infty\intG m_\Ga(\psi)\Gradg\theta\cdot\Gradg\zeta_\Ga\dGt 
\end{align*}
for all $(\zeta,\zeta_\Ga)\in C_c^\infty(0,\infty;\mathcal{H}^1)$. Here, we have additionally used that the limiting surface velocity field satisfies $\bw = \mathbf 0$ a.e.~on $\Sigma$ and the convergence \eqref{Asymp:mt:boundary}. In particular, choosing $(\zeta,0)$ and $(0,\zeta_\Ga)$ for $\zeta\in C_c^\infty(0,\infty;H^1(\Om))$ and $\zeta_\Ga\in C_c^\infty(0,\infty;H^1(\Ga))$ as test functions, respectively, we get
\begin{align*}
    &-\int_0^\infty\intO \phi\,\delt\zeta\dxt  - \int_0^\infty\intO\phi\bv\cdot\Grad\zeta\dxt = -\int_0^\infty\intO m_\Om(\phi)\Grad\mu\cdot\Grad\zeta\dxt
\end{align*}
as well as
\begin{align*}
    -\int_0^\infty\intG \psi\,\delt\zeta_\Ga\dGt = - \int_0^\infty\intG m_\Ga(\psi)\Gradg\theta\cdot\Gradg\zeta_\Ga\dGt.
\end{align*}
For the chemical potentials we can argue as in the proof of Theorem~\ref{NSCH:Theorem:Existence:K=0} and obtain $(\phi,\psi)\in D(\partial\widetilde E_\free^\infty)$ with
\begin{align*}
    (\mu + c_F\phi,\theta + c_G\psi) = \partial \widetilde E_\free^\infty(\phi,\psi).
\end{align*}
In particular, it holds
\begin{alignat*}{2}
    \mu &= -\Lap\phi + F^\prime(\phi) &&\qquad\text{a.e.~in~}Q, \\
    \deln\phi &= 0 &&\qquad\text{a.e.~on~}\Sigma,
\end{alignat*}
and
\begin{align*}
    \theta = -\Lapg\psi + G^\prime(\psi) \qquad\text{a.e.~on~}\Sigma.
\end{align*} 
This already proves that $(\psi,\theta)$ is a weak solution to the surface Cahn--Hilliard equation with initial data $\psi_0$.
To show that $(\bv,\phi,\mu)$ is a weak solution to the AGG-model \eqref{System:AGG}, it remains to show the validity of the energy inequality. To do so, we rely on the fact that the separate energy inequalities \eqref{EI:DISCR:KIN} and \eqref{EI:DISCR:FREE} for the kinetic and free energy hold. First, from \eqref{EI:DISCR:KIN}, we know that
\begin{align*}%\label{KS:EnergyInequality:Kinetic}
    &\intO\frac12\rho(\phi_m(t))\abs{\bv_m(t)}^2\dx + \intG\frac12\sigma(\psi_m(t))\abs{\bw_m(t)}^2\dG \nonumber \\
    &\qquad + \int_s^t\intO 2\nu_\Om(\phi_m)\abs{\D\bv_m}^2\dxtau + \int_s^t\intG 2\nu_\Ga(\psi_m)\abs{\Dg\bw_m}^2\dGtau \nonumber \\
    &\qquad + \int_s^t\intG \gamma(\phi_m,\psi_m)\abs{\bw_m}^2\dGtau \nonumber \\
    &\quad\leq \intO\frac12\rho(\phi_m(s))\abs{\bv_m(s)}^2\dx + \intG\frac12\sigma(\psi_m(s))\abs{\bw_m(s))}^2\dG - \int_s^t\intO \phi_m\bv_m\cdot\Grad\mu_m\dxtau \nonumber \\
    &\qquad - \int_s^t\intG \psi_m\bw_m\cdot\Gradg\theta_m\dGtau
\end{align*}
for all $t\in[s,\infty)$ and a.e. $s\in[0,\infty)$ including $s = 0$. To pass to the limit, we proceed as in the proof of Theorem~\ref{NSCH:Theorem:ExistenceWeakSolutions:K>0}. The only differences concern the surface contributions. Indeed, after passing to a subsequence,
\begin{align*}
    \bw_m(s) \rightarrow \mathbf 0 \qquad\text{strongly in~}\mathbf{L}^2(\Ga)
\end{align*}
for a.e. $s\in(0,\infty)$, and hence the surface kinetic energy converges to zero for a.e. $s\in(0,\infty)$. Moreover, the surface convective coupling term converges to zero, whereas the non-negative surface dissipation terms may be discarded when taking the the limes inferior. Consequently,
\begin{align*}
    &\intO\frac12\rho(\phi(t))\abs{\bv(t)}^2\dx + \int_s^t\intO 2\nu_\Om(\phi)\abs{\D\bv}^2\dxtau \\
    &\quad\leq \intO\frac12\rho(\phi(s))\abs{\bv(s)}^2\dx - \int_s^t\intO \phi\bv\cdot\Grad\mu\dxtau 
\end{align*}
for a.e. $0 < s < t < \infty$. As in the proof of Theorem~\ref{NSCH:Theorem:ExistenceWeakSolutions:K>0}, the lower semicontinuity in time of the limiting energies and the continuity of the time integral terms extend this inequality from a.e. terminal time $t$ to every $t\in[s,\infty)$. The case $s = 0$ follows separately as $(\bv_m(0),\bw_m(0)) = (\bv_0,\mathbf 0)$ a.e.~in $\Om\times\Ga$ and $\phi_m(0) = \phi_0$ a.e.~in $\Om$.
Concerning the free energy, an application of the chain rule in $L^2(0,T;H^1(\Om))\cap H^1(0,T;H^1(\Om)^\prime)$ shows that
\begin{align*}
    E_\Om(\phi(t)) + \int_s^t\intO m_\Om(\phi)\abs{\Grad\mu}^2\dxtau = E_\Om(\phi(s)) + \int_s^t\intO \phi\bv\cdot\Grad\mu\dxtau
\end{align*}
for all $t\in[s,\infty)$ and a.e.~$s\in[0,\infty)$ including $s = 0$. In particular, we obtain the corresponding energy inequality for the total energy, which entails that $(\bv,\phi,\mu)$ is a weak solution to the Abels--Garcke--Grün model \eqref{System:AGG} in the sense of \cite[Definition~3.3]{Abels2013}.
\end{proof}

\section*{Appendix}

\renewcommand\thesection{A}
\setcounter{theorem}{0}
\setcounter{equation}{0}

In this appendix, we collect some useful results that were used throughout this manuscript. We start with the following lemma, whose proof can be found in \cite[Lemma~4.3]{Abels2009}.

\begin{lemma}\label{NSCH:Prelim:Lemma:Energy}
    Let $E:[0,T)\rightarrow[0,\infty)$, $0 < T \leq \infty$, be a lower semicontinuous function and let $D:(0,T)\rightarrow[0,\infty)$ be an integrable function. Then
    \begin{align*}
        E(0)\pi(0) + \int_0^T E(t)\pi^\prime(t)\dt \geq \int_0^T D(t)\pi(t)\dt
    \end{align*}
    holds for all $\pi\in W^{1,1}(0,T)$ with $\pi \geq 0$ and $\pi(T) = 0$ if and only if
    \begin{align*}
        E(t) + \int_s^t D(\tau)\dtau \leq E(s)
    \end{align*}
    holds for all $t\in[s,T)$ and almost all $s\in[0,T)$ including $s = 0$.
\end{lemma}

We also report the following well-posedness and regularity result for an elliptic system with bulk-surface coupling (see \cite[Lemma 3.1]{Gal2024}).

\begin{lemma}\label{NSCH:Proposition:Elliptic:MuTheta}
    Suppose that \ref{Assumption:NSCH:Coefficients} holds, and let $(\phi,\psi)\in\mathcal{H}^2$ with $\abs{\phi}\leq 1$ a.e.~in $\Om$ and $\abs{\psi} \leq 1$ a.e.~on $\Ga$, $(f,f_\Ga)\in\mathcal{L}^2$ and $L\in[0,\infty]$ be given. Then there exists a unique solution $(\mu,\theta)\in\mathcal{H}^2\cap\mathcal{H}^1_{L,\beta}$
    of
    \begin{subequations}\label{NSCH:Subsystem}
    \begin{alignat}{2}
        -\Div(m_\Om(\phi)\Grad\mu) + \intO\mu\dx &= f &&\qquad\text{in~}\Om, \\
        -\Divg(m_\Ga(\psi)\Gradg\theta) + \beta m_\Om(\phi)\deln\mu + \intG\theta\dG &= f_\Ga &&\qquad\text{on~}\Ga, \\
        Lm_\Om(\phi)\deln\mu &= \beta\theta - \mu 
         &&\qquad\text{on~}\Ga,
    \end{alignat}
    \end{subequations}
    in the sense that it satisfies the variational formulation
    \begin{align*}
        &\intO m_\Om(\phi)\Grad\mu\cdot\Grad\zeta\dx + \intG m_\Ga(\psi)\Gradg\theta\cdot\Gradg\zeta_\Ga\dG \\
        &\qquad + \chi(L)\intG(\beta\theta-\mu)(\beta\zeta_\Ga-\zeta)\dG + \intO\mu\dx\intO\zeta\dx + \intG\theta\dG\intG\zeta_\Ga\dG \\
        &\quad = \intO f\,\zeta\dx + \intG f_\Ga\,\zeta_\Ga\dG
    \end{align*}
     for all $(\zeta,\zeta_\Ga)\in\mathcal{H}^1_{L,\beta}$. Furthermore, there exists $C > 0$ such that
    \begin{align*}
        \norm{(\mu,\theta)}_{\mathcal{H}^2} \leq C\norm{(f,f_\Ga)}_{\mathcal{L}^2},
    \end{align*}
    where the constant $C$ may depend on $\norm{(\phi,\psi)}_{\mathcal{H}^2}, m_\ast, m^\ast, L, \Om$ and $\Ga$.
\end{lemma}

The following lemma provides a smooth sequence of functions approximating an initial datum $(\phi_0,\psi_0)\in\mathcal{H}^1$.

\begin{lemma}\label{NSCH:Lemma:ApproxIC}
    Let $(\phi_0,\psi_0)\in\mathcal{H}^1$ with $\abs{\phi_0} \leq 1$ a.e.~in $\Om$ and $\abs{\psi_0} \leq 1$ a.e.~on $\Ga$ be given. Then there exists a sequence $\{(\phi_0^N,\psi_0^N)\}_{N\in\N}\subset\mathcal{H}^2$ with $\abs{\phi_0^N} \leq 1$ a.e.~in $\Om$, $\abs{\psi_0^N} \leq 1$ a.e.~on $\Ga$, and $(\phi_0^N,\psi_0^N)\rightarrow(\phi_0,\psi_0)$ strongly in $\mathcal{H}^1$ as $N\rightarrow\infty$. 
\end{lemma}

\begin{proof}
    Let $u$ be a solution to
    \begin{alignat*}{2}
        \delt u - \Lap u &= 0 &&\qquad\text{in~}Q_T, \\
        \deln u &= 0 &&\qquad\text{on~}\Sigma_T, \\
        u\vert_{t=0} &= \phi_0 &&\qquad\text{in~}\Om,
    \end{alignat*}
    and choose $v$ as a solution to 
    \begin{alignat*}{2}
        \delt v - \Lapg v &= 0 &&\qquad\text{on~}\Sigma_T, \\
        v\vert_{t=0} &= \psi_0 &&\qquad\text{on~}\Ga.
    \end{alignat*}
    Then it holds
    \begin{align*}
        (u,v)\in C((0,T];\mathcal{H}^2)\cap C([0,T];\mathcal{H}^1).
    \end{align*}
    Now, setting $(\phi_0^N,\psi_0^N) \coloneqq (u,v)\vert_{t=\frac1N}$ for any $N\in\N$ ensures $(\phi_0^N,\psi_0^N)\in\mathcal{H}^2$, $\abs{\phi_0^N} \leq 1$ a.e.~in $\Om$, $\abs{\psi_0^N} \leq 1$ a.e.~on $\Ga$ as well as $(\phi_0^N,\psi_0^N)\rightarrow (\phi_0,\psi_0)$ strongly in $\mathcal{H}^1$ as $N\rightarrow\infty$.
\end{proof}

\section*{Acknowledgement}
\noindent

The author was supported by the Deutsche Forschungsgemeinschaft (DFG,
German Research Foundation): on the one hand by the DFG-project 524694286, and on the other
hand by the RTG 2339 “Interfaces, Complex Structures, and Singular Limits”. Their support is
gratefully acknowledged.

\section*{Conflict of Interests and Data Availability Statement}

There is no conflict of interest.

There is no associated data with the manuscript.

\bibliographystyle{abbrv}
\bibliography{DissertationBibliography.bib}
\end{document}